[2020/11/05]
\documentclass[a4paper,draft]{amsart}
\usepackage{amsmath}
\usepackage{amsthm}
\usepackage{amscd}
\usepackage{amsfonts}
\usepackage{amssymb}
\usepackage{mathrsfs}
\usepackage{xcolor}
\usepackage{cancel}
\usepackage[T1]{fontenc}
\usepackage{lipsum}

\allowdisplaybreaks[1]

\chardef\bslash=`\\ 

\newtheorem{thm}{Theorem}[section]
\newtheorem{cor}[thm]{Corollary}
\newtheorem{lem}[thm]{Lemma}

\newtheorem{prop}[thm]{Proposition}

\newtheorem{defn}[thm]{Definition}

\theoremstyle{remark}

\newcommand{\B}{\mathcal{B}}
\newcommand{\X}{\mathcal{X}}

\newcommand{\C}{\mathcal{C}}
\newcommand{\D}{\mathcal{D}}

\newcommand{\LS}{\mathcal{L}}
\newcommand{\US}{\mathfrak{A}}
\def\f#1{\mathfrak{#1}}
\newcommand{\fS}{\f{S}}

\def\MN(#1){ M_{#1}(\mathbb{N})}
\def\MNR(#1,#2){ M_{#1}(\mathbb{N})_{#2}}
\def\MNS(#1){ M_{#1}(\mathbb{N})^{\pm}}
\newcommand{\ZZ}{\mathbb{Z}}
\newcommand{\CC}{\mathbb{C}}
\newcommand{\ZG}{{\ZZ}_2}
\newcommand{\NN}{\mathbb{N}}
\newcommand{\QQ}{\mathbb{Q}}
\def\MZ(#1){ M_{#1}(\ZG)}
\def\NZ(#1){ (\NN|\ZG)^{#1}}
\def\MNZ(#1,#2){ M_{#1}(\NN | \ZG)_{#2}}
\def\MNZN(#1){ M_{#1}(\NN | \ZG)}
\def\CMNZ(#1,#2,#3){\Lambda(#1,#2|#3)}
\def\CMN(#1,#2){\Lambda(#1,#2)}
\def\MNZNS(#1){ \MNZN(#1)^{\pm}}
\def\SE#1{#1^{0}}
\def\SO#1{#1^{1}}
\def\SEE#1{#1^{0}}
\def\SOE#1{#1^{1}}
\def\SUP#1{\SE{#1}|\SO{#1}}

\def\bs#1{\boldsymbol{#1}}
\def\Qqs(#1,#2){ \mathcal{Q}(#1,#2) }
\newcommand{\SerA}{\f{H}^c_r}
\newcommand{\SerZ}{\f{H}^c_{r, \ZZ}}
\newcommand{\HCR}{\mathcal{H}^c_{r,R}}

\newcommand{\QqscR}{\mathcal{Q}_q(n,r; R)}

\newcommand{\QqscZ}{\mathcal{Q}_1(n,r; \ZZ)}
\newcommand{\Qnr}{\mathcal{Q}(n,r)}
\newcommand{\QnrZ}{\mathcal{Q}(n,r)_{\ZZ}}
\newcommand{\Qqsn}{\mathcal{Q}(n)}
\newcommand{\USn}{\US[n]}
\def\ol#1{\overline{#1}}
\newcommand{\ep}{\epsilon}

\def\qn{\mathfrak{q}_n}
\def\Uqn{U(\qn)}
\def\VZ{V_{\ZZ}}
\def\UqnZ{U(\qn)_{\ZZ}}
\def\Uqqn{U_q(\qn)}
\newcommand{\dij}{\delta_{i,j}}
\def\ABJR(#1,#2,#3,#4){({#1}|{#2})[\bs{#3}, #4]}
\def\ABJRS(#1,#2,#3,#4){({#1}|{#2})[#3, #4]}
\def\ABJS(#1,#2,#3){({#1}|{#2})[#3]}
\def\AJRS(#1,#2,#3){{#1}[{#2}, #3]}
\def\AJS(#1,#2){{#1}[#2]}
\def\AJRS(#1,#2,#3){{#1}[{#2}, #3]}

\def\ABSUM#1{\widehat{#1}}

\def\MNZADD(#1, #2){\SE{#1} + {#2}|\SO{#1}}
\newcommand{\tspan}{\mbox{span}}

\newcommand{\Tr}{\mathscr{D}}
\newcommand{\fsP}{\mathscr{P}}

\newcommand{\fsG}{\mathscr{G}}

\newcommand{\End}{\mathrm{End}}
\newcommand{\co}{\mbox{co}}
\newcommand{\ro}{\mbox{ro}}

\newcommand{\diag}{\mbox{diag}}
\newcommand{\where}{\ \bs{\big|} \ }
\def\snorm#1{|#1|}

\def\SAIJ(#1,#2){ S_{({#1}, {#2})} }
\def\SDIJ(#1,#2){ S_{({#1}, {#2})} }

\begin{document}
\title{The integral Schur-Weyl-Sergeev duality}
\author{Haixia Gu, Zhenhua Li$^\dagger$,  Yanan Lin}
\address{Haixia Gu, School of Science, Huzhou University, Huzhou 313000, China}
\email{ghx@zjhu.edu.cn}
\address{Zhenhua Li, School of Mathematical Sciences, Xiamen University, Xiamen 361005, China}
\email{zhenhua@stu.xmu.edu.cn}
\address{Yanan Lin, School of Mathematical Sciences, Xiamen University, Xiamen 361005, China}
\email{ynlin@xmu.edu.cn}
\date{}
\thanks{$^\dagger$Corresponding author.}

\begin{abstract}
	Degenerating the quantum queer Schur superalgebra ${\mathcal{Q}_q(n,r; R)}$ to the case $q=1$, the queer Schur superalgebra ${\mathcal{Q}(n,r)}$ is obtained. In this article, we reconstruct the universal enveloping algebra ${U({\mathfrak{q}_n})}$ of the queer Lie superalgebra ${\mathfrak{q}_n}$ via ${\mathcal{Q}(n,r)}$, and achieve another explanation of the Schur-Weyl-Sergeev duality. Finally, we depict the Schur-Weyl-Sergeev duality over $\mathbb{Z}$.
\end{abstract}


\maketitle

\section{Introduction}\label{sec_intro}
As an important object in the Lie theory, the queer Lie superalgebra $\qn$ has been intensively studied  in the last two decades.
In order to discover more  properties and applications of {$\qn$},
one usually attempts to establish some relationships between $U(\qn)$, the universal enveloping superalgebra, and other algebras, especially some finite dimensional quotient algebras.
The Schur-Weyl-Sergeev duality is such a remarkable one. This duality was exploited by Sergeev  during the study of the polynomial representations of the queer Lie superalgebra $\qn$  (see \cite{Ser})  
 and queer Schur superalgebras $\Qnr$ were used to bridge representations of
$\qn$ with those of the Sergeev superalgebra $\SerA$.

More precisely, let $V=\mathbb{C}^{n|n}$ be the super vector space.
The natural representation of $\qn$ over $V$ induces the representation $\rho_r$ of  $\Uqn$ over $V^{\otimes r}$.
The Sergeev superalgebra $\SerA$ is defined to be the group algebra of the double cover of symmetric group $\fS_r$ (see \cite{BK2}),
and it is also considered as the degenerating case of Hecke-Clifford superalgebra(see \cite{DW1, Ol}) with $q=1$. The superspace $V^{\otimes r}$ can be endowed with a right $\SerA$-module structure. The Schur-Weyl-Sergeev duality is described as the double centralizer property as follows:
\begin{align*}
	{\End}_{\SerA}(V^{\otimes r})= \rho_r(\Uqn), \qquad
	{\End}_{\Uqn}{(V^{\otimes r})}^{op}=\SerA.
\end{align*}
The queer Schur superalgebra $\Qnr$ was defined as $\rho_r(\Uqn)$  in  \cite{BK},
which means $\Qnr$ is the image of $\Uqn$ over $\rho_r$.
Thus every representation of $\Qnr$ can be inflated to a polynomial representation of $\qn$, and representations of $\SerA$ and $\Uqn$ are tied together through $\Qnr$.

Can such an inflation be made at the positive characteristic level? An affirmative answer to this question would result in a
direct application of the work \cite{BK,BK2} by
Brundan and Kleshchev, in which they determined
the classifications of the irreducible representations of  {$\SerA$} in positive characteristic $p>2$
and the irreducible polynomial  representations of the queer supergroup (scheme) $G=Q(n)$. They also link the representations of $Q(n)$ to those of
$U(\qn)$ via the distribution algebras Dist$(G)$. 

To answer the question above, we need a Schur-Weyl-Sergeev duality at the integral level.
Note that the Schur-Weyl-Sergeev duality is usually considered as a super analog of the Schur-Weyl duality for $\mathfrak{gl}_n$ (see \cite{Sch1,Sch2}). As is well known, the integral Schur-Weyl duality for the quantum group $U_q(\mathfrak{gl}_n)$ was established in
\cite{Du}. This work is built on a new realization of quantum $\mathfrak{gl}_n$
via the quantum Schur algebras given by Beilinson, Lusztig and MacPherson (BLM) in \cite{BLM}. The BLM's idea has been extended to quantum $\mathfrak{gl}_{m|n}$ in \cite{DG} and affine quantum $\widehat{\mathfrak{gl}}_n$ in \cite{DF}. This motivates us to extend it to $\Uqn$.
In \cite{DW2}, Du and Wan showed the  presentation of {$\Qnr$} using the Drinfeld-Jimbo presentation of $\Uqn$ in \cite{LS},
and then reconstructed  $\Qnr$ in term of Sergeev superalgebra in \cite{DW1}. It would be natural to extend the BLM's work to the queer case.

In this article, based on the studies in \cite{DW1,DW2}, we will extend BLM's construction to the queer Lie superalgebra $\mathfrak q_n$. We will reconstruct $\Uqn$ via $\Qnr$, using some explicit multiplication formulas of the generators on the basis, and establish the Schur--Weyl--Sergeev duality at the integral level.
Since {$\SerA$} is the combination of symmetric group and  Clifford  superalgebra,  the structure of $\Qnr$ is much more complicated than the Schur algebra. It requires many sophisticated calculations when we compute the fundamental multiplication formulas in $\Qnr$ in section 4.

This  paper is organized as follows.
We recall the Sergeev superalgebra $\SerA$,
the queer Schur superalgebra $\Qnr$ and its standard basis
in Section \ref{sec_qschur}.
Then we concentrate on the study of some preliminary formulas with respect to the Sergeev superalgebra in Section \ref{sec_sergeev}.
By applying these formulas, Section \ref{sec_mul_qschur} gives the main multiplication formulas,
which will be used  for the realization of $\Uqn$ via $\Qnr$.
In order to construct a  superalgebra isomorphic to $\Uqn$,
it is necessary to draw these queer Schur superalgebras in a unified form.
In Section \ref{sec_spanning}, we construct some uniform spanning set for {$\Qnr$} and study their uniform multiplication formulas among some of its elements
especially the generators, then  we achieve a matrix presentation of $\Qnr$ which is equivalent to the one in \cite{DW2}.
In Section \ref{sec_realization}, referring to the BLM type realization of $U_q(\mathfrak{gl}_n)$,
based on the result of Section \ref{sec_spanning},
we construct a superalgebra {$\USn$} isomorphic to $\Uqn$
from the direct product of $\Qnr$ for all $r\geq 0$.
This superalgebra has a basis with explicit multiplication formulas.
In the last section,
identifying $\Uqn$ with {$\USn$} and applying  the multiplication formulas in Section \ref{sec_spanning},
we get a {$\ZZ$}-subsuperalgebra of {$\Uqn$}, which is the Kostant {$\ZZ$}-form of $\Uqn$,
and then achieve the Schur-Weyl-Sergeev duality over $\ZZ$.

 For the quantum case, Olshanski constructed a quantum deformation $\Uqqn$ of $\Uqn$,
 and established the quantum Schur-Weyl-Sergeev duality in \cite{Ol}.
 Du and Wan investigated queer $q$-Schur superalgebras $\QqscR$ via their presentations in \cite{DW2}, building on
 the Drinfeld-Jimbo type presentation of $\Uqqn$ in \cite{GJKK,LS}. Recently,  Du, Lin and Zhou obtained
 a new realisation of $\Uqqn$ by directly constructing the regular representation via $v$-differential operators in \cite{DLZ}. This approach does not use queer $q$-Schur superalgebras.
All these results suggest the existence of an integral Schur-Weyl-Sergeev duality in the quantum case.
We hope to tackle this problem in the near future.

Throughout the article, assume $n, r\in \NN^+$.  Without special statement, all the algebras are defined over $\CC$. For $i,j \in \ZZ $ with $i<j$, denote $[i,j]=\{i,i+1,\cdots,j\}$.
When we consider the addition between the numbers in $\mathbb N$ and the ones in $\ZG=\{\bar{0},\bar{1}\}$, $\bar{0}$ and $\bar{1}$ are identified with $0$ and $1$ respectively.

\section{The Queer Schur Superalgebra }\label{sec_qschur}

Let {$\fS_r$}  be the symmetric group on $r$ letters,
{$\CMN(n,r) \subset \NN^n$} be  the set of  compositions of $r$ into $n$ parts.
For any given {$\lambda \in \CMN(n,r)$}, denote the standard Young subgroup  of {$\fS_r$} associated to $\lambda$ as $\fS_{\lambda}$, which is the subgroup of {$\fS_r$} stabilizing the following sets
\begin{align*}
\{1,2,\cdots,\lambda_1\},
\{\lambda_1+1,\lambda_1+2,\cdots,\tilde{\lambda}_2\},
\cdots\cdots
\{ \tilde{\lambda}_{n-1}+1, \tilde{\lambda}_{n-1}+2,\cdots,\tilde{\lambda}_n\}.
\end{align*}
where $\tilde{\lambda}_k=\sum_{i=1}^k \lambda_i$.
Let {$\D_{\lambda}$} be the set of shortest representatives of  right cosets of {$\fS_{\lambda}$} in {$\fS_r$},
and  {$\D_{\lambda, \mu} $} be the set of shortest representatives of the {$\fS_{\lambda}$}-{$\fS_{\mu}$} double cosets of {$\fS_r$} for any {$ \lambda,\mu \in \CMN(n,r)$}.

For any {$J \subset \fS_r$}, set
\begin{align*}
	x_{J} = \sum_{w \in J} w.
\end{align*}
Specially, if $J=\fS_{\lambda}$, denote $x_J=x_\lambda$. 
According to \cite[Lemma 7.32]{DDPW},
 for any {$s=(i,i+1)\in\fS_{\lambda}$}, it follows
\begin{equation}
	s x_\lambda= x_{\lambda} s = x_{\lambda}. \label{s_on_S}
\end{equation}

The Clifford algebra {$\C_r$} over  $\CC$  is the algebra  generated by {$c_1, \cdots, c_r$} satisfying
\begin{gather}
	c_i^2 = -1, \quad c_i c_j = - c_j c_i, \qquad \mbox{ for } 1 \le i \ne j \le r; \label{eqclif}
\end{gather}

As in \cite{Ser} and \cite{DW2},
the Sergeev superalgebra {$\SerA$} is defined to be the group algebra of
{$\fS_r \ltimes  \C_r$}  over $\CC$.
In other words, {$\SerA$} is the superalgebra generated by even generators  {$s_1, \cdots, s_{r-1}$}
and odd generators {$c_1, \cdots, c_r$}  subject to \eqref{eqclif} and  the following relations
\begin{align*}
	& s_i^2 = 1, \quad s_i s_j = s_j s_i, \quad 1 \le i,j \le r-1, |i-j|>1,\\
	& s_i s_{i+1} s_i = s_{i+1} s_i s_{i+1}, \quad 1 \le i \le r-2, \\
	& s_i c_i = c_{i+1} s_i, \   s_{i}c_{i+1}=c_is_i, \  s_i c_j = c_j s_i,\ 1\le i \le r-1, 1 \le j \le r, j \ne i, i+1.
\end{align*}

For {$r \ge 1$} and {$1 \le i  \le  j \le r$}, define the elements in $\SerA$:
\begin{align}
	c_{i,j} = c_i + c_{i+1} + \cdots + c_{j-1} + c_j, \label{c_ij}
\end{align}
and
\begin{equation}
\begin{aligned}
	& \SAIJ(i,j) = 1 + s_{i} + s_{i} s_{i+1} + \cdots +  s_{i} s_{i+1}  \cdots s_{j} , \\
	& \SDIJ(j,i) = 1 + s_{j} + s_{j} s_{j-1} + \cdots +  s_{j} s_{j-1}  \cdots s_{i}   .
\end{aligned}	\label{s_ij}
\end{equation}

For {$\lambda = (\lambda_1, \cdots, \lambda_n) \in \CMN(n, r), \alpha \in \ZG^n $},  let
\begin{align}
	c^{\alpha}_{\lambda} = {(c_{ 1, \tilde{\lambda}_1})}^{\alpha_1}
				{(c_{ \tilde{\lambda}_1+1,  \tilde{\lambda}_2})}^{\alpha_2}
				\cdots
				{(c_{\tilde{\lambda}_{n-1}+1,  \tilde{\lambda}_n})}^{\alpha_n}.
	\label{c_alpha}
\end{align}

 The quantum queer Schur superalgebra {$\QqscR$} over commutative ring $R$ is introduced in \cite{DW1}.
Specializing {$R$} to {$\CC$} and {$q=1$},
 {$\QqscR$} degenerates to the queer Schur superalgebra $\Qnr$,
which is defined as
\begin{align*}
	\Qnr = \End_{\SerA}(\bigoplus_{\lambda \in \CMN(n,r)} x_{\lambda} \SerA).
\end{align*}
The algebra $\Qnr$ can also be defined as $\End_{\SerA}(V^{\otimes r})$ in \cite[\S4]{BK}, where $V$ is a superspace.  
Du and Wan proved that both definitions are equivalent.

Given $M=(m_{i,j})\in \MN(n)$, denote $\snorm{M}=\sum_{i,j}m_{i,j}$, and
\begin{align*}
&\ro(M)=\left(
	\sum_{i=1}^n m_{1,i},
	\sum_{i=1}^n m_{2,i},
	\cdots,
	\sum_{i=1}^n m_{n,i}
\right), \\
&\co(M)=\left(
	\sum_{i=1}^n m_{i,1},
	\sum_{i=1}^n m_{i,2},
	\cdots,
	\sum_{i=1}^n m_{i,n}
\right), \\
	& \nu_{M} = (m_{1,1}, \cdots, m_{n,1}, m_{1,2}, \cdots, m_{n,2}, \cdots,  m_{1,n}, \cdots, m_{n,n}),
\end{align*}
for given {$h \in [1,n-1]$}, {$k \in [1,n]$}, denote
\begin{align}
&	\tilde{m}_{h,k}
	= \sum_{j=1}^{k-1}\sum_{i=1}^{n} m_{i,j} + \sum_{i=1}^{h} m_{i,k}. \label{mhk}
\end{align}

Set
\begin{align*}
& \MNR(n,r) = \{ M \in \MN(n) \where \snorm{M} = r \}.
\end{align*}
According to \cite[Section 4.2]{DDPW}, there is a bijection
\begin{align*}
	\mathfrak{j}: \MNR(n,r) &\to \{ (\lambda,  d, \mu) \where \lambda,\mu \in \CMN(n,r), d \in \D_{\lambda, \mu} \} \\
	M & \mapsto (\ro(M), d_M, \co(M)).
\end{align*}
The more explicit presentation of {$d_M$} is shown in Section 3.

Denote
\begin{align*}
	& \MNZN(n) = \{ M=(\SUP{M}) \where \SE{M} \in \MN(n), \SO{M} \in \MZ(n) \}, \\
	& \MNZ(n,r) = \{ M=(\SUP{M}) \in \MNZN(n) \where  \snorm{\SE{M}} + \snorm{\SO{M}} = r \}.
\end{align*}
 For {$M \in  \MNZN(n) $} with  {$\SE{M}=(\SEE{m}_{i,j}), \SO{M}=(\SOE{m}_{i,j})$}, denote  {$\ABSUM{M} = (m_{i,j})=\SE{M}  + \SO{M}$} where
{$m_{i,j} = \SEE{m}_{i,j} + \SOE{m}_{i,j} $}. In order to simplify the notations, set
\begin{equation}\label{notations}
\begin{aligned}
\snorm{M} = \sum_{1 \le i,j \le n}m_{i,j},\,\ro({M}) =  \ro(\ABSUM{M}),\,\co({M}) =  \co(\ABSUM{M}),\,\nu_{M} =\nu_{\ABSUM{M}},\,d_{M}= d_{\ABSUM{M}}.
\end{aligned}
\end{equation}
%


For any $M \in \MNZ(n,r)$, let $\lambda = \ro(M), \mu = \co(M)$,  and
\begin{align}
	c_{M} =  c_{\nu_{M}}^{\nu_{\SO{M}}}   \in \C_r, \qquad
	T_{M} =  x_{\lambda} d_{M} c_{M}
			\sum _{{\sigma} \in \D_{{\nu}_{M}} \cap \fS_{\mu}} \sigma \in \SerA. \label{c_A}
\end{align}

According to \cite[Theorem 5.3]{DW1},  $\Qnr$ has a standard basis
\begin{align}
	\{ \phi_{M} \where M \in \MNZ(n,r)\}, \label{basis_phi}
\end{align}
where {$\phi_{M}$} is defined as
\begin{align*}
	\phi_{M}(x_{\mu}h) = \delta_{\mu, \co(M)} T_{M}h,
	\quad  \mbox{ for all } h \in \SerA, \  \mu \in \CMN(n,r).
\end{align*}

\section{Some Preliminary Multiplication Formulas in {$\SerA$}}\label{sec_sergeev}

During the process of constructing {$\Uqn$} via queer Schur superalgebras {$\Qnr (r\geq 0)$}, we need some multiplication formulas for some special elements of the standard basis
in $\Qnr$.
Before all these works, some formulas for {$\SerA$} must be considered at first.

For {$A =(a_{i,j})= \MN(n)$} and  {$h \in [1,n-1]$},  {$k \in [1,n]$},
set
\begin{align*}
A^+_{h,k}=A+E_{h,k}-E_{h+1,k}, \qquad A^-_{h,k}=A-E_{h,k}+E_{h+1,k}.
\end{align*}
According to \cite[Lemma 3.2]{DGZ},  $d_A$ is presented as
\begin{displaymath}\label{d_A}
	d_A =
		(w_{2,1}w_{3,1}\cdots w_{n,1})
		(w_{2,2}w_{3,2}\cdots w_{n,2})
		\cdots
		(w_{2,n-1}w_{3,n-1}\cdots w_{n,n-1})
\end{displaymath}
or
\begin{align}
	d_A =
	\begin{pmatrix}
		w_{2,1} &w_{2,2} &\cdots &w_{2,k} &\cdots &w_{2,n-1}  \\
		w_{3,1} &w_{3,2} &\cdots &w_{2,k} &\cdots &w_{3,n-1}  \\
		\vdots &\vdots   &\cdots &\vdots   &\cdots&\vdots     \\
		w_{n,1} &w_{n,2} &\cdots  &w_{n,k} &\cdots &w_{n,n-1}  \\
	\end{pmatrix} \quad,
\end{align}
where {$w_{i,j}$} is defined as : \\
	\ \qquad if {$a_{i,j} = 0$}, or {$a_{i,j}>0$} but {$\sigma_{i-1,j} = \tilde{a}_{i-1,j}$}, then {$w_{i,j} = 1$}; \\
	\ \qquad if {$a_{i,j}>0$} and {$\sigma_{i-1,j} > \tilde{a}_{i-1,j}$}, then
\begin{equation}
\begin{aligned}
	w_{i,j} =
		&(s_{\sigma_{i-1,j}} s_{\sigma_{i-1,j} - 1} \cdots s_{\tilde{a}_{i-1,j} +1}) \\
		&(s_{\sigma_{i-1,j}+1} s_{\sigma_{i-1,j} } \cdots s_{\tilde{a}_{i-1,j} +2}) \\
		&\cdots \\
		&(s_{\sigma_{i-1,j}+a_{i,j}-1} s_{\sigma_{i-1,j}+a_{i,j}-2 } \cdots s_{\tilde{a}_{i,j}})
\end{aligned}		\label{wij}
\end{equation}
where
\begin{displaymath}
	{\sigma}_{i,j} = \sum_{h=1}^{n}\sum_{k=1}^{j-1} a_{h,k}
		+ \sum_{h \le i, k \ge j}  a_{h,k}.
\end{displaymath}

The following proposition is  from  \cite{DGZ}.
\begin{prop}{\cite[Proposition 3.6]{DGZ}}\label{prop_shift}
For  $A =(a_{i,j})= \MN(n) $,
set $\lambda = \ro(A)$ and  $a_k=\sum^{k-1}_{u=1}a_{h+1, u}$,  $b_k=\sum^{n}_{u=k+1}a_{h, u}$.
 Then  for  {$ 0 \leq p <a_{h+1,k} $} and  {$0 \leq  q<a_{h,k}$},
\begin{align*}
{\rm (i)}\quad
	&
	s_{\tilde{\lambda}_{h}+1} \cdots s_{\tilde{\lambda}_{h}+a_k+p }  {d_{A}} = 	
	s_{\tilde{\lambda}_{h}} s_{\tilde{\lambda}_{h}-1}\cdots s_{\tilde{\lambda}_{h}-b_k+1}  {d_{A^+_{h,k}}}
  	s_{\tilde{a}_{h,k}+1} \cdots s_{\tilde{a}_{h,k}+p}, \\
{\rm (ii)} \quad
	&
	s_{\tilde{\lambda}_{h}-1} \cdots s_{\tilde{\lambda}_{h}-b_k-q } {d_{A}} =
	s_{\tilde{\lambda}_{h}} s_{\tilde{\lambda}_{h}+1}\cdots s_{\tilde{\lambda}_{h}+a_{k}-1}  {d_{A^-_{h,{k}}}}
	s_{\tilde{a}_{h,{k}}-1} \cdots s_{\tilde{a}_{h,{k}}-q}.
\end{align*}

\end{prop}

\begin{lem}\label{CS}
For $i_t<k+t-1\leq j_t $ with $1\leq t\leq l $, we have
$$
\begin{aligned}
&c_k (s_{j_1}s_{j_1-1}\cdots s_{i_1})(s_{j_2}s_{j_2-1}\cdots s_{i_2})\cdots(s_{j_l}s_{j_l-1}\cdots s_{i_l})\\
=&(s_{j_1}s_{j_1-1}\cdots s_{i_1})(s_{j_2}s_{j_2-1}\cdots s_{i_2})\cdots(s_{j_l}s_{j_l-1}\cdots s_{i_l})c_{k+l}.
\end{aligned}$$
\end{lem}
\begin{proof}
When $l=1$, according to the assumption, $i_1<k\leq j_1$. From the relations in $\SerA$,
\begin{equation*}
\begin{aligned}
&c_k(s_{j_1}s_{j_1-1}\cdots s_{i_1})=(s_{j_1}s_{j_1-1}\cdots s_{k+1})c_ks_k(s_{k-1}\cdots s_{i_1})\\
&=(s_{j_1}s_{j_1-1}\cdots s_{k+1})s_kc_{k+1}(s_{k-1}\cdots s_{i_1})=(s_{j_1}s_{j_1-1}\cdots s_{i_1})c_{k+1}.
\end{aligned}
\end{equation*}

By induction on $l$, and assume that the formula is correct for all $i<l$,
then
\begin{align*}
&c_k (s_{j_1}s_{j_1-1}\cdots s_{i_1})(s_{j_2}s_{j_2-1}\cdots s_{i_2})\cdots(s_{j_l}s_{j_l-1}\cdots s_{i_l})\\
=&[(s_{j_1}s_{j_1-1}\cdots s_{i_1})\cdots(s_{j_{l-1}}s_{j_{l-1}-1}\cdots s_{i_{l-1}})] c_{k+l-1}(s_{j_l}s_{j_l-1}\cdots s_{i_l}).
\end{align*}
Because $i_l<k+l-1\leq j_l$, similar to the case $l=1$, it follows
\begin{align*}
c_{k+l-1}(s_{j_l}s_{j_l-1}\cdots s_{i_l})=(s_{j_l}s_{j_l-1}\cdots s_{i_l})c_{k+l} ,
\end{align*}
which proves the result.
\end{proof}

\begin{prop}\label{propR3}
For $A =(a_{i,j})= \MNZN(n) $, let $\lambda = \ro(A)$, with the  notations above, the following formulas can be derived:
\begin{align*}
{\rm (i)} \quad &c_{\tilde{\lambda}_{h}+a_k+p+1} {d_A} =  {d_A} c_{\tilde{a}_{h,k}+p+1}, \\
{\rm (ii)}  \quad &c_{\tilde{\lambda}_{h}-b_{k}-q} {d_A} = {d_A} c_{\tilde{a}_{h,{k}} - q }.
\end{align*}
\end{prop}

\begin{proof}
The key to the proof is  comparing the index of {$c_{i}$} with the indices of {$s$} in the expression of {$d_A$}. Only {\rm (i)} will be proved, and  {\rm (ii)} can be derived similarly.

Denote {$w_{2,0}\cdots w_{n,0} = 1$}, {$\sum_{l=h+2}^n \sum_{i=1}^{0}a_{l,i} = 0$}.
First of all, we claim
\begin{equation}
\begin{aligned}\label{clm_eq}
&c_{\tilde{\lambda}_{h}+a_k+p+1} (w_{2,1}\cdots w_{n,1}) \cdots (w_{2,k-1}\cdots w_{n,k-1})\\
=& (w_{2,1}\cdots w_{n,1}) \cdots (w_{2,k-1}\cdots w_{n,k-1})
	 c_{\tilde{\lambda}_{h}+ a_k + p + 1 + \sum_{l=h+2}^n \sum_{i=1}^{k-1}a_{l,i}}.
\end{aligned}
\end{equation}

Indeed, when $k=1$,  there is nothing to prove.
When $k=2$, since $a_k=a_{h+1,1}$, for all $w_{i,1}$ with $i\leq h+1$,
the maximal index of $s$ is { $\tilde{\lambda}_h+a_{h+1,1}-1 (< \tilde{\lambda}_h+a_k+p)$}, then
\begin{align*}
 c_{\tilde{\lambda}_{h}+a_k+p+1}  w_{2,1} \cdots w_{h+1,1}
 =   w_{2,1} \cdots w_{h+1,1} c_{\tilde{\lambda}_{h}+a_k+p+1} .
\end{align*}
%
According to the expression of $w_{i,1}$ for $i>h+1$,
which is the product of
\begin{align*}
(s_{\tilde{\lambda}_{i-1}+\tilde{a}_{i,1}+t} s_{\tilde{\lambda}_{i-1}+\tilde{a}_{i,1}+t-1}
	\cdots s_{\tilde{a}_{i-1,1}+t+1})
\end{align*}
for all $t$  satisfying $0 \leq t <a_{i,1}$,
since
\begin{align*}
\tilde{\lambda}_{i-1}+\tilde{a}_{i,1}+t \geq \tilde{\lambda}_h+a_k+p+1+\sum_{l=h+2}^{i-1}a_{l,1}+t+1 > \tilde{a}_{i-1,1}+t+1,
\end{align*}
 Lemma \ref{CS} shows
\begin{align*}
c_{ \tilde{ \lambda}_{h} + a_k + p + 1 } (w_{h + 2, 1} \cdots w_{n,1} )
= (w_{h+2,1}  \cdots  w_{n,1}) c_{\tilde{\lambda}_{h} + a_k + \sum_{l=h+2}^n  a_{l,1}+p+1}.
\end{align*}

Assume $k>2$, and for any $u<k-1$, we have
\begin{align*}
&c_{\tilde{\lambda}_{h}+a_k+p+1} (w_{h+2,1} \cdots w_{n,1}) \cdots (w_{2,u}\cdots w_{n,u})\\
=&(w_{h+2,1} \cdots w_{n,1}) \cdots (w_{2,u} \cdots w_{n,u})
	c_{\tilde{\lambda}_{h}+a_k+\sum_{l=h+2}^n \sum_{t=1}^u a_{l,t}+p+1 }.
\end{align*}
By induction on $u$, when $u=k-1$,
consider
\begin{align*}
c_{\tilde{\lambda}_{h}+a_k+\sum_{l=h+2}^n \sum_{t=1}^{k-2} a_{l,t}+p+1} (w_{2,k-1} w_{3,k-1} \cdots w_{n,k-1}).
\end{align*}
Reminding the expression of $w_{i,k-1}$, for any $i<h+2$, the maximal index of $s$ is $\sigma_{i-1,k-1}+a_{i,k-1}-1$ which is less than
$\tilde{\lambda}_{h} + a_k + \sum_{l=h+2}^n \sum_{t=1}^{k-2} a_{l,t}+p $,
and this causes $c_{\tilde{\lambda}_{h} + a_k + \sum_{l=h+2}^n \sum_{t=1}^u a_{l,t}+p+1}$ commutes with these $w_{i,k-1}$.
For $i\geq h+2$, since $w_{i,k-1}$ is the product of $(s_{\sigma_{i-1,k-1}+t}s_{\sigma_{i-1,k-1}+t-1}\cdots s_{\tilde{a}_{i-1,k-1}+t})$
with $1\leq t\leq a_{i,k-1}$,
and since
\begin{align*}
\sigma_{i-1,k-1}+t\geq \tilde{\lambda}_{h}+a_k+\sum_{l=h+2}^n\sum_{t=1}^{k-2} a_{l,t}+\sum_{j=h+1}^{i-1}a_{j,k-1}+t+p+1> \tilde{a}_{i-1,k-1}+t,
\end{align*}
applying Lemma \ref{CS} again, it follows
\begin{align*}
&c_{\tilde{\lambda}_{h} + a_k + \sum_{l=h+2}^n \sum_{t=1}^{k-2} a_{l,t}+p+1} (w_{h+2,k-1} \cdots w_{n,k-1})\\
=&(w_{h+2,k-1}\cdots w_{n,k-1})
	c_{\tilde{\lambda}_{h} + a_k + \sum_{l=h+2}^n \sum_{t=1}^{k-1} a_{l,t}+p+1 },
\end{align*}
 hence the equation \eqref{clm_eq} is proved.

Now we consider
\begin{align*}
\X = c_{\tilde{\lambda}_{h}+a_k+\sum_{l=h+2}^n \sum_{t=1}^{k-1}a_{l,t}+p+1}
		(w_{2,k}\cdots w_{n,k})\cdots (w_{2,n-1}\cdots w_{n,n-1}).
\end{align*}
Notice
$\tilde{\lambda}_{h}+a_k+\sum_{l=h+2}^n\sum_{t=1}^{k-1}a_{l,t}+p+1=\sigma_{h,k}+p+1$
where $0\leq p+1\leq a_{h+1,k}$.
For $i<h+1$, the maximal index of $s$ in $w_{i,k}$ is $\sigma_{h-1,k}+a_{h,k}-1$
which is less than $\sigma_{h,k}+p$,
so $c_{\sigma_{h,k}+p+1}$
commutes with these $w_{i,k}$,
then
\begin{align*}
\X =
		 w_{2,k}\cdots w_{h,k} c_{\sigma_{h,k}+p+1}   w_{h+1,k} \cdots w_{n,k} \cdots  w_{2,n-1}\cdots w_{n,n-1}.
\end{align*}
Because
\begin{align*}
& c_{\sigma_{h,k}+p+1}   w_{h+1,k}  \\
&=  c_{\sigma_{h,k}+p+1}   (s_{\sigma_{h,k}} \cdots s_{\tilde{a}_{h,k} +1})
		\cdots
		(s_{\sigma_{h,k}+p} \cdots s_{\tilde{a}_{h,k} + p+1})
		\cdots
		(s_{\sigma_{h,k}+a_{h+1,k}-1} \cdots s_{\tilde{a}_{h+1,k}}) \\
&=  (s_{\sigma_{h,k}} \cdots s_{\tilde{a}_{h,k} +1})
		\cdots
		c_{\sigma_{h,k}+p+1}
		(s_{\sigma_{h,k}+p} \cdots s_{\tilde{a}_{h,k} + p+1})
		\cdots
		(s_{\sigma_{h,k}+a_{h+1,k}-1} \cdots s_{\tilde{a}_{h+1,k}}) \\
& \qquad ( \mbox{apply the relation } c_{i+1} s_i =s_i c_i  )\\
&=  (s_{\sigma_{h,k}} \cdots s_{\tilde{a}_{h,k} +1})
		\cdots
		(s_{\sigma_{h,k}+p} \cdots s_{\tilde{a}_{h,k} + p+1})
		c_{\tilde{a}_{h,k} + p+1}
		\cdots
		(s_{\sigma_{h,k}+a_{h+1,k}-1} \cdots s_{\tilde{a}_{h+1,k}}) \\
&=   w_{h+1,k}
		c_{\tilde{a}_{h,k} + p+1}
\end{align*}
and
\begin{align*}
&c_{\tilde{a}_{h,k} + p+1}   w_{h+2,k} \cdots   w_{n,k}
= w_{h+2,k} \cdots   w_{n,k}  c_{\tilde{a}_{h,k} + p+1},
\end{align*}
it follows
\begin{align*}
\X&=
	w_{2,k}\cdots w_{h,k} w_{h+1,k} \cdots w_{n,k}  c_{\tilde{a}_{h,k} + p+1}
	\cdots  w_{2,n-1}\cdots w_{n,n-1} \\
&=
	w_{2,k}\cdots w_{h,k} w_{h+1,k} \cdots w_{n,k}
	\cdots  w_{2,n-1}\cdots w_{n,n-1} c_{\tilde{a}_{h,k} + p+1} .
\end{align*}
As a consequence,  {\rm(i)} is proved.
\end{proof}

\begin{prop}\label{prop_splus}
	Use the notations in \eqref{s_ij}.
	For  {$ h \in [1,n-1], k \in [1,n] $}, {$A=(a_{i,j}) \in \MN(n)$}
	with {$a_{h+1,k} \ge 1$},  let $\lambda = \ro{(A)}$, and $\mu = \co{(A)}$,
	then
\begin{align*}
	& \SAIJ({{\tilde{a}_{h,k}} + 1},  {{\tilde{a}_{h,k}} + a_{h+1,k}-1})
		\sum_{\sigma\in\mathcal{D}_{\nu_A}\cap \mathfrak{S}_\mu}\sigma
	 = \SDIJ({\tilde{a}_{h,k}}, {\tilde{a}_{h,k} - a_{h,k} + 1})
		\sum_{\sigma\in\mathcal{D}_{\nu_{A^+_{h,k}}}\cap \mathfrak{S}_\mu}\sigma.
\end{align*}
\end{prop}
\begin{proof}
Recall $A^+_{h,k}=A+E_{h,k}-E_{h+1,k}$, and
\begin{align*}
\nu_A=(\cdots, a_{h,k},a_{h+1,k},\cdots), \qquad
\nu_{A^+_{h,k}}=(\cdots,a_{h,k}+1,a_{h+1,k}-1,\cdots),
\end{align*}
set $\alpha=(\cdots,a_{h,k},1,a_{h+1,k}-1,\cdots)$,
where the parts which do not appear are the same as {$\nu_{A}$}.
One observes that
$\mathfrak{S}_\alpha\subseteq\mathfrak{S}_{\nu_A}$,
$ \mathfrak{S}_\alpha\subseteq\mathfrak{S}_{\nu_{A^+_{h,k}}}$ and
\begin{align*}
&\mathcal{D}_{\alpha}\cap \mathfrak{S}_{\nu_A}=\{ 1, s_{{\tilde{a}_{h,k}}+1}, s_{{\tilde{a}_{h,k}} + 1}s_{{\tilde{a}_{h,k}} + 2}, \cdots,
		s_{{\tilde{a}_{h,k}} + 1}s_{{\tilde{a}_{h,k}} + 2} \cdots s_{{\tilde{a}_{h,k}} + a_{h+1, k}-1}\}, \\
&\mathcal{D}_{\alpha}\cap \mathfrak{S}_{\nu_{A^+_{h,k}}}= \{ 1, s_{{\tilde{a}_{h,k}}}, s_{{\tilde{a}_{h,k}}}s_{{\tilde{a}_{h,k}}-1}, \cdots,
		s_{{\tilde{a}_{h,k}}}s_{{\tilde{a}_{h,k}} - 1} \cdots s_{{\tilde{a}_{h,k}}-a_{h, k}+1}\}.
\end{align*}

So
\begin{align*}
\sum_{\sigma\in\mathcal{D}_\alpha\cap\mathfrak{S}_\mu}\sigma
	=\sum_{\sigma_1\in\mathcal{D}_{\alpha}\cap \mathfrak{S}_{\nu_A}}	
		\sum_{\sigma_2\in\mathcal{D}_{\nu_A}\cap \mathfrak{S}_\mu}
			\sigma_1\sigma_2
	=\sum_{\sigma_1\in\mathcal{D}_{\alpha}\cap \mathfrak{S}_{\nu_{A^+_{h,k}}}}
		\sum_{\sigma_2\in\mathcal{D}_{\nu_{A^+_{h,k}}}\cap \mathfrak{S}_\mu}
		\sigma_1\sigma_2,
\end{align*}
hence the assertion is proved.
\end{proof}

For $A=(a_{i,j}) \in \MN(n),$  recall that $A^-_{h,k}=A-E_{h,k}+E_{h+1,k}$.
\begin{cor}\label{prop3R5}
	For {$ h,k \in [1,n] $} and {$A=(a_{i,j}) \in \MN(n)$}
	with {$a_{h,k} \ge 1$}, $\ro{(A)}=\lambda$ and $\co{(A)}=\mu$, we have
\begin{align*}
	& \SDIJ({{\tilde{a}_{h,k}} - 1} , {{\tilde{a}_{h,k}} - a_{h,k} + 1})
		\sum_{\sigma\in\mathcal{D}_{\nu_A}\cap \mathfrak{S}_\mu}\sigma
	= \SAIJ( {{\tilde{a}_{h,k}}}, {{\tilde{a}_{h,k}} + a_{h+1,k} - 1})
		\sum_{\sigma\in\mathcal{D}_{\nu_{A^-_{h,k}}}\cap \mathfrak{S}_\mu}\sigma.
\end{align*}
\end{cor}
\begin{proof}
Replace $A$ with $A^-_{h,k}$ and apply Proposition \ref{prop_splus}, the result is proved.
\end{proof}

Recall the notation  {$	c_{i,j}$} in \eqref{c_ij}
and {$\SAIJ(i, j)$} in \eqref{s_ij}.
\begin{lem}\label{ComC-S}
Suppose {$u$}, {$v$}, {$t$} are positive integers and $t<u$, then
\begin{align*}
{\rm (i)}\quad &c_{u-t+1,u+1} \SDIJ(u, {u-t+1})
 = \SDIJ(u, {u-t+1}) c_{u-t+1,u+1};\\
{\rm (ii)}\quad &c_{u,u+v} \SAIJ(u,{u+v-1})
 =\SAIJ(u, {u+v-1}) c_{u,u+v}.
\end{align*}
\end{lem}
\begin{proof}
For any integer {$k$} such that {$u - t + 1 \le k < u+1$},
since
$c_k s_k = s_k c_{k+1} $,  $c_{k+1} s_k =s_k c_k$,
it follows $(c_k + c_{k+1}) s_k = s_k (c_k + c_{k+1})$.
Meanwhile,  $c_j s_k=s_k c_j$ for $ j \neq k, k+1$.
Consequently, {$c_{u-t+1,u+1} s_k  = s_k  c_{u-t+1,u+1} $}
and {\rm (i)} is proved.
The proof of
{\rm (ii)} is similar to {\rm (i)}.
\end{proof}

Lemma \ref{propupper1}-\ref{prop3F10.0} are important for the next section, and
they can be proved by
applying the relations of the generators of Sergeev superalgebra,
especially, $s_i x_\mu=x_\mu s_i = x_\mu$ for any $s_i\in \fS_\mu$.
\begin{lem}\label{propupper1}
 For {$\mu \in \CMN(n, r)$}, assume {$u, v, t \in [1, r]$} with {$t \le u$} and $s_i\in \fS_{\mu}$ for any $i$ satisfying $u-t+1\leq i\leq u$.
	Then
\begin{align*}
{\rm (i)}\quad & x_{\mu} \SDIJ(u, {u-t+1})
		 =  (t+1) x_{\mu} ;\\
	{\rm (ii)}\quad & x_{\mu} c_{u+1} \SDIJ(u, {u-t+1})=x_{\mu}c_{u-t+1,u+1}; \\
	{\rm (iii)}\quad &x_{\mu} c_{u+1}c_{u+1,u+v}
		\SDIJ(u, {u-t+1}) \\
	& =
	\left\{
	\begin{aligned}
		& -(t+1) x_{\mu}, \mbox{ if } v=1;  \\
		& -(t+1) x_{\mu}
			+ x_{\mu} c_{u-t+1,u+1}c_{u+2,u+v}, \mbox{ if } v >1 ;\\
	\end{aligned}
	\right. \\
	{\rm (iv)}\quad
		&x_{\mu} c_{u+1}c_{u-t+1,u} \SDIJ(u, {u-t+1})=0;\\
	{\rm (v)}\quad
		&x_{\mu} c_{u+1}c_{u-t+1,u}c_{u+1,u+v} \SDIJ(u, {u-t+1})
		 = t x_{\mu}c_{u-t+1,u+1}.
\end{align*}
\end{lem}
\begin{proof}
	For {$j \in [u-t+1, u]$},
	{$s_j \in \fS_{\mu}$}
	implies {$x_{\mu} s_j = x_{\mu}$}.
	
{\rm (i)}
The proof is trivial because of equation \eqref{s_on_S}.

{\rm (ii)}Apply  the relations for generators of $\SerA$,
\begin{align*}
	& x_{\mu} c_{u+1} \SDIJ(u, {u-t+1}) \\
	&= x_{\mu} (c_{u+1} + s_u c_u+ s_u s_{u-1}c_{u-1} + \cdots + s_u s_{u-1} \cdots s_{u-t+1}c_{u-t+1})  \\
	&=  x_{\mu} ( c_{u+1} + c_u  + c_{u-1} + \cdots + c_{u-t+1})=x_{\mu}c_{u-t+1,u+1} .
\end{align*}

	{\rm (iii)}
	When {$v=1$}, $c_{u+1,u+v}=c_{u+1}$, then
\begin{align*}
	&x_{\mu} c^2_{u+1} \SDIJ(u, {u-t+1})
	  = -x_{\mu}\SDIJ(u, {u-t+1})= -(t+1) x_{\mu} .
\end{align*}
	When {$v >1$},
\begin{align*}
	x_{\mu} c_{u+1}c_{u+1,u+v}\SDIJ(u, {u-t+1})
	& = x_{\mu} c_{u+1}(c_{u+1} + c_{u+2,u+v})
		\SDIJ(u, {u-t+1}) \\
	& = x_{\mu} c_{u+1}^2 \SDIJ(u, {u-t+1})
		 + x_{\mu} c_{u+1}c_{u+2,u+v}
			\SDIJ(u, {u-t+1}) \\
	& \mbox{( since $c_{u+2,u+v}$ commutes with $\SDIJ(u, {u-t+1}) $ )} \\
	&= -(t+1)x_{\mu}
		 + x_{\mu} c_{u+1}
			\SDIJ(u, {u-t+1}) c_{u+2,u+v} \\
	&= -(t+1)x_{\mu}
		+ x_{\mu} c_{u-t+1,u+1}c_{u+2,u+v}.
\end{align*}

{\rm (iv)}
	Rewriting $c_{u-t+1,u}=c_{u-t+1,u+1}-c_{u+1}$, we have
\begin{align*}
&x_{\mu} c_{u+1}c_{u-t+1,u}\SDIJ(u, {u-t+1})
  = x_{\mu} c_{u+1}(c_{u-t+1,u+1}-c_{u+1})\SDIJ(u, {u-t+1}).
\end{align*}
	Since $c^2_{u-t+1,u+1}=-(t+1)$, according to Lemma \ref{ComC-S}{\rm (i)}and the assertion {\rm (ii)},
\begin{align*}
&x_{\mu} c_{u+1}c_{u-t+1,u+1}\SDIJ(u, {u-t+1})
 =x_{\mu}c^2_{u-t+1,u+1}=-(t+1)x_\mu.
\end{align*}
And since {$x_{\mu} c_{u+1}^2
		\SDIJ(u, {u-t+1})=-(t+1)x_\mu,$}
as a consequence,
$$x_{\mu} c_{u+1}c_{u-t+1,u}\SDIJ(u, {u-t+1})=0.$$

	{\rm (v)} Let {$\X=x_{\mu} c_{u+1}c_{u-t+1,u}c_{u+1,u+v}  \SDIJ(u, {u-t+1})$}. \\
If {$v=1$}, then
\begin{align*}
	\X=x_{\mu} c_{u+1} c_{u-t+1,u}
	c_{u+1}  \SDIJ(u, {u-t+1}).
\end{align*}
Because
{$	c_{u+1}c_{u-t+1,u} c_{u+1}=-c_{u-t+1,u}c_{u+1}c_{u+1}=c_{u-t+1,u} $},
we can obtain
\begin{align*}
	\X & =  x_{\mu} c_{u-t+1,u} \SDIJ(u, {u-t+1})
	  =  x_{\mu} (c_{u-t+1,u+1}-c_{u+1}) \SDIJ(u, {u-t+1}).
\end{align*}
According to Lemma \ref{ComC-S}{\rm (i)} and the assertion {\rm (ii)}, we have
\begin{align*}
	\X =(t+1)x_\mu c_{u-t+1,u+1}-x_\mu c_{u-t+1,u+1}=tx_\mu c_{u-t+1,u+1}.
\end{align*}
So the assertion {\rm (v)} is correct in this case.

If {$v >1$}, then
\begin{align*}
	\X & =  x_{\mu} c_{u+1}c_{u-t+1,u}(c_{u+1}+c_{u+2,u+v})  \SDIJ(u, {u-t+1}) \\
	& =  x_{\mu} c_{u+1}c_{u-t+1,u} c_{u+1}
		\SDIJ(u, {u-t+1}) \\
	&\quad +  x_{\mu} c_{u+1}c_{u-t+1,u}c_{u+2,u+v}  \SDIJ(u, {u-t+1}) .
\end{align*}
	As $c_{u+2,u+v}$ commutes with $\SDIJ(u, {u-t+1})$,  assertion {\rm (iv)} leads
$$x_{\mu} c_{u+1}c_{u-t+1,u}c_{u+2,u+v}\cdot \SDIJ(u, {u-t+1})=0.$$
According to the result in case $v=1$, we have
{$\X=tx_\mu c_{u-t+1,u+1}$} and hence the result is proved.
\end{proof}

Similar to Lemma \ref{propupper1}, 
the proof of the following lemmas is straight forward,
so we omit it.
\begin{lem}\label{prop_upper0}Suppose {$\mu \in \CMN(n, r)$},
	and {$u, v, t \in [1, r]$} with {$t \le u$}
	such that  {$s_u, s_{u-1}, \cdots, s_{u-t+1} \in \fS_{\mu}$}.
We can obtain the following assertions
\begin{align*}
	{\rm (i)}\quad &x_{\mu} c_{u+1,u+v}\SDIJ(u, {u-t+1})
	   =
	\left\{
	\begin{aligned}
		& x_{\mu} c_{u-t+1,u+1},
			\mbox{ if } v=1;  \\
		& x_{\mu} c_{u-t+1,u+1}
			+ (t+1)x_{\mu}c_{u+2,u+v}, \mbox{ if } v >1 ;\\
	\end{aligned}
	\right. \\
	{\rm (ii)}\quad
		&x_{\mu}c_{u-t+1,u} \SDIJ(u, {u-t+1}) = t x_{\mu} c_{u-t+1,u+1}; \\
	{\rm (iii)}\quad
		&x_{\mu}  c_{u-t+1,u}c_{u+1,u+v}  \SDIJ(u, {u-t+1})
	  = \left\{
	\begin{aligned}
		& 0,
			\mbox{ if } v=1; \\
		& t x_{\mu}c_{u-t+1,u+1}c_{u+2,u+v}, \mbox{ if } v >1.\\
	\end{aligned}
	\right.
\end{align*}
\end{lem}

\begin{lem}\label{prop3F10}
Suppose {$\mu \in \CMN(n, r)$}, and {$u, v, t \in [1, r]$} with {$t \le u$} such that  {$s_u, s_{u+1}, \cdots, s_{u+v-1} \in \fS_{\mu}$}.
We can obtain the following assertions
\begin{align*}
{\rm (i)} \quad & x_{\mu}
		\SAIJ(u, {u+v-1})= (v+1) x_{\mu} ; \\
{\rm (ii)} \quad  & x_{\mu} c_{u}
		\SAIJ(u, {u+v-1})= x_{\mu}c_{u,u+v};\\
{\rm (iii)} \quad  &x_{\mu} c_{u} c_{u+1,u+v}  \SAIJ(u, {u+v-1})= 0; \\
{\rm (iv)} \quad  &x_{\mu} c_{u}c_{u-t+1,u}
		\SAIJ(u, {u+v-1})
	  =
	\left\{
	\begin{aligned}
		& -(v+1) x_{\mu}, \mbox{ if } t=1; \\
		& -(v+1) x_{\mu}
		- x_{\mu}c_{u-t+1,u-1} c_{u,u+v},
			\mbox{ if } t>1; \\
	\end{aligned}
	\right. \\
{\rm (v)} \quad  &x_{\mu} c_{u} c_{u-t+1,u} c_{u+1,u+v} \SAIJ(u, {u+v-1})
  = -v x_{\mu} c_{u,u+v}.
\end{align*}
\end{lem}

\begin{lem}\label{prop3F10.0}
	For {$\mu \in \CMN(n, r)$}, suppose
	{$u, v, t \in [1, r]$} with {$t \le u$}
	such that {$s_u, s_{u+1}, \cdots, s_{u+v-1} \in \fS_{\mu}$}.
	The following assertions can be obtained
\begin{align*}
	{\rm (i)}\quad
		&x_{\mu}  c_{u+1,u+v} \SAIJ(u, {u+v-1}) = v x_{\mu} c_{u,u+v} ; \\
	{\rm (ii)}\quad &x_{\mu} c_{u-t+1,u}
		\SAIJ(u, {u+v-1})
	  =
	\left\{
	\begin{aligned}
		& x_{\mu} c_{u,u+v}, \mbox{ if } t=1; \\
		& x_{\mu} c_{u,u+v}+ (v+1) x_{\mu}c_{u-t+1,u-1},
			\mbox{ if } t>1; \\
	\end{aligned}
	\right. \\
	{\rm (iii)}\quad
		&x_{\mu}  c_{u-t+1,u}c_{u+1,u+v}  \SAIJ(u, {u+v-1})
		  =
	\left\{
	\begin{aligned}
		& 0,
			\mbox{ if } t=1; \\
		&  v x_{\mu}c_{u-t+1,u-1} c_{u,u+v},
			\mbox{ if } t>1.
	\end{aligned}
	\right.
\end{align*}
\end{lem}

\section{Some Multiplication Formulas in {$\Qnr$} }\label{sec_mul_qschur}

As the BLM realizations of quantum group $U_q(\mathfrak{gl}_n)$ in \cite{BLM}
and quantum supergroup $U_q(\mathfrak{gl}_{m|n})$ in \cite{DG},
the multiplication formulas of some special elements in the $q$-Schur algebras
or $q$-Schur superalgebras are the main ingredients for whole procedures. In this section,
 We shall focus on the calculation
of  multiplication formulas for some elements
in the standard basis  {$\{ \phi_{M} \where M \in \MNZ(n,r)\} $},  which will be widely used in Section 5.
More precisely,
 fix {$A \in \MNZ(n,r)$},
 denote {$\ro(A)=\lambda$}, {$\co(A)=\mu$}.
 We will  calculate {$\phi_X \phi_A$} for the following three cases:
\begin{itemize}
	\item {${X} ={(\lambda-E_{h+1, h+1}+ E_{h, h+1}|O)} $}
		 or  {${(\lambda-E_{h+1, h+1}| E_{h, h+1} )}$}, for  {$ 1 \le h \le n-1 $};
	\item {$X = { (\lambda-E_{h,h}|E_{h,h}) }$} or  {${ (\lambda|O) }$}, for  {$ 1 \le h \le n$};
	\item  {$X = (\lambda-E_{h,h}+E_{h+1,h}|O)$} or {$(\lambda-E_{h,h}|E_{h+1,h})$},  for {$ 1 \le h \le n-1 $}.
\end{itemize}
A straightforward computation shows {$\co(X) = \lambda$} for these three cases.
Denote {$\ro(X)=\xi$}.
According to \cite{DW1},  we have
\begin{align}
\phi_X \phi_A=
	\delta_{\co(X),\ro(A)}
	\sum_{\substack{
	M\in\MNZ(n,r) \\
	\ro({M})=\ro(X), \co({M})=\co(A)}
	}f_M\phi_M,
	\quad f_M\in \CC. \label{get_fM}
\end{align}
The purpose of this section is to determine {$f_M$} for each {$M$}
corresponding to the given {$X$} and {$A$}.
For any {$h \in \HCR$},
\begin{align*}
&{\phi}_{X} {\phi}_{A} (x_{\mu}h)
 =  (x_{\xi} {d_X} c_{X}
			\sum _{{\sigma} \in \D_{{\nu}_{_X}} \cap \fS_{\lambda}} {\sigma} )
	({{d_A}} c_{A}  \sum _{{\sigma} \in \D_{{\nu}_{_A}} \cap \fS_{\mu}} {\sigma}) \cdot h,
\end{align*}
so the coefficients {$f_M$} are determined by the following equation
\begin{align}
(x_{\xi} {d_X} c_{X}
			\sum _{{\sigma} \in \D_{{\nu}_{_X}} \cap \fS_{\lambda}} {\sigma} )
	({ {d_A}} c_{A}  \sum _{{\sigma} \in \D_{{\nu}_{_A}} \cap \fS_{\mu}}  {\sigma})
 = \sum_{M \in \MNZ(n,r)} f_M T_M. \label{eq_gB}
\end{align}

The factors of {$c_A$} that affect  the result
are in position {$(h,k)$} and {$(h+1,k)$}.
Referring to  \eqref{c_alpha} and \eqref{c_A}, simplify the expression of
{$c_A$} to be
\begin{align}
	&c_A  = (\clubsuit) \cdot  c_{\gamma} \cdot (\clubsuit\clubsuit)  \label{suitsig}
\end{align}
where the most important part is
{$	c_{\gamma} = c^{\SOE{a}_{h,k}}_{ {\tilde{a}}_{h-1,k}+1, {\tilde{a}}_{h,k}}
		\cdot c^{\SOE{a}_{h+1,k}}_{ {\tilde{a}}_{h,k}+1, {\tilde{a}}_{h+1,k}} $}.

 For the given {$h$}, using the notations in Proposition \ref{prop_shift}, set $a_k=\sum^{k-1}_{i=1}a_{h+1, i}$ and $b_k=\sum^{n}_{i=k+1}a_{h, i}$ for $1\leq k\leq n$.

\subsection{The case for the Upper Triangular Matrix $\ABSUM{X}$}
\ \\
When {$X = (\lambda -E_{h+1, h+1} + E_{h, h+1}| O)$} or {$(\lambda-E_{h+1, h+1} | E_{h, h+1})$},
we have
\begin{align*}
&\ro(X) = \xi = \lambda - \bs{\ep}_{h+1}+ \bs{\ep}_h,\,d_X = 1 \text{ and },\\
	&\nu_{X} = (\lambda_1, \cdots, \lambda_h, 1, \lambda_{h+1}-1 , \lambda_{h+2}, \cdots, \lambda_n).
\end{align*}
Then
\begin{align*}
	 \D_{{\nu}_{X}} \cap \fS_{{\lambda}}   = \{ 1, s_{{\tilde{\lambda}_{h}+1}},
		s_{{\tilde{\lambda}_{h}+1}} s_{{\tilde{\lambda}_{h}+2}},
		\cdots,
		s_{{\tilde{\lambda}_{h}+1}} s_{{\tilde{\lambda}_{h}+2}} \cdots s_{{\tilde{\lambda}_{h+1}-1}}
	\}
\end{align*}
and
\begin{align*}
 \sum _{{\sigma} \in \D_{{\nu}_{_X}} \cap \fS_{\lambda}} {\sigma}
	&= 1 +  s_{\tilde{\lambda}_{h}+1}+
		s_{\tilde{\lambda}_{h}+1} s_{\tilde{\lambda}_{h}+2} +
		\cdots +
		s_{\tilde{\lambda}_{h}+1} s_{\tilde{\lambda}_{h}+2} \cdots s_{\tilde{\lambda}_{h+1}-1} \\
&= \SAIJ({\tilde{\lambda}_{h}+1}, {\tilde{\lambda}_{h+1}-1}) .
\end{align*}
Let
\begin{align*}
z	&=  (x_{\xi} d_X  c_{X}  \sum _{{\sigma} \in \D_{{\nu}_{_X}} \cap \fS_{\lambda}} {\sigma})
	 \cdot 	(  d_A  c_{A}  \sum _{{\sigma} \in \D_{{\nu}_{_A}} \cap \fS_{\mu}} {\sigma}) \\
	&=
		\sum_{j=0}^{{\lambda}_{h+1}-1}
		x_{\xi}  c_{X}
		s_{\tilde{\lambda}_{h}+1} s_{\tilde{\lambda}_{h}+2} \cdots s_{\tilde{\lambda}_{h}+j}
		 d_A  c_{A}
		\sum _{{\sigma} \in \D_{{\nu}_{_A}} \cap \fS_{\mu}} {\sigma}
	= \sum_{k=1}^n  z_k,
\end{align*}
where
{$ s_{\tilde{\lambda}_{h}+1} s_{\tilde{\lambda}_{h}+2} \cdots s_{\tilde{\lambda}_{h}+j} $} is  the identity
when  	{$ j = 0 $}, and
\begin{align}
	z_k&=
		\sum_{j=a_{k}}^{a_{k+1} - 1}
		x_{\xi}  c_{X}
		s_{\tilde{\lambda}_{h}+1} s_{\tilde{\lambda}_{h}+2} \cdots s_{\tilde{\lambda}_{h}+j}
		 d_A c_{A}
		\sum _{{\sigma} \in \D_{{\nu}_{_A}} \cap \fS_{\mu}} {\sigma} . \label{eq_zk}
\end{align}

By calculating $z_k$, one could get the multiplication formulas of
\begin{align*}
	\phi_{(\lambda-E_{h+1, h+1}+ E_{h, h+1}|O)}\phi_{A},
	\qquad \phi_{(\lambda-E_{h+1, h+1}| E_{h, h+1} )}\phi_{A}.
\end{align*}
Because  {$d_{(\SE{A} + E_{h,k}-E_{h+1,k}|\SO{A})} = d_{(\SE{A} |\SO{A} + E_{h,k}-E_{h+1,k})} $}, briefly, we use {$d_{A^+_{h,k}} $} to denote them.
So does {${\nu}_{A^+_{h,k}} $}.
\begin{lem}\label{sc+}
	For  {$A=(a_{i,j}) \in \MNZ(n, r)$}, and {$h\in [1,n-1], k \in [1,n]$},
we can obtain
\begin{equation}\label{form-sc}
\SAIJ( {{\tilde{a}_{h, k}} + 1} , {{\tilde{a}_{h, k}} + a_{h+1, k} - 1})c_A \sum_{\sigma\in\mathcal{D}_{\nu_A}\cap \mathfrak{S}_\mu}\sigma
	=c_A\SDIJ({{\tilde{a}_{h, k}}}, {{\tilde{a}_{h, k}} - a_{h, k} + 1}) \sum_{\sigma\in\mathcal{D}_{\nu_{A^+_{h,k}}}\cap \mathfrak{S}_\mu}\sigma.
\end{equation}
\end{lem}
\begin{proof}
Use the notations in \eqref{suitsig}.
Observing the expression of $\SAIJ( {{\tilde{a}_{h, k}} + 1} , {{\tilde{a}_{h, k}} + a_{h+1, k} - 1})$,
since all those $c_i$'s appearing in
$(\clubsuit)$, $(\clubsuit\clubsuit)$ and {$ c^{\SOE{a}_{h,k}}_{ {\tilde{a}}_{h-1,k}+1, {\tilde{a}}_{h,k}} $}
 are with $i\leq\tilde{a}_{h,k}$ or $i>\tilde{a}_{h+1,k}$,
 they commute with $\SAIJ( {{\tilde{a}_{h, k}} + 1} , {{\tilde{a}_{h, k}} + a_{h+1, k} - 1})$.
Lemma \ref{ComC-S}{\rm (ii)} implies
\begin{align*}
\SAIJ( {{\tilde{a}_{h, k}} + 1} , {{\tilde{a}_{h, k}} + a_{h+1, k} - 1})c^{\SOE{a}_{h+1,k}}_{ {\tilde{a}}_{h,k}+1, {\tilde{a}}_{h+1,k}}
 	=c^{\SOE{a}_{h+1,k}}_{ {\tilde{a}}_{h,k}+1, {\tilde{a}}_{h+1,k}}\SAIJ( {{\tilde{a}_{h, k}} + 1} , {{\tilde{a}_{h, k}} + a_{h+1, k} - 1}),
\end{align*}
which means {$\SAIJ( {{\tilde{a}_{h, k}} + 1} , {{\tilde{a}_{h, k}} + a_{h+1, k} - 1})c_A  = c_A \SAIJ( {{\tilde{a}_{h, k}} + 1} , {{\tilde{a}_{h, k}} + a_{h+1, k} - 1})$}.

By Proposition \ref{prop_splus}, we have
\begin{align*}
\SAIJ( {{\tilde{a}_{h, k}} + 1} , {{\tilde{a}_{h, k}} + a_{h+1, k} - 1})\sum_{\sigma\in\mathcal{D}_{\nu_A}\cap \mathfrak{S}_\mu}\sigma
=\SDIJ({{\tilde{a}_{h, k}}}, {{\tilde{a}_{h, k}} - a_{h, k} + 1})\sum_{\sigma\in\mathcal{D}_{\nu_{A^+_{h,k}}}\cap \mathfrak{S}_\mu}\sigma.
\end{align*}
Therefore the equation \eqref{form-sc} is followed.
\end{proof}

\begin{thm}\label{upper0}
For the given $A \in \MNZ(n,r)$ and $\lambda = \ro(A)$,  we have
\begin{equation}\label{U_0-A}
\begin{aligned}
\phi_{(\lambda+E_{h,h+1}-E_{h+1,h+1}|O) } \phi_A&=
\sum_{1\leq k\leq n,\atop a_{h+1,k}\geq 1}[(\SEE{a}_{h,k}+1)\phi_{(\SE{A} + E_{h,k}-E_{h+1,k}|\SO{A})} \\
& \qquad\qquad  +\phi_{(\SE{A} |\SO{A} + E_{h,k}-E_{h+1,k})}],
\end{aligned}
\end{equation}
where {$(\SE{A} + E_{h,k}-E_{h+1,k}|\SO{A}), (\SE{A} |\SO{A} + E_{h,k}-E_{h+1,k}) \in \MNZ(n,r)$}.
\end{thm}
\begin{proof}

When $X=(\lambda+E_{h,h+1}-E_{h+1,h+1}|O) $,
we have $c_X=1$.
One may prove the result by calculating {$z_k$} in \eqref{eq_zk}.

Applying Proposition \ref{prop_shift} and equation \eqref{s_on_S},
	it follows
\begin{align*}
z_k
&=  x_\xi(s_{\tilde{\lambda}_h}s_{\tilde{\lambda}_h-1}
		\cdots  s_{\tilde{\lambda}_h-b_k+1})
		d_{A^+_{h,k}} \SAIJ( {{\tilde{a}_{h, k}} + 1} , {{\tilde{a}_{h, k}} + a_{h+1, k} - 1}) c_A
		\sum_{\sigma\in\mathcal{D}_{\nu_A}\cap\mathfrak{S}_\mu}  \sigma  \\
&=  x_\xi  d_{A^+_{h,k}} \SAIJ( {{\tilde{a}_{h, k}} + 1} , {{\tilde{a}_{h, k}} + a_{h+1, k} - 1}) c_A
	\sum_{\sigma\in\mathcal{D}_{\nu_A}\cap\mathfrak{S}_\mu}\sigma.
\end{align*}
By {\cite[Corollary 3.5]{DW1}}, we have
\begin{align*}
x_\xi  d_{A^+_{h,k}}
=  (\sum_{\tau\in \mathcal{D}^{-1}_{\nu_{(A^+_{h,k})^T}} \cap \mathfrak{S}_\xi}\tau)
	d_{A^+_{h,k}} 	x_{\nu_{A^+_{h,k}}},
\end{align*}
and according to Lemma \eqref{sc+},
\begin{align*}
z_k
&=  (\sum_{\tau\in \mathcal{D}^{-1}_{\nu_{(A^+_{h,k})^T}} \cap \mathfrak{S}_\xi}\tau)
	d_{A^+_{h,k}} 	x_{\nu_{A^+_{h,k}}}
	c_A\SDIJ({{\tilde{a}_{h, k}}}, {{\tilde{a}_{h, k}} - a_{h, k} + 1})\sum_{\sigma\in\mathcal{D}_{\nu_{A^+_h,k}}\cap\mathfrak{S}_\mu}\sigma .
\end{align*}
So what we need to consider is $c_A\SDIJ({{\tilde{a}_{h, k}}}, {{\tilde{a}_{h, k}} - a_{h, k} + 1})$.
Since every   $s_i$ appears  in $\SDIJ({{\tilde{a}_{h, k}}}, {{\tilde{a}_{h, k}} - a_{h, k} + 1})$ is with $\tilde{a}_{h-1, k}+1\leq i \leq\tilde{a}_{h, k}$,
every $s_j$ appearing in $(\clubsuit) $ and $(\clubsuit\clubsuit)$
	is with $j<\tilde{a}_{h-1}+1$ or $j>\tilde{a}_{h,k}+1$,
	which means they commute  with $\SDIJ({{\tilde{a}_{h, k}}}, {{\tilde{a}_{h, k}} - a_{h, k} + 1})$.

Case 1: $\SOE{a}_{h,k}=0$ and $\SOE{a}_{h+1,k}=0$.  In this case
{$c_{\gamma} = 1$},
$a_{h,k}=\SEE{a}_{h,k}$, and $(\SE{A} |\SO{A} + E_{h,k}-E_{h+1,k})\notin \MNZ(n,r)$, then
$c_{(\SE{A} + E_{h,k}-E_{h+1,k}|\SO{A})}=c_A$ and
then
$$
x_{\nu_{A^+_{h,k}}}c_A
	\SDIJ({{\tilde{a}_{h, k}}}, {{\tilde{a}_{h, k}} - a_{h, k} + 1})
	=(\SEE{a}_{h,k}+1)x_{\nu_{A^+_{h,k}}}c_{(\SE{A} + E_{h,k}-E_{h+1,k}|\SO{A})}.
$$

Case 2: $\SOE{a}_{h,k}=1$ and $\SOE{a}_{h+1,k}=0$.
Then {$c_{\gamma} = c_{(\tilde{a}_{h-1,k}+1,\tilde{a}_{h,k})}$},
 $a_{h,k}=\SEE{a}_{h,k}+1$ and $(\SE{A} |\SO{A} + E_{h,k}-E_{h+1,k})\notin \MNZ(n,r)$.
From Lemma \ref{prop_upper0}{\rm (ii)}, we have
$$x_{\nu_{A^+_{h,k}}}c_{\gamma}\SDIJ({{\tilde{a}_{h, k}}}, {{\tilde{a}_{h, k}} - a_{h, k} + 1})
=(\SEE{a}_{h,k}+1)x_{\nu_{A^+_{h,k}}}c_{(\tilde{a}_{h-1,k}+1,\tilde{a}_{h,k}+1)},$$
which means
$$
x_{\nu_{A^+_{h,k}}}c_A \SDIJ({{\tilde{a}_{h, k}}}, {{\tilde{a}_{h, k}} - a_{h, k} + 1}) =(\SEE{a}_{h,k}+1)x_{\nu_{A^+_{h,k}}}c_{(\SE{A} + E_{h,k}-E_{h+1,k}|\SO{A})},
$$

Case 3: $\SOE{a}_{h,k}=0$ and $\SOE{a}_{h+1,k}=1$. In this situation
{$c_{\gamma} = c_{(\tilde{a}_{h,k}+1,\tilde{a}_{h+1,k})} $},
$a_{h,k}=\SEE{a}_{h,k}$ and $(\SE{A} |\SO{A} + E_{h,k}-E_{h+1,k})\in \MNZ(n,r)$.
 From Lemma \ref{prop_upper0}{\rm (i)}, we have
 $$x_{\nu_{A^+_{h,k}}}c_{\gamma} \SDIJ({{\tilde{a}_{h, k}}}, {{\tilde{a}_{h, k}} - a_{h, k} + 1})
 =x_{\nu_{A^+_{h,k}}}c_{(\tilde{a}_{h-1,k}+1,\tilde{a}_{h,k})}+
 (\SEE{a}_{h,k}+1)x_{\nu_{A^+_{h,k}}}c_{(\tilde{a}_{h,k}+2,\tilde{a}_{h+1,k})},$$
 which means
\begin{align*}
&x_{\nu_{A^+_{h,k}}}c_A \SDIJ({{\tilde{a}_{h, k}}}, {{\tilde{a}_{h, k}} - a_{h, k} + 1}) \\
&=	x_{\nu_{A^+_{h,k}}}c_{(\SE{A} |\SO{A} + E_{h,k}-E_{h+1,k})}
	+ (\SEE{a}_{h,k}+1)x_{\nu_{A^+_{h,k}}}c_{(\SE{A} + E_{h,k}-E_{h+1,k}|\SO{A})}.
\end{align*}

 Case 4: $\SOE{a}_{h,k}=1$ and $\SOE{a}_{h+1,k}=1$.
Then {$c_{\gamma} = c_{(\tilde{a}_{h-1,k}+1,\tilde{a}_{h,k})} c_{(\tilde{a}_{h,k}+1,\tilde{a}_{h+1,k})} $},
   $a_{h,k}=\SEE{a}_{h,k}+1$,
 and $(\SE{A} |\SO{A} + E_{h,k}-E_{h+1,k})\notin \MNZ(n,r)$. Specially, if $a_{h+1,k}=1$, then $ (\SE{A} + E_{h,k}-E_{h+1,k}|\SO{A})\notin \MNZ(n,r)$. According to Lemma \ref{prop_upper0}{\rm (iii)}, we have
\begin{align*}
&x_{\nu_{A^+_{h,k}}}
c_{\gamma}
 \SDIJ({{\tilde{a}_{h, k}}}, {{\tilde{a}_{h, k}} - a_{h, k} + 1}) \\
 &=\left\{
\begin{aligned}
&0, \mbox{ if } a_{h+1,k}=1;\\
&(\SEE{a}_{h,k}+1)x_{\nu_{A^+_{h,k}}}c_{(\tilde{a}_{h-1,k}+1,\tilde{a}_{h,k}+1)}
c_{(\tilde{a}_{h,k}+2,\tilde{a}_{h+1,k})}, \mbox{ if } a_{h+1,k}>1.
\end{aligned}\right.
\end{align*}
which means
{$x_{\nu_{A^+_{h,k}}}c_A \SDIJ({{\tilde{a}_{h, k}}}, {{\tilde{a}_{h, k}} - a_{h, k} + 1}) =(\SEE{a}_{h,k}+1)x_{\nu_{A^+_{h,k}}}c_{(\SE{A} + E_{h,k}-E_{h+1,k}|\SO{A})}$}.

 In summary, the assertion in this proposition is derived.
\end{proof}


\begin{thm}\label{U1A}
For the given $A\in \MNZ(n,r)$  and  $\lambda=\ro(A)$, we have
\begin{equation}\label{U1A'}
\begin{aligned}
\phi_{(\lambda-E_{h+1,h+1}|E_{h,h+1})} \phi_A
	&= \sum_{1 \leq k \leq n, \atop a_{h+1,k} \geq 1}
		[(-1)^{ \SOE{\tilde{a}}_{h+1,k}} (\SEE{a}_{h,k}+1)\phi_{(\SE{A} + E_{h,k}|\SO{A} - E_{h+1,k})}\\
		&\quad\quad
		+(-1)^{ \SOE{\tilde{a}}_{h,k}}\phi_{(\SE{A} - E_{h+1,k}|\SO{A} + E_{h,k})}],
\end{aligned}
\end{equation}
where {$(\SE{A} + E_{h,k}|\SO{A} - E_{h+1,k}), (\SE{A}-E_{h+1,k}|\SO{A} + E_{h,k})\in\MNZ(n,r)$}.
\end{thm}
\begin{proof}
When $X=(\lambda-E_{h+1,h+1}|E_{h,h+1})$,
we have   $c_X=c_{\tilde{\lambda}_h+1}$ and
\begin{align*}
	z_k&=
		\sum_{j=a_{k}}^{a_{k+1} - 1}
		x_{\xi}  c_{\tilde{\lambda}_h+1}
		s_{\tilde{\lambda}_{h}+1} s_{\tilde{\lambda}_{h}+2} \cdots s_{\tilde{\lambda}_{h}+j}
		 d_A c_{A}
		\sum _{{\sigma} \in \D_{{\nu}_{_A}} \cap \fS_{\mu}} {\sigma}  \\
	&=
		\sum_{j=a_{k}}^{a_{k+1} - 1}
		x_{\xi}
		s_{\tilde{\lambda}_{h}+1} s_{\tilde{\lambda}_{h}+2} \cdots s_{\tilde{\lambda}_{h}+j}
		c_{\tilde{\lambda}_h+j + 1}
		 d_A c_{A}
		\sum _{{\sigma} \in \D_{{\nu}_{_A}} \cap \fS_{\mu}} {\sigma} .
\end{align*}

For each {$j$} with {$a_{k} \le j \le a_{k+1} - 1$}, set {$p_j = j - a_k$}.
Proposition \ref{propR3} implies
\begin{align*}
c_{\tilde{\lambda}_h+j + 1}d_A=d_Ac_{\tilde{a}_{h,k}+p_j+1}.
\end{align*}
By  Proposition \ref{prop_shift},
\begin{align*}
x_{\xi} s_{\tilde{\lambda}_{h}+1} s_{\tilde{\lambda}_{h}+2} \cdots s_{\tilde{\lambda}_{h}+j} d_A
&=
 x_{\xi} s_{\tilde{\lambda}_{h}} s_{\tilde{\lambda}_{h}-1}\cdots s_{\tilde{\lambda}_{h}-b_k+1}
		 {d_{A^+_{h,k}}}
  	s_{\tilde{a}_{h,k}+1} \cdots s_{\tilde{a}_{h,k}+p_j} \\
&= x_{\xi} 		 {d_{A^+_{h,k}}}
  	s_{\tilde{a}_{h,k}+1} \cdots s_{\tilde{a}_{h,k}+p_j}.
\end{align*}
 Since $s_{\tilde{a}_{h,k}+1} \cdots s_{\tilde{a}_{h,k}+p_j}c_{\tilde{a}_{h,k}+p_j+1}=c_{\tilde{a}_{h,k}+1}
  	s_{\tilde{a}_{h,k}+1} \cdots s_{\tilde{a}_{h,k}+p_j}$, by Lemma \ref{sc+},
\begin{align*}
z_k	&=
		x_\xi d_{A^+_{h,k}}c_{\tilde{a}_{h,k}+1}c_A \SDIJ({{\tilde{a}_{h, k}}}, {{\tilde{a}_{h, k}} - a_{h, k} + 1})
		\sum_{\sigma\in\mathcal{D}_{\nu_{A^+_{h,k}}} \cap \mathfrak{S}_\mu} \sigma.
\end{align*}

By {\cite[Corollary 3.5]{DW1}}, we have
\begin{align*}
z_k &=
		 (\sum_{\tau\in \mathcal{D}^{-1}_{\nu_{(A^+_{h,k})^T}} \cap \mathfrak{S}_\xi}\tau)
	d_{A^+_{h,k}} 	x_{\nu_{A^+_{h,k}}} c_A \SDIJ({{\tilde{a}_{h, k}}}, {{\tilde{a}_{h, k}} - a_{h, k} + 1})
		\sum_{\sigma\in\mathcal{D}_{\nu_{A^+_{h,k}}} \cap \mathfrak{S}_\mu} \sigma \\
&=
		(-1)^{ \SOE{\tilde{a}}_{h-1,k} ( \SOE{a}_{h,k} + \SOE{a}_{h+1,k})}
		 (\sum_{\tau\in \mathcal{D}^{-1}_{\nu_{A^+_{h,k}}} \cap \mathfrak{S}_\xi}\tau)
	d_{A^+_{h,k}} 	 \\
	&   \qquad  \qquad \cdot 	
	x_{\nu_{A^+_{h,k}}}   c_{\tilde{a}_{h,k}+1}  c_{\gamma}  \SDIJ({{\tilde{a}_{h, k}}}, {{\tilde{a}_{h, k}} - a_{h, k} + 1})
	\cdot  (\clubsuit) (\clubsuit \clubsuit)
		\sum_{\sigma\in\mathcal{D}_{\nu_{A^+_{h,k}}} \cap \mathfrak{S}_\mu} \sigma .
\end{align*}

The expression
{$ x_{\nu_{{A}^+_{h,k}}} c_{\tilde{a}_{h,k}+1}c_{\gamma}  \SDIJ({{\tilde{a}_{h, k}}}, {{\tilde{a}_{h, k}} - a_{h, k} + 1})$}
 can be discussed in four cases for different values of {${\SOE{a}_{h,k}}$} and {${\SOE{a}_{h+1,k}}$}.

Case 1: $\SOE{a}_{h,k}=0$, $\SOE{a}_{h+1,k}=0$.
Then  $(\SE{A} + E_{h,k}|\SO{A} - E_{h+1,k})\notin \MNZ(n,r)$, {$c_{\gamma} = 1$}.
Lemma \ref{propupper1}{\rm (ii)} shows
\begin{align*}
 x_{\nu_{A^+_{h,k}}} c_{\tilde{a}_{h,k}+1}
	c_{\gamma}  \SDIJ({{\tilde{a}_{h, k}}}, {{\tilde{a}_{h, k}} - a_{h, k} + 1})
=x_{\nu_{A^+_{h,k}}}c_{(\tilde{a}_{h-1,k}+1,\tilde{a}_{h,k}+1)}.
\end{align*}

Case 2:  $\SOE{a}_{h,k}=1$ and $\SOE{a}_{h+1,k}=0$.
 In this case,   {$c_{\gamma} = c_{(\tilde{a}_{h-1,k}+1,\tilde{a}_{h,k})}$},
  $(\SE{A} + E_{h,k}|\SO{A} - E_{h+1,k})\notin \MNZ(n,r)$ and $(\SE{A} - E_{h+1,k}|\SO{A} + E_{h,k})\notin \MNZ(n,r)$.
Lemma \ref{propupper1}{\rm (iv)} shows
\begin{align*}
& x_{\nu_{A^+_{h,k}}} c_{\tilde{a}_{h,k}+1} c_{\gamma}   \SDIJ({{\tilde{a}_{h, k}}}, {{\tilde{a}_{h, k}} - a_{h, k} + 1})
=0.
\end{align*}

Case 3: $\SOE{a}_{h,k}=0$ and $\SOE{a}_{h+1,k}=1$.
In this case,  {$c_{\gamma} = c_{(\tilde{a}_{h,k}+1,\tilde{a}_{h+1,k})} $},
$(\SE{A} - E_{h+1,k}|\SO{A} + E_{h,k})\notin \MNZ(n,r)$ when $a_{h+1,k}=1$.
Applying Lemma \ref{propupper1}{\rm (iii)}, we get
\begin{align*}
&x_{\nu_{A^+_{h,k}}}  c_{\tilde{a}_{h,k}+1}  c_{\gamma}  \SDIJ({{\tilde{a}_{h, k}}}, {{\tilde{a}_{h, k}} - a_{h, k} + 1}) \\
&=\left\{\begin{aligned}
&-(\SEE{a}_{h,k}+1)x_{\nu_{A^+_{h,k}}}, \mbox{ if } a_{h+1,k}=1;\\
&-(\SEE{a}_{h,k}+1)x_{\nu_{A^+_{h,k}}} + x_{\nu_{A^+_{h,k}}}
c_{(\tilde{a}_{h-1,k}+1,\tilde{a}_{h,k}+1)}c_{(\tilde{a}_{h,k}+2,\tilde{a}_{h+1,k})}, \mbox{ otherwise.}
\end{aligned}\right.
\end{align*}

Case 4: $\SOE{a}_{h,k}=1$,  $\SOE{a}_{h+1,k}=1$.
Then
$(\SE{A} - E_{h+1,k}|\SO{A} + E_{h,k})\notin \MNZ(n,r)$
and {$
c_{\gamma} = c_{(\tilde{a}_{h-1,k}+1,\tilde{a}_{h,k})}
	c_{(\tilde{a}_{h,k}+1,\tilde{a}_{h+1,k})}.
$}
Applying Lemma \ref{propupper1} {\rm (v)}, we can derive that
\begin{align*}
&x_{\nu_{A^+_{h,k}}} c_{\tilde{a}_{h,k}+1}
	c_{\gamma} \SDIJ({{\tilde{a}_{h, k}}}, {{\tilde{a}_{h, k}} - a_{h, k} + 1})
 =(\SEE{a}_{h,k}+1)x_{\nu_{A^+_{h,k}}}c_{(\tilde{a}_{h-1,k}+1,\tilde{a}_{h,k}+1)}.
\end{align*}
Summarize the four cases above, then the assertion is  proved.
\end{proof}

\subsection{The Case for the Diagonal Matrix $\ABSUM{X}$}
\ \\
In this section, we will be focused on $\phi_X\phi_A$
when {$X = { (\lambda-E_{h,h}|E_{h,h}) }$} or  {${ (\lambda|O) }$}.
Here {$\lambda = \ro(X) = \co(X) = \ro(A) $}, $\mu=\co(A)$.

\begin{thm}\label{formD1}
For the given $A\in \MNZ(n,r)$  and  $\lambda=\ro(A)$,
 the following formula is satisfied in $\Qnr$
\begin{align*}
\phi_{ (\lambda-E_{h,h}|E_{h,h}) }\phi_A
=&\sum_{1\leq k\leq n,\atop a_{h,k}\geq 1}(-1)^{ \SOE{\tilde{a}}_{h,k}}
	[ \phi_{(\SE{A} - E_{h,k}|\SO{A} + E_{h,k})}
	+  {a}_{h,k} \phi_{(\SE{A} + E_{h,k}|\SO{A} - E_{h,k})}],
\end{align*}
where {$M= (\SE{A} - E_{h,k}|\SO{A} + E_{h,k}), (\SE{A} + E_{h,k}|\SO{A} - E_{h,k})\in\MNZ(n,r)$}.
\end{thm}
\begin{proof}
Assume {${\lambda}_h > 0$}. When {$X = (\lambda-E_{h,h}| E_{h,h})$},
we have {${d_X} = 1$},
 $c_{\gamma}=c_{(\tilde{\lambda}_{h-1}+1,\tilde{\lambda}_h)}$,
{$ \D_{{\nu}_{_X}} \cap \fS_{\lambda} = \{ 1 \}$}.
Let
\begin{displaymath}
\begin{aligned}
	z&= x_{\xi} c_{(\tilde{\lambda}_{h-1}+1,\tilde{\lambda}_h)}{d_A} c_{A}
		(\sum _{{\sigma} \in \D_{{\nu}_{_A}} \cap \fS_{\mu}} {\sigma})\\
&=\sum^{n}_{k=1}\sum^{a_{h,k}-1}_{p=0}x_{\xi}c_{\tilde{\lambda}_{h-1}+a_k+p+1}{d_A} c_{A}(\sum _{{\sigma} \in \D_{{\nu}_{_A}} \cap \fS_{\mu}} {\sigma}).
\end{aligned}
\end{displaymath}
From Proposition \ref{propR3}, we have $c_{\tilde{\lambda}_{h-1}+a_k+p+1}{d_A}={d_A}
		c_{\tilde{a}_{h,k}+p+1}$. Then
\begin{align*}
z= \sum^{n}_{k=1}
		x_{\xi} {d_A}
		c_{(\tilde{a}_{h-1,k}+1,\tilde{a}_{h,k})} c_A
		 \sum _{{\sigma} \in \D_{{\nu}_{_A}} \cap \fS_{\mu}} {\sigma} .
\end{align*}

It is sufficient to calculate {$c_{(\tilde{a}_{h-1,k}+1,\tilde{a}_{h,k})}c_A$} to get the result.

Case 1: $\SOE{a}_{h,k}=0$.
 It forces $(\SE{A} + E_{h,k}|\SO{A} - E_{h,k})\notin \MNZ(n,r)$,
 then
\begin{align*}
c_{(\tilde{a}_{h-1,k}+1,\tilde{a}_{h,k})}c_A=(-1)^{\tilde{a}^1_{h,k}}c_{(\SE{A} - E_{h,k}|\SO{A} + E_{h,k})}.
\end{align*}

Case 2: $\SOE{a}_{h,k}=1$.  It forces $a_{h,k}=\SEE{a}_{h,k}+1$ and $(\SE{A} - E_{h,k}|\SO{A} + E_{h,k})\notin \MNZ(n,r)$, then
\begin{align*}
&c_{(\tilde{a}_{h-1,k}+1,\tilde{a}_{h,k})}c_A
=(-1)^{\tilde{a}^1_{h,k}} {a}_{h,k} c_{(\SE{A} + E_{h,k}|\SO{A} - E_{h,k})}.
\end{align*}

As a consequence,  the result is proved.
\end{proof}
\begin{thm}\label{formD0}
	For the given $A\in \MNZ(n,r)$  and  $\lambda=\ro(A)$, $\mu=\ro(A)$, we have
\begin{align*}
\phi_{ (\lambda|O) }\phi_A= \phi_A  \phi_{ (\mu|O) } = \phi_A.
\end{align*}
\end{thm}
\begin{proof}
For $ X = (\lambda|O) $, we have
$
	d_{ X} = 1,
	c_{ X} = 1,
	\mathcal{D}_{{\nu}_{X }} \cap \fS_{\lambda} = \{1\},
$
so
\begin{align*}
	 x_{\lambda}{d_A} c_{A}\sum _{{\sigma} \in \mathcal{D}_{{\nu}_{A}} \cap \fS_{\mu}}{\sigma}=T_{A},
\end{align*}
which means $\phi_{ (\lambda|O) }\phi_A=\phi_A.$
One may prove $\phi_A \phi_{ (\mu|O) }=\phi_A $ in the same way.
\end{proof}

\subsection{The case for the Lower Triangular Matrix $\ABSUM{X}$}
\ \\
When {$X = (\lambda-E_{h,h}+E_{h+1,h}|O)$} or {$(\lambda-E_{h,h}|E_{h+1,h})$},
notice that
the matrices {$(\lambda-E_{h,h}+E_{h+1,h})$} and {$(\lambda-E_{h+1,h+1}+E_{h,h+1})$}
are transposes to each other.
The product of $\phi_X\phi_A$ can be studied in a symmetric way to the case in Section 4.1.


According to \cite{DG},  we have $d_X=1$ and
\begin{align*}
\mathcal{D}_{\nu_X}\cap \mathfrak{S}_\lambda
= \{ 1,  s_{\tilde{\lambda}_h - 1}, s_{\tilde{\lambda}_h - 1}s_{\tilde{\lambda}_h-2},
	\cdots,
	s_{\tilde{\lambda}_h-1}s_{\tilde{\lambda}_h-2}\cdots s_{\tilde{\lambda}_{h-1}+1}\}.
\end{align*}

Similar to the upper triangular case,  let
\begin{align*}
z&= (x_{\xi} {d_X} c_X \sum _{{\sigma} \in \D_{{\nu}_{_X}} \cap \fS_{\lambda}} {\sigma})
	 \cdot
	( {{d_A}} c_{A}  \sum _{{\sigma} \in \D_{{\nu}_{_A}} \cap \fS_{\mu}} {\sigma}) \\
&=
		\sum_{j=0}^{\lambda_h -1}
		x_{\xi}  c_{X}
		s_{\tilde{\lambda}_{h}-1} s_{\tilde{\lambda}_{h}-2} \cdots s_{\tilde{\lambda}_{h}-j}
		{{d_A}} c_{A}
		\sum _{{\sigma} \in \D_{{\nu}_{_A}} \cap \fS_{\mu}} {\sigma} = \sum_{k=1}^n z_k,
\end{align*}
where  {$s_{\tilde{\lambda}_{h}-1} s_{\tilde{\lambda}_{h}-2} \cdots s_{\tilde{\lambda}_{h}-j} = 1$} when {$j = 0$},
and
\begin{align*}
z_k&=
		\sum_{j=b_{k}}^{b_{k-1} - 1}
		x_{\xi}  c_{X}
		s_{\tilde{\lambda}_{h}-1} s_{\tilde{\lambda}_{h}-2} \cdots s_{\tilde{\lambda}_{h}-j}
		{{d_A}} c_{A}
		\sum _{{\sigma} \in \D_{{\nu}_{_A}} \cap \fS_{\mu}} {\sigma}.
\end{align*}

As
 {$d_{(\SE{A} - E_{h,k}+E_{h+1,k}|\SO{A})} =d_{(\SE{A} |\SO{A} - E_{h,k}+E_{h+1,k})}$}
, we use {$d_{A^-_{h,k}}$} to denote them.
 So does the notation {$\nu_{A^-_{h,k}}$}.


\begin{lem}\label{sc-}
For the given {$A$}, we have
\begin{align*}
	 \SDIJ({{\tilde{a}_{h,k}} - 1}, {{\tilde{a}_{h,k}} - a_{h,k} + 1})  c_A
	\sum_{\sigma\in\mathcal{D}_{\nu_A}\cap\mathfrak{S}_\mu}\sigma=
	c_A  \SAIJ({{\tilde{a}_{h,k}}}, {{\tilde{a}_{h,k}} + a_{h+1,k} - 1})
	\sum_{\sigma\in\mathcal{D}_{\nu_{A^-_{h,k}}}\cap\mathfrak{S}_\mu}\sigma.
\end{align*}
\end{lem}
\begin{proof}
The proof is analogous to Lemma \ref{sc+} by applying Corollary \ref{prop3R5}
and relations in {$\SerA$}.
\end{proof}

\begin{thm}\label{L0A}
For the given $A\in \MNZ(n,r)$,  set  $\lambda=\ro(A)$.
 The following formula is satisfied
\begin{align*}
\phi_{(\lambda-E_{h,h}+E_{h+1,h}|O)}\phi_A
=\sum_{1\leq k\leq n,\atop  a_{h,k}\geq 1}
	&[		(\SEE{a}_{h+1,k}+1) \phi_{ (\SE{A} - E_{h,k}+E_{h+1,k} | \SO{A})} \\
	&  +\phi_{(\SE{A} |\SO{A} - E_{h,k}+E_{h+1,k})} ],
\end{align*}
where {$M= (\SE{A} - E_{h,k}+E_{h+1,k}|\SO{A}), (\SE{A} |\SO{A} - E_{h,k}+E_{h+1,k})\in\MNZ(n,r)$}.
\end{thm}
\begin{proof}
When $X=(\lambda-E_{h,h}+E_{h+1,h}|O)$, we have $c_X=1$, and
\begin{align*}
z_k&=
		\sum_{j=b_{k}}^{b_{k-1} - 1}
		x_{\xi}
		s_{\tilde{\lambda}_{h}-1} s_{\tilde{\lambda}_{h}-2} \cdots s_{\tilde{\lambda}_{h}-j}
		{{d_A}} c_{A}
		\sum _{{\sigma} \in \D_{{\nu}_{_A}} \cap \fS_{\mu}} {\sigma}.
\end{align*}

Referring to \cite{DG}, and applying  Lemma \ref{sc-}, we have
\begin{align*}
z_k&=
	x_\xi d_{A^-_{h,k}}  \SDIJ({{\tilde{a}_{h,k}} - 1}, {{\tilde{a}_{h,k}} - a_{h,k} + 1})  c_A
	\sum_{\sigma\in\mathcal{D}_{\nu_A}\cap\mathfrak{S}_\mu}\sigma \\
&=
	x_\xi d_{A^-_{h,k}} c_A  \SAIJ({{\tilde{a}_{h,k}}}, {{\tilde{a}_{h,k}} + a_{h+1,k} - 1})
	\sum_{\sigma\in\mathcal{D}_{\nu_{A^-_{h,k}}}\cap\mathfrak{S}_\mu}\sigma.
\end{align*}

Because every $s_i$ in $ \SAIJ({{\tilde{a}_{h,k}}}, {{\tilde{a}_{h,k}} + a_{h+1,k} - 1}) $ is with $\tilde{a}_{h,k}\leq i<\tilde{a}_{h+1,k}$, which means $ \SAIJ({{\tilde{a}_{h,k}}}, {{\tilde{a}_{h,k}} + a_{h+1,k} - 1}) $ commutes with $(\clubsuit)$ and $(\clubsuit\clubsuit)$,
it yields
\begin{align*}
	z_k&=(-1)^{ \SOE{\tilde{a}}_{h-1,k}(\SOE{a}_{h,k} + \SOE{a}_{h+1,k})}
	x_\xi d_{A^-_{h,k}} c_{\gamma}  \SAIJ({{\tilde{a}_{h,k}}}, {{\tilde{a}_{h,k}} + a_{h+1,k} - 1})
	(\clubsuit) (\clubsuit\clubsuit)
	\sum_{\sigma\in\mathcal{D}_{\nu_{A^-_{h,k}}}\cap\mathfrak{S}_\mu}\sigma.
\end{align*}

By {\cite[Corollary 3.5]{DW1}}, we have
\begin{align*}
x_\xi d_{A^-_{h,k}}=
	(\sum_{\sigma\in\mathcal{D}^{-1}_{\nu_{(A^-_{h,k})^{T}}}\cap \mathfrak{S}_\xi}\sigma)d_{A^-_{h,k}}x_{\nu_{A^-_{h,k}}}.
\end{align*}

In order to figure out $z_k$, similar to the proof of Theorem \ref{upper0},
we apply  Lemma \ref{prop3F10.0} to study the expression
\begin{align*}
x_{\nu_{A^-_{h,k}}} 	c_{\gamma}
		 \SAIJ({{\tilde{a}_{h,k}}}, {{\tilde{a}_{h,k}} + a_{h+1,k} - 1})
\end{align*}
in all cases of $\SOE{a}_{h,k}$ and $\SOE{a}_{h+1,k}$. Then  the following results can be achieved.

Case 1: $\SOE{a}_{h,k}=0$ and $\SOE{a}_{h+1,k}=0$. In this case
\begin{align*}
z_k=(\SEE{a}_{h+1}+1)T_{(\SE{A} - E_{h,k}+E_{h+1,k}|\SO{A})}.
\end{align*}

Case 2: $\SOE{a}_{h,k}=1$ and $\SOE{a}_{h+1,k}=0$. In this case
\begin{align*}
z_k
=
\left\{\begin{aligned}
& T_{(\SE{A} |\SO{A} - E_{h,k}+E_{h+1,k})}, \mbox{ if } a_{h,k}=1;\\
& T_{(\SE{A} |\SO{A} - E_{h,k}+E_{h+1,k})}+(\SEE{a}_{h+1,k}+1)T_{(\SE{A} - E_{h,k}+E_{h+1,k}|\SO{A})}, \mbox{ if } a_{h,k}>1.
\end{aligned}\right.
\end{align*}

Case 3: $\SOE{a}_{h,k}=0$ and $\SOE{a}_{h+1,k}=1$. It follows
\begin{align*}
z_k
=(\SEE{a}_{h+1,k}+1)T_{(\SE{A} - E_{h,k}+E_{h+1,k}|\SO{A})}.
\end{align*}

Case 4:   $\SOE{a}_{h,k}=1$ and $\SOE{a}_{h+1,k}=1$. It follows
\begin{align*}
z_k
=\left\{\begin{aligned}
&0, &\mbox{ if }a_{h,k}=1;\\
&(\SEE{a}_{h+1,k}+1)T_{(\SE{A} - E_{h,k}+E_{h+1,k}|\SO{A})}, &\mbox{ if }a_{h,k}>1.
\end{aligned}
\right.
\end{align*}

Summarizing the four cases above, one concludes the result.
\end{proof}

\begin{thm}\label{L1A}
For the given $A\in \MNZ(n,r)$  and  $\lambda=\ro(A)$,
 the formula is satisfied in $\Qnr$
\begin{align*}
\phi_{(\lambda-E_{h,h}|E_{h+1,h})}\phi_A
&= \sum_{ 1 \leq  k  \leq  n, \atop a_{h,k} \geq 1}
		[
		(-1)^{\SOE{\tilde{a}}_{h,k}}
		(\SEE{a}_{h+1,k}+1) \phi_{(\SE{A} + E_{h+1,k}|\SO{A} - E_{h,k})} \\
	& \qquad 	+
	(-1)^{\SOE{\tilde{a}}_{h,k}}
	\phi_{(\SE{A} - E_{h,k}|\SO{A} + E_{h+1,k})}],
\end{align*}
where {$M=  (\SE{A} + E_{h+1,k}|\SO{A} - E_{h,k}), (\SE{A} - E_{h,k}|\SO{A} + E_{h+1,k})\in\MNZ(n,r)$}.
\end{thm}

\begin{proof}
When $X=(\lambda-E_{h,h}|E_{h+1,h})$,  we have $c_X=c_{\tilde{\lambda}_h}$
and
\begin{align*}
z_k&=
		\sum_{j=b_{k}}^{b_{k-1} - 1}
		x_{\xi}  c_{{\lambda}_h}
		s_{\tilde{\lambda}_{h}-1} s_{\tilde{\lambda}_{h}-2} \cdots s_{\tilde{\lambda}_{h}-j}
		{{d_A}} c_{A}
		\sum _{{\sigma} \in \D_{{\nu}_{_A}} \cap \fS_{\mu}} {\sigma}.
\end{align*}
According to the relations of $\SerA$, $$c_{{\lambda}_h}
		s_{\tilde{\lambda}_{h}-1} s_{\tilde{\lambda}_{h}-2} \cdots s_{\tilde{\lambda}_{h}-j}=
		s_{\tilde{\lambda}_{h}-1} s_{\tilde{\lambda}_{h}-2} \cdots s_{\tilde{\lambda}_{h}-j}c_{\tilde{\lambda}_h-j}.$$

Let $q_j = j - b_k$. 
Applying  Proposition \ref{propR3} and \ref{prop_splus}, it follows
\begin{align*}
 c_{\tilde{\lambda}_h-j}{{d_A}} &={{d_A}}  c_{\tilde{a}_{h,k} - q_j},\\
 s_{\tilde{\lambda}_{h}-1} s_{\tilde{\lambda}_{h}-2} \cdots s_{\tilde{\lambda}_{h}-j}
		{{d_A}}&=s_{\tilde{\lambda}_h}\cdots s_{\tilde{\lambda}_h+a_k-1}
		d_{A^-_{h,k}}
		s_{\tilde{a}_{h,k}-1}\cdots s_{\tilde{a}_{h,k}-q_j}.
 \end{align*}
Since $s_{\tilde{\lambda}_h}, \cdots, s_{\tilde{\lambda}_h+a_k-1}\in \mathfrak{S}_\xi$, we have $x_\xi s_{\tilde{\lambda}_h}\cdots s_{\tilde{\lambda}_h+a_k-1}=x_\xi$. Therefore
\begin{align*}
z_k=x_\xi
		d_{A^-_{h,k}} c_{\tilde{a}_{h,k}}
		 \SDIJ({{\tilde{a}_{h,k}} - 1}, {{\tilde{a}_{h,k}} - a_{h,k} + 1})
		c _A
		\sum_{\sigma\in\mathcal{D}_{\nu_A}\cap\mathfrak{S}_\mu}\sigma,
\end{align*}

Referring to \cite{DG}, and applying Corollary \ref{prop3R5}, we have
\begin{align*}
z_k
&=
	x_\xi d_{A^-_{h,k}}  c_{\tilde{a}_{h,k}} c_A  \SAIJ({{\tilde{a}_{h,k}}}, {{\tilde{a}_{h,k}} + a_{h+1,k} - 1})
	\sum_{\sigma\in\mathcal{D}_{\nu_{A^-_{h,k}}}\cap\mathfrak{S}_\mu}\sigma.
\end{align*}

For every $s_j$ in $ \SAIJ({{\tilde{a}_{h,k}}}, {{\tilde{a}_{h,k}} + a_{h+1,k} - 1}) $, $\tilde{a}_{h,k}\leq j<\tilde{a}_{h+1,k}$, so
 $ \SAIJ({{\tilde{a}_{h,k}}}, {{\tilde{a}_{h,k}} + a_{h+1,k} - 1}) $ commutes with $(\clubsuit)$ and $(\clubsuit\clubsuit)$,
then
\begin{align*}
z_k&=
	(-1)^{ \SOE{\tilde{a}}_{h,k} ( \SOE{a}_{h,k} + \SOE{a}_{h+1,k})}
	x_\xi d_{A^-_{h,k}}c_{\tilde{a}_{h,k}}
	 c_{\gamma}  \SAIJ({{\tilde{a}_{h,k}}}, {{\tilde{a}_{h,k}} + a_{h+1,k} - 1})
	(\clubsuit) (\clubsuit\clubsuit)
	\sum_{\sigma\in\mathcal{D}_{\nu_{A^-_{h,k}}}\cap\mathfrak{S}_\mu}\sigma.
\end{align*}

By {\cite[Corollary 3.5]{DW1}}, we have
\begin{align*}
z_k&=
	(-1)^{ \SOE{\tilde{a}}_{h,k} ( \SOE{a}_{h,k} + \SOE{a}_{h+1,k})}
	(\sum_{\sigma\in\mathcal{D}^{-1}_{\nu_{(A^-_{h,k})^T}}\cap \mathfrak{S}_\xi}\sigma)
	d_{A^-_{h,k}} \\
	& \qquad \qquad \cdot
	x_{\nu_{A^-_{h,k}}}
	c_{\tilde{a}_{h,k}}
	c_{\gamma}  \SAIJ({{\tilde{a}_{h,k}}}, {{\tilde{a}_{h,k}} + a_{h+1,k} - 1})
	\quad \cdot \quad
	 (\clubsuit) (\clubsuit\clubsuit)
	\sum_{\sigma\in\mathcal{D}_{\nu_{A^-_{h,k}}}\cap\mathfrak{S}_\mu}\sigma.
\end{align*}

Similar to  Theorem \ref{L0A}, it is enough to calculate
\begin{align*}
 x_{\nu_{A^-_{h,k}}}
	c_{\tilde{a}_{h,k}}c_{\gamma}
	 \SAIJ({{\tilde{a}_{h,k}}}, {{\tilde{a}_{h,k}} + a_{h+1,k} - 1}) .
\end{align*}

Applying Lemma \ref{prop3F10}, the following assertions can be obtained.

Case 1: if $\SOE{a}_{h,k}=0$ and $\SOE{a}_{h+1,k}=0$, then
$
z_k =(-1)^{\tilde{a}_{h,k}}T_{\SE{A} - E_{h,k}|\SO{A} + E_{h+1,k}}.
$

Case 2: if $\SOE{a}_{h,k}=1$ and $\SOE{a}_{h+1,k}=0$, then
\begin{align*}
z_k
=(-1)^{\SOE{\tilde{a}}_{h,k}}[(\SEE{a}_{h+1,k}+1)T_{(\SE{A} + E_{h+1,k}|\SO{A} - E_{h,k})} +T_{(\SE{A} - E_{h,k}|\SO{A} + E_{h+1,k})}]
\end{align*}

Case 3: if $\SOE{a}_{h,k}=0$ and $\SOE{a}_{h+1,k}=1$, then
 {$ z_k = 0 $}.

Case 4: if $\SOE{a}_{h,k}=1$ and $\SOE{a}_{h+1,k}=1$, then
{$
z_k
=(-1)^{\SOE{\tilde{a}}_{h,k}}T_{(\SE{A} + E_{h+1,k}|\SO{A} - E_{h,k})}.
$}

In summary, the result is then proved.
\end{proof}

\section{Uniform Spanning Sets and Multiplications for   {$\Qnr$} }\label{sec_spanning}
In this section, we will construct the consistent spanning set {$\LS_r$} of {$\Qnr$} for all $r\geq 0$,
and derive the multiplication formulas for some elements in this set for in order to find the generators for the realization of {$\Uqn$}.


Let
{$
\MNZNS(n) =\{ A=(\SUP{A}) \in \MNZN(n)  \where \SEE{a}_{i,i} = 0 \mbox{, for all } i\}
$}.
	
For any  {$\bs{j} \in \NN^n$}, denote
{$\snorm{\bs{j}} = \sum_{k} j_{k}$}.
If {$\lambda \in \CMN(n,t) $} for some {$ t \in \NN^+$},
set {$ {\lambda}^{\bs{j}} = \prod_{i=1}^n {{\lambda}_i}^{{j}_i}$}.

Let {$\{ \bs{\ep}_{i} \where 1 \le i \le n\}$} be the standard basis of {$\CC^n$}.

For $A \in \MN(n)$,
denote $A+\lambda : =  {A} + \diag{(\lambda)}$.

\begin{defn}\label{A[j,r]}
	For any given {$A=(\SUP{A}) \in \MNZNS(n), \bs{j} \in \NN^n $}, define
\begin{displaymath}
	\AJRS(A,\bs{j},r) =
\left\{
\begin{aligned}
	& \sum_{\substack{\lambda \in \CMN(n, r-\snorm{A})} }
	\lambda ^{\bs{j}} \phi_{( \MNZADD(A, \lambda) )},
		\ &\mbox{if } \snorm{A} \le r; \\
	&0, &\mbox{otherwise};
\end{aligned}
\right.
\end{displaymath}
\end{defn}

 Recall the basis {$\B_r$} of {$\Qnr$},
\begin{align*}
\B_r = \{ \phi_{A} \where  (\SE{A}|\SO{A}) \in \MNZ(n,r)\} .
\end{align*}

\begin{lem}Let
\begin{align*}
	\LS_r = \{ \AJRS(A, \bs{j}, r)
			\where  A \in \MNZNS(n), \bs{j} \in \NN^n
	\}.
\end{align*}
Then $\LS_r$ spans $\Qnr$ for all $r\geq0$, i.e., {$\Qnr = \mbox{span}_{\CC}\LS_r$}.
\end{lem}
\begin{proof} Choose the subset $\LS'_r$ of {$\LS_r $}, where
\begin{align*}
	\LS'_r = \{ \AJRS(A, \bs{j}, r)
			\where  A \in \MNZNS(n), \bs{j} \in \NN^n, \bs{j}_n = 0,
			\snorm{A} + \snorm{\bs{j}} \le r
	\}.
\end{align*}
For any given {$A \in \MNZNS(n)$},  assume {$r \ge \snorm{A}$}. Set
\begin{align*}
	&\LS'_{r,A} = \{ \AJRS(A, \bs{k}, r) \in \LS_r \where
	\bs{k} \in \NN^n,    \bs{k}_n = 0,
	 \snorm{A} + \snorm{\bs{k}} \le r  \}
	\subset \LS'_r, \\
	&\B_{r,A} = \{ \phi_{(\MNZADD(A, \lambda))} \in \B_r \where  \lambda \in \CMN(n, r-\snorm{A}) \},
\end{align*}
	then the transition matrix from {$\B_{r,A}$} to {$\LS'_{r,A}$}
	is {$D = (d_{i,j})$},
	with {$d_{i,j}$} of the form  {${\lambda}^{\bs{k}}$},
	for each
	{$\lambda \in  \CMN(n, r-\snorm{A})$},
	and {$\bs{k} \in  \NN^n$} with {$\bs{k}_n = 0, \snorm{\bs{k}}+ \snorm{A} \le r$}.
	According to \cite[Corollary 4.2]{F}, $D$ is invertible,
	hence {$\LS'_{r,A}$} is a linearly independent set,
	furthermore {$\LS'_r$} forms a basis of {$\Qnr$}.
Hence, for any $r\geq0$,
$\mbox{span}_{\CC} \LS_r
= \Qnr$.
\end{proof}
The uniform definition of the spanning sets {$\LS_r $} for $\Qnr$  for all $r\geq0$ is the key for achieving the new realization for $\Uqn$.
We now aim to derive some uniform multiplication formulas for the spanning set. We need the following lemma as a preparation.

\begin{lem}\label{mulformajrb1.5}
\begin{align*}
{\rm (i)}\quad  &\sum_{\substack{ \lambda \in \CMN(n,r-\snorm{A}-1) } }
	{\lambda} ^{\bs{j}} (\lambda_h + 1)
	\phi_{(\SE{A} + {\lambda} + E_{h,h} |\SO{A} )} \\
 &\qquad =  \sum_{k=0}^{j_h} {\binom{j_h}{k}}
	{(-1)}^{k}
	\ABJRS( \SE{A}, \SO{A}, \bs{ j }+(1- k) \bs{\ep}_h, r ); \\
{\rm (ii)}\quad  &\sum_{\substack{ \lambda \in \CMN(n,r-\snorm{A}+1)\\\lambda_h\geq 1 } }
	{\lambda} ^{\bs{j}}
	\phi_{(\SE{A} + {\lambda} - E_{h,h} |\SO{A}) }
=\sum_{k=0}^{j_h} {\binom{j_h}{k}}
	\ABJRS( \SE{A}, \SO{A}, \bs{ j }- k \bs{\ep}_h, r ).
\end{align*}
\end{lem}
\begin{proof}
First of all, we prove   {\rm (i)}.
\begin{align*}
\mbox{LHS of {\rm (i)}}
&=
	\sum_{\substack{ \lambda \in \CMN(n,r-\snorm{A}-1) } }
	(\prod_{l \ne h} {{\lambda}_l}^{{j}_l} ) \cdot
	{({\lambda}_h + 1 - 1)}^{j_h}  (\lambda_h + 1) \cdot
	\phi_{(\SE{A} + ({\lambda} + \bs{\ep}_h) |\SO{A} )}  \\
&=
    \sum_{k=0}^{j_h}
    {\binom{j_h}{k}}
	{(-1)}^{k}
	\sum_{\substack{ \lambda \in \CMN(n,r-\snorm{A}-1) } }
    ({\lambda} + \bs{\ep}_h)^{\bs{j} + (1-k)\bs{\ep}_h}
	\phi_{(\SE{A} + ({\lambda} + \bs{\ep}_h) |\SO{A} )}\\
&=\sum_{k=0}^{j_h}
    {\binom{j_h}{k}}
	{(-1)}^{k}
	\sum_{\substack{ \mu \in \CMN(n,r-\snorm{A})\\\mu_h\geq 1 } }
    {\mu}^{\bs{j} + (1-k)\bs{\ep}_h}
	\phi_{(\SE{A} + \mu |\SO{A} )}
	 \quad (\mbox {set }{\mu=\lambda+\bs{\ep}_h}).
\end{align*}
Since $j_h+1-k>0$, it yields $\mu_h^{j_h+1-k}=0$ when $\mu_h=0$.
As a result,
\begin{align*}
\mbox{ LHS of {\rm (i)}}=\sum_{k=0}^{j_h}
    {\binom{j_h}{k}}
	{(-1)}^{k}
	\sum_{\substack{ \mu \in \CMN(n,r-\snorm{A})} }
    (\mu)^{\bs{j} + (1-k)\bs{\ep}_h}
	\phi_{(\SE{A} + \mu |\SO{A} )} =\mbox{RHS of {\rm (i)}}.
\end{align*}
Secondly,
\begin{align*}
\mbox{LHS of {\rm (ii)}}=&\sum_{\substack{ \lambda \in \CMN(n,r-\snorm{A}+1)\\\lambda_h\geq 1 } }
	{\lambda} ^{\bs{j}}
	\phi_{\SE{A} + {\lambda} - E_{h,h} |\SO{A} }  \\
&=
	\sum_{\substack{ \lambda \in \CMN(n,r-\snorm{A}+1) \\\lambda_h\geq 1} }
	\prod_{l \ne h} {{\lambda}_l}^{{j}_l} \cdot
	{({\lambda}_h -1 + 1)}^{j_h}  \cdot
	\phi_{\SE{A} + ({\lambda} - \bs{\ep}_h) |\SO{A} }  \\
&=
    \sum_{k=0}^{j_h} {\binom{j_h}{k}}
	\sum_{\substack{ \lambda \in \CMN(n,r-\snorm{A}+1) \\\lambda_h\geq 1} }
	{(\lambda - \bs{\ep}_h)}^{\bs{j}  - k \bs{\ep}_h}  \cdot
	\phi_{\SE{A} + ({\lambda} - \bs{\ep}_h) |\SO{A} }  \\
&= \sum_{k=0}^{j_h} {\binom{j_h}{k}}
	\sum_{\substack{ \mu \in \CMN(n,r-\snorm{A}) } }
	{\mu}^{\bs{j} - k \bs{\ep}_h}  \cdot
	\phi_{\SE{A} + \mu |\SO{A} }=\mbox{ RHS of {\rm (ii)}}.
\end{align*}
\end{proof}

\begin{prop}\label{mulformajrb1}
	Let {$n, r\in \NN$} and {$h \in [1,n]$}.
	For any {$A \in \MNZNS(n)$} and 	{$\bs{j} \in \NN^n$},
	the following multiplication formulas are satisfied  in {$\Qnr$} for all $r\geq\snorm{A}$:
\begin{align*}
	{\rm (i)} \quad & \AJRS(\bs{O}, \bs{0}, r) \cdot \ABJRS( \SE{A}, \SO{A}, \bs{ j },  r )
		= \ABJRS( \SE{A}, \SO{A}, \bs{ j },  r ); \\
	{\rm (ii)} \quad & \AJRS(\bs{O}, \bs{\ep}_h, r) \cdot \ABJRS( \SE{A}, \SO{A}, \bs{ j },  r ) = \ABJRS( \SE{A}, \SO{A}, \bs{ j } + \bs{\ep}_h,  r )
		\ + \ \sum_{k=1}^n a_{h,k} \ABJRS( \SE{A}, \SO{A}, \bs{ j },  r ); \\
	{\rm (iii)} \quad & \ABJRS( \SE{A}, \SO{A}, \bs{ j },  r )  \cdot  \AJRS(\bs{O}, \bs{\ep}_h, r)  = \ABJRS( \SE{A}, \SO{A}, \bs{ j } + \bs{\ep}_h,  r )
		\ + \ \sum_{k=1}^n a_{k,h} \ABJRS( \SE{A}, \SO{A}, \bs{ j },  r ); \\
{\rm (iv)}
	\quad & \ABJRS( O, E_{h,h}, \bs{ 0 },  r ) \cdot \ABJRS( \SE{A}, \SO{A}, \bs{ j },  r ) \\
	&= \sum_{1\leq k\leq n,k\neq h} {(-1)}^{\SOE{\tilde{a}}_{h,k}}
			(\SE{A} - E_{h,k}|\SO{A} + E_{h,k})[\bs{j},r] \\
   	& \qquad +
        {(-1)}^{\SOE{\tilde{a}}_{h,h}}
        \sum_{k=0}^{j_h} {\binom{j_h}{k}}
    	\ABJRS( \SE{A}, \SO{A} + E_{h,h}, \bs{ j } - k \bs{\ep}_h, r ) \\
	& \qquad +
	\sum_{ 1\leq k \leq n, k \neq h} {(-1)}^{\SOE{\tilde{a}}_{h,k}}
		 a_{h,k}(\SE{A} + E_{h,k}|\SO{A} - E_{h,k})[\bs{j},r] \\
	& \qquad +
	{(-1)}^{\SOE{\tilde{a}}_{h,h}}
        \sum_{k=0}^{j_h} {\binom{j_h}{k}}
    	{(-1)}^{k}
    	\ABJRS( \SE{A}, \SO{A}  - E_{h,h}, \bs{ j } +(1- k) \bs{\ep}_h, r ).
\end{align*}
\end{prop}
Note that the multiplication formulas above as well as those given in Propostions \ref{mulformajrb2}--\ref{mulformajrb3} are independent of $r$ as long as $r\geq|A|.$ This remarkable property allows to extend the formulas to a ``limit'' algebra---the direct product of $\Qnr$---which will be discussed in the next section.
\begin{proof}
%
%
{\rm (i)}. Theorem \ref{formD0} implies
\begin{align*}
\AJRS(\bs{O}, \bs{0}, r) \cdot \ABJRS( \SE{A}, \SO{A}, \bs{ j },  r )
&=\sum_{\substack{  \\\mu \in \CMN(n,r)  }}
		\mu ^{\bs{0}} \phi_{{\mu} | O }
		\cdot
		\sum_{\substack{ \\ \lambda \in  \CMN(n,r-\snorm{A}) }
		} \lambda ^{\bs{j}} \phi_{(\MNZADD(A, \lambda))} \\
&=\sum_{\substack{ \\ \lambda \in  \CMN(n, r-\snorm{A}) \\
		} }
		\lambda ^{\bs{j}} \phi_{\MNZADD(A, \lambda) } = \ABJRS( \SE{A}, \SO{A}, \bs{ j },  r ).
\end{align*}
	{\rm (ii)}. Similar to {\rm (i)}, according to Theorem \ref{formD0},
\begin{align*}
	 \AJRS(\bs{O}, \bs{\ep}_h, r) \cdot \ABJRS( \SE{A}, \SO{A}, \bs{ j },  r )
	& =\sum_{\substack{  \\ \mu \in \CMN(n,r)   }}
		\mu ^{\bs{\ep}_h} \phi_{{\mu} | O }
		\cdot
		\sum_{\substack{\\ \lambda \in \CMN(n, r-\snorm{A}) } }
		\lambda ^{\bs{j}} \phi_{\MNZADD(A, \lambda)} \\
	& =\sum_{\substack{ \lambda \in \CMN(n, r-\snorm{A})
		} }
		{(\ro(A + \lambda))} ^{\bs{\ep}_h} \phi_{{\mu} | O }
		\lambda ^{\bs{j}} \phi_{ \MNZADD(A, \lambda)} \\
	& =\sum_{\substack{ \lambda \in \CMN(n, r-\snorm{A})
		} }
		( \sum_{k=1}^n a_{h,k} + \lambda_h) \phi_{{\mu} | O }
		\lambda ^{\bs{j}} \phi_{ \MNZADD(A, \lambda)}.
\end{align*}
 Since ${\lambda}_h\lambda ^{\bs{j}}={\lambda}^{\bs{j}+\bs{\ep}_h}$, we have
 $\sum_{\substack{ \\ \lambda \in  \CMN(n, r-\snorm{A}) } }
		{\lambda}_h
		\lambda ^{\bs{j}}
		\phi_{\MNZADD(A, \lambda) }=\ABJRS( \SE{A}, \SO{A}, \bs{ j } + \bs{\ep}_h,  r ),$
hence
\begin{align*}
	 \AJRS(\bs{O}, \bs{\ep}_h, r) \cdot \ABJRS( \SE{A}, \SO{A}, \bs{ j },  r ) = \ABJRS( \SE{A}, \SO{A}, \bs{ j } + \bs{\ep}_h,  r )
		\ + \ \sum_{k=1}^n a_{h,k} \ABJRS( \SE{A}, \SO{A}, \bs{ j },  r ).
\end{align*}
{\rm (iii)}. It can be proved in the same way as {\rm(ii)}.\\
{\rm (iv)}.
\begin{align*}
&\ABJRS( O, E_{h,h}, \bs{ 0 },  r ) \cdot \ABJRS( \SE{A}, \SO{A}, \bs{ j },  r ) \\
	& =\sum_{\substack{   \\ \mu \in \CMN(n, r-1) } }
		\mu ^{\bs{0}} \phi_{({\mu} | E_{h,h}) }
		\cdot
		\sum_{\substack{ \\ \lambda \in \CMN(n, r-\snorm{A}) } }
		\lambda ^{\bs{j}} \phi_{(\MNZADD(A, \lambda))} \\
	& = \sum_{\substack{
			\mu \in \CMN(n, r-1),
			\lambda \in \CMN(n, r-\snorm{A}) \\
			\co(\mu + E_{h,h}) = \ro(A+\lambda)
			}
		}
		\mu ^{\bs{0}}  \lambda ^{\bs{j}}
		\cdot
		\phi_{({\mu}|E_{h,h})} \phi_{(\MNZADD(A, \lambda)) }
		\quad \mbox{ (by Theorem \ref{formD1})}\\
	& = \sum_{\substack{ \lambda \in \CMN(n,r-\snorm{A})\\ 1\leq k\leq n } }
		\lambda ^{\bs{j}}		
		[
		{(-1)}^{\SOE{\tilde{a}}_{h,k}}
		\phi_{(\SE{A} + \lambda-E_{h,k}|\SO{A} + E_{h,k})}\\
		&\qquad \qquad  +
		{(-1)}^{\SOE{\tilde{a}}_{h,k}}
		{({A} + \lambda)}_{h,k} \phi_{(\SE{A} + \lambda+E_{h,k}|\SO{A} - E_{h,k})}
		]
 \\
    &= F_1 + F_2 + F_3 + F_4,
\end{align*}
    where
\begin{align*}
F_1 &= \sum_{\substack{  \lambda \in \CMN(n,r-\snorm{A}) \\1\leq k\leq n,k\neq h } }
		\lambda ^{\bs{j}}  {(-1)}^{\SOE{\tilde{a}}_{h,k}}
		\phi_{(\SE{A} + \lambda-E_{h,k}|\SO{A} + E_{h,k})}  \\
	&= \sum_{1\leq k\leq n,k\neq h} {(-1)}^{\SOE{\tilde{a}}_{h,k}}
			\ABJRS(\SE{A} - E_{h,k}, \SO{A} + E_{h,k}, \bs{j},r);\\
F_2 &=
        {(-1)}^{\SOE{\tilde{a}}_{h,h}}
        \sum_{\substack{  \lambda \in \CMN(n,r-\snorm{A}) } }
		\lambda ^{\bs{j}}
		\phi_{(\SE{A} + \lambda-E_{h,h}|\SO{A} + E_{h,h})} \\
    &=
        {(-1)}^{\SOE{\tilde{a}}_{h,h}}
        \sum_{k=0}^{j_h} {\binom{j_h}{k}}
    	\ABJRS( \SE{A}, \SO{A} + E_{h,h}, \bs{ j } - k \bs{\ep}_h, r )
    	 \mbox{ ( by Lemma \ref{mulformajrb1.5} )}; \\
F_3 &= \sum_{\substack{
				\lambda \in \CMN(n,r-\snorm{A})\\
				1\leq k\leq n,k\neq h
			} }
		\lambda ^{\bs{j}}{(-1)}^{\SOE{\tilde{a}}_{h,k}}
		{a}_{h,k} \phi_{(\SE{A} + \lambda+E_{h,k}|\SO{A} - E_{h,k})}
		 \\
	&= \sum_{ 1\leq k \leq n, k \neq h} {(-1)}^{\SOE{\tilde{a}}_{h,k}}
		 a_{h,k}
		 \ABJRS(\SE{A} + E_{h,k}, \SO{A} - E_{h,k}, \bs{j},r); \\
F_4 &=
        {(-1)}^{\SOE{\tilde{a}}_{h,h}}
        \sum_{\substack{  \lambda \in \CMN(n,r-\snorm{A}) } }
		\lambda ^{\bs{j}}(\lambda_h+1)
		\phi_{(\SE{A} + \lambda+E_{h,h}|\SO{A} - E_{h,h})} \\
&=
	{(-1)}^{\SOE{\tilde{a}}_{h,h}}
        \sum_{k=0}^{j_h} {\binom{j_h}{k}}
    	{(-1)}^{k}
    	\ABJRS( \SE{A}, \SO{A}  - E_{h,h}, \bs{ j } +(1- k) \bs{\ep}_h, r )
    	\mbox{ ( by Lemma \ref{mulformajrb1.5} )} .
\end{align*}
Taking the sum of $F_1$  to $F_4$ will prove {\rm (iv)}.
\end{proof}
\begin{prop}\label{mulformajrb2}
	For any {$A \in \MNZNS(n)$},
	{$\bs{j} \in \NN^n$},
	and {$h \in [1,n-1]$},
	the following multiplication formulas hold in {$\Qnr$} for all $r\geq\snorm{A}$:
\begin{align*}
	{\rm (i)} \quad & \ABJRS( E_{h, h+1}, O, \bs{ 0 },  r ) \cdot \ABJRS( \SE{A}, \SO{A}, \bs{ j },  r ) \\
	&=
		\sum _{\substack{
			1 \le k \le n \\
			k \ne h, h+1
		}}
		(\SEE{a}_{h,k}+1)
		\ABJRS( \SE{A} + E_{h,k} - E_{h+1,k}, \SO{A}, \bs{ j },  r ) \\
    &+ \sum_{k=0}^{j_h} {\binom{j_h}{k}}
	{(-1)}^{k}
	\ABJRS( \SE{A} - E_{h+1,h}, \SO{A}, \bs{ j } +(1- k) \bs{\ep}_h, r ) \\
	&+
		(\SEE{a}_{h,h+1}+1)
		\sum_{k=0}^{j_{h+1}} {(\substack{ j_{h+1} \\ \ \\ k })}
		\ABJRS( \SE{A}+ E_{h,h+1}, \SO{A}, \bs{ j } - k \bs{\ep}_{h+1}, r ) \\
	&+
	\sum _{\substack{1 \le k \le n  }}
		\ABJRS( \SE{A}, \SO{A}+ E_{h,k} - E_{h+1,k}, \bs{ j },  r ) ;\\
{\rm (ii)} \quad & \ABJRS( O, E_{h, h+1}, \bs{ 0 },  r  ) \cdot \ABJRS( \SE{A}, \SO{A}, \bs{ j },  r ) \\
	&= \sum _{\substack{ 1 \le k \le n \\ k \ne h }}
	 {(-1)}^{\SOE{\tilde{a}}_{h+1,k}} (\SEE{a}_{h,k}+1)
	\ABJRS( \SE{A} + E_{h,k}, \SO{A} - E_{h+1,k}, \bs{ j },  r ) \\
&+	{(-1)}^{\SOE{\tilde{a}}_{h+1,h}}
    \sum_{k=0}^{j_h} {\binom{j_h}{k}}
	{(-1)}^{k}
	\ABJRS( \SE{A}, \SO{A} - E_{h+1,h}, \bs{ j } +(1- k) \bs{\ep}_h, r ) \\
&+
	\sum _{\substack{ 1 \le k \le n \\ k \ne h+1 }}
	{(-1)}^{\SOE{\tilde{a}}_{h,k}}
	\ABJRS( \SE{A}- E_{h+1,k}, \SO{A} + E_{h,k}, \bs{ j },  r ) \\
&+
	{(-1)}^{\SOE{\tilde{a}}_{h,h+1}}
	\sum_{k=0}^{j_{h+1}} {(\substack{ j_{h+1} \\ \ \\ k })}
	\ABJRS( \SE{A}, \SO{A} + E_{h,h+1}, \bs{ j } - k \bs{\ep}_{h+1}, r ).
\end{align*}
\end{prop}
\begin{proof}
Similar to the proof of Proposition \ref{mulformajrb1},
applying Theorem \ref{upper0} and \ref{U1A} may prove these formulas.

{\rm (i)} For the given {$A$},
\begin{align*}
&\ABJRS( E_{h, h+1}, O, \bs{ 0 },  r ) \cdot \ABJRS( \SE{A}, \SO{A}, \bs{ j },  r ) \\
	& =\sum_{\substack{  \mu \in \CMN(n, r-1) } }
		\mu ^{\bs{0}} \phi_{(E_{h,h+1} + \mu | O)}
		\cdot
		\sum_{\substack{ \\ \lambda \in \CMN(n,r-\snorm{A}) } }
		\lambda ^{\bs{j}} \phi_{(\MNZADD(A, \lambda))} \\
	& =\sum_{\substack{
			\lambda \in \CMN(n,r-\snorm{A}),
			\mu \in \CMN(n, r-1) \\
			\co(E_{h,h+1}) + \mu = \ro(A) +\lambda
		} }
		\mu ^{\bs{0}}
		\lambda ^{\bs{j}}
		\cdot
		\phi_{(E_{h, h+1} + \mu | O)}
		\phi_{(\MNZADD(A, \lambda)) } \\
	& =
		\sum_{\substack{
			\lambda \in \CMN(n,r-\snorm{A}) \\
		} }
		\lambda ^{\bs{j}}
		\sum _{\substack{ 1 \le k \le n }} [
			(\SEE{{(A+\lambda)}}_{h,k}+1) \phi_{(\SE{A} + \lambda + E_{h,k} - E_{h+1,k} |\SO{A}) }\\
			&\qquad +
			\phi_{(\SE{A} + {\lambda}|\SO{A} + E_{h,k} - E_{h+1,k})} ]
			\quad \mbox{( by Theorem \ref{upper0}})\\
	& =
		\sum _{\substack{
			1 \le k \le n \\
			k \ne h, h+1
		}}
		\sum_{\substack{
			\lambda \in \CMN(n,r-\snorm{A}) \\
		} }
		\lambda ^{\bs{j}}
			(\SEE{(A+\lambda)}_{h,k}+ 1)
			\phi_{(\SE{A} + {\lambda} + E_{h,k} - E_{h+1,k} |\SO{A} )}  \\
		&\qquad +
		\sum_{\substack{ \lambda \in \CMN(n,r-\snorm{A}) } }
		\lambda ^{\bs{j}}
			( \SEE{(A+\lambda)}_{h,h}+1)
			\phi_{(\SE{A} + {\lambda} + E_{h,h} - E_{h+1,h} |\SO{A} )}  \\
		&\qquad +
		\sum_{\substack{ \lambda \in \CMN(n,r-\snorm{A}) } }
		\lambda ^{\bs{j}}
			(\SEE{(A+\lambda)}_{h,h+1}+1)
			\phi_{(\SE{A} + {\lambda} + E_{h,h+1} - E_{h+1,h+1} |\SO{A} )}  \\
		&\qquad +
		\sum _{\substack{
			1 \le k \le n \\
		}}
		\sum_{\substack{
			\lambda \in \CMN(n,r-\snorm{A})
		} }
		\lambda ^{\bs{j}}
			\phi_{(\SE{A} + {\lambda}|\SO{A} + E_{h,k} - E_{h+1,k})}
			\quad (\mbox{ by Theorem }\ref{upper0})\\
	&= F_1 + F_2 + F_3 + F_4,
\end{align*}
where
\begin{align*}
	F_1 &= \sum _{\substack{
			1 \le k \le n \\
			k \ne h, h+1
		}}
		\sum_{\substack{
			\lambda \in \CMN(n,r-\snorm{A}) \\
		} }
		\lambda ^{\bs{j}} (\SEE{a_{h,k}}+1)
		\phi_{(\SE{A} + {\lambda} + E_{h,k} - E_{h+1,k} |\SO{A}) }  \\
	&= \sum _{\substack{
			1 \le k \le n \\
			k \ne h, h+1
		}}
		(\SEE{a_{h,k}}+1)
		\ABJRS( \SE{A} + E_{h,k} - E_{h+1,k}, \SO{A}, \bs{ j },  r ); \\
	F_2 &=	\sum_{\substack{ \lambda \in \CMN(n,r-\snorm{A}) } }
		\lambda ^{\bs{j}}
			(\SEE{{a}_{h,h}}+ {\lambda}_{h} + 1)
			\phi_{(\SE{A} + {\lambda} + E_{h,h} - E_{h+1,h} |\SO{A}) }  \\
	&=	\sum_{\substack{ \lambda \in \CMN(n,r-\snorm{A}) } }
		\lambda ^{\bs{j}} ({\lambda}_{h} + 1)
			\phi_{(\SE{A} - E_{h+1,h} + {\lambda} + E_{h,h} |\SO{A} )}
  \\
    &=  \sum_{k=0}^{j_h} {\binom{j_h}{k}}
	{(-1)}^{k}
	\ABJRS( \SE{A} - E_{h+1,h}, \SO{A}, \bs{ j } +(1- k) \bs{\ep}_h, r )
	\quad (\mbox{by Lemma } \ref{mulformajrb1.5} {\rm (i)}); \\
	F_3 &= \sum_{\substack{ \lambda \in \CMN(n,r-\snorm{A}) } }
		\lambda ^{\bs{j}}
			( \SEE{a}_{h,h+1}+1)
			\phi_{(\SE{A} + {\lambda} + E_{h,h+1} - E_{h+1,h+1} |\SO{A}) }  \\
	&=
		( \SEE{a}_{h,h+1}+1)
		\sum_{k=0}^{j_{h+1}} {(\substack{ j_{h+1} \\ \ \\ k })}
		\ABJRS( \SE{A}+ E_{h,h+1}, \SO{A}, \bs{ j } - k \bs{\ep}_{h+1}, r )
		 \mbox{ ( by Lemma \ref{mulformajrb1.5}{\rm (ii)} )}; \\
	F_4& =\sum _{\substack{
			1 \le k \le n
		}}
		\sum_{\substack{
			\lambda \in \CMN(n,r-\snorm{A})
		} }
		\lambda ^{\bs{j}}
			\phi_{(\SE{A} + {\lambda}|\SO{A}+ E_{h,k} - E_{h+1,k})} \\
	&=
	\sum _{\substack{
			1 \le k \le n
		}}
		\ABJRS( \SE{A}, \SO{A}+ E_{h,k} - E_{h+1,k}, \bs{ j },  r ).
\end{align*}
Taking the sum of $F_1$ to $F_4$,  formula {\rm (i)} is proved.

{\rm (ii)}.
\begin{align*}
& \ABJRS( O, E_{h, h+1}, \bs{ 0 },  r  ) \cdot \ABJRS( \SE{A}, \SO{A}, \bs{ j },  r ) \\
& =\sum_{\substack{  \mu \in \CMN(n, r-1) } }
	\mu ^{\bs{0}} \phi_{(\mu| E_{h,h+1})}
	\cdot
	\sum_{\substack{ \\ \lambda \in \CMN(n,r-\snorm{A}) } }
	\lambda ^{\bs{j}} \phi_{(\MNZADD(A, \lambda))} \\
& =\sum_{\substack{
		\lambda \in \CMN(n,r-\snorm{A}),
		\mu \in \CMN(n, r-1) \\
		\co(E_{h,h+1}) + \mu = \ro(A) +\lambda
	} }
	\mu ^{\bs{0}} \lambda ^{\bs{j}}
	\cdot
	\phi_{(\mu| E_{h, h+1})}
	\phi_{(\MNZADD(A, \lambda) )} \\
& =\sum_{\substack{
		\lambda \in \CMN(n,r-\snorm{A}) \\
	} }
	\lambda ^{\bs{j}}
		\sum _{\substack{ 1 \le k \le n  }} [  {(-1)}^{\SOE{\tilde{a}}_{h+1,k}} (\SEE{(A + \lambda)}_{h,k}+1)\phi_{(\SE{A}+ \lambda + E_{h,k}| \SO{A} - E_{h+1,k})}  \\
		& \qquad + {(-1)}^{\SOE{\tilde{a}}_{h,k}}  \phi_{(\SE{A}+ \lambda - E_{h+1,k}| \SO{A} + E_{h,k})}]
		\quad (\mbox{by Theorem }\ref{U1A})
	\\
& = F_1 + F_2 + F_3 + F_4  ;
\end{align*}
where
\begin{align*}
F_1 &= \sum_{\substack{  \lambda \in \CMN(n,r-\snorm{A}) } }
	\lambda ^{\bs{j}}
	\sum _{\substack{ 1 \le k \le n \\ k \ne h  }}
	{(-1)}^{\SOE{\tilde{a}}_{h+1,k}} (\SEE{(A + \lambda)}_{h,k}+1)
	\phi_{(\SE{A}+ \lambda + E_{h,k}| \SO{A} - E_{h+1,k})} \\
&=\sum _{\substack{ 1 \le k \le n \\ k \ne h  }}
	{(-1)}^{\SOE{\tilde{a}}_{h+1,k}} (\SEE{a}_{h,k}+1)
	 \sum_{\substack{  \lambda \in \CMN(n,r-\snorm{A}) } }
	\lambda ^{\bs{j}}
	\phi_{(\SE{A}+ \lambda + E_{h,k}| \SO{A} - E_{h+1,k})}  \\
&=\sum _{\substack{ 1 \le k \le n \\ k \ne h }}
	{(-1)}^{\SOE{\tilde{a}}_{h+1,k}} (\SEE{a}_{h,k}+1)
	\ABJRS( \SE{A} + E_{h,k}, \SO{A} - E_{h+1,k}, \bs{ j },  r ); \\
F_2 &= \sum_{\substack{  \lambda \in \CMN(n,r-\snorm{A}) } }
	\lambda ^{\bs{j}}
	{(-1)}^{\SOE{\tilde{a}}_{h+1,h}} (\SEE{(A + \lambda)}_{h,h}+1)
	\phi_{(\SE{A}+ \lambda + E_{h,h}| \SO{A} - E_{h+1,h})} \\
&=
\hspace{-1mm}
	{(-1)}^{\SOE{\tilde{a}}_{h+1,h}}
    \sum_{\substack{  \lambda \in \CMN(n,r-\snorm{A}) } }
	\lambda ^{\bs{j}}  ({\lambda}_h + 1)
	\phi_{(\SE{A}+ \lambda + E_{h,h}| \SO{A} - E_{h+1,h})}\quad ( \mbox{ since $\SEE{a}_{h,h}=0$})\\
&=\hspace{-1mm}
	{(-1)}^{\SOE{\tilde{a}}_{h+1,h}}   \!
    \sum_{k=0}^{j_h}    \!
     {\binom{j_h}{k}}	
	{(-1)}^{k}
	\ABJRS( \SE{A}, \SO{A} - E_{h+1,h}, \bs{ j } +(1- k) \bs{\ep}_h, r )  \!
	 \mbox{ (by Lemma \ref{mulformajrb1.5})};\\
F_3 &= \sum_{\substack{  \lambda \in \CMN(n,r-\snorm{A}) } }
	\lambda ^{\bs{j}}
	\sum _{\substack{ 1 \le k \le n \\ k \ne h+1 }}
	{(-1)}^{\SOE{\tilde{a}}_{h,k}}
	\phi_{(\SE{A}+ \lambda - E_{h+1,k}| \SO{A} + E_{h,k})} \\
&=
	\sum _{\substack{ 1 \le k \le n \\ k \ne h+1 }}
	{(-1)}^{\SOE{\tilde{a}}_{h,k}}
	\ABJRS( \SE{A}- E_{h+1,k}, \SO{A} + E_{h,k}, \bs{ j },  r ); \\
F_4 &= \sum_{\substack{ \lambda \in \CMN(n,r-\snorm{A}) } }
	\lambda ^{\bs{j}}
	{(-1)}^{\SOE{\tilde{a}}_{h,h+1}}
	\phi_{(\SE{A}+ \lambda - E_{h+1,h+1}| \SO{A} + E_{h,h+1})} \\
&=
	{(-1)}^{\SOE{\tilde{a}}_{h,h+1}}
	\sum_{k=0}^{j_{h+1}} {(\substack{ j_{h+1} \\ \ \\ k })}
	\ABJRS( \SE{A}, \SO{A} + E_{h,h+1}, \bs{ j } - k \bs{\ep}_{h+1}, r )
	\quad \mbox{ (by Lemma  \ref{mulformajrb1.5}) }.
\end{align*}

Hence the formulas are proved.
\end{proof}

\begin{prop}\label{mulformajrb3}
	For any {$A \in \MNZNS(n)$},
	{$\bs{j} \in \NN^n$},
	and {$h \in [1,n-1]$},
	the following multiplication formulas hold in {$\Qnr$} for all  $r\geq\snorm{A}$:
\begin{align*}
{\rm (i)} \quad & \ABJRS( E_{h+1, h}, O, \bs{ 0 },  r ) \cdot \ABJRS( \SE{A}, \SO{A}, \bs{ j },  r ) \\
&=
	\sum _{\substack{ 1 \le k \le n \\k \ne h,h+1}}
	(\SEE{a}_{h+1,k} + 1)
	\ABJRS( \SE{A} -E_{h,k} + E_{h+1,k}, \SO{A}, \bs{ j },  r ) \\
	&+ (\SEE{a}_{h+1,h} + 1)
	\sum_{k=0}^{j_h} {\binom{j_h}{k}}
	\ABJRS( \SE{A} + E_{h+1,h}, \SO{A}, \bs{ j } - k \bs{\ep}_h, r ) \\
    &+ \sum_{k=0}^{j_{h+1}} {(\substack{ j_{h+1} \\ \ \\ k })}
	{(-1)}^{k}
	\ABJRS( \SE{A}-E_{h,h+1}, \SO{A}, \bs{ j } +(1- k) \bs{\ep}_{h+1}, r ) \\
	&+ \sum _{\substack{ 1 \le k \le n }}
	\ABJRS( \SE{A}, \SO{A} -E_{h,k} + E_{h+1,k}, \bs{ j },  r ); \\
{\rm (ii)} \quad & \ABJRS( O, E_{h+1, h}, \bs{ 0 },  r ) \cdot \ABJRS( \SE{A}, \SO{A}, \bs{ j },  r ) \\
&=
	\sum _{ \substack{ 1 \le k \le n \\ k \ne h+1 }}
		{(-1)}^{\SOE{\tilde{a}}_{h,k} }
		(\SEE{a}_{h+1,k} + 1)
	\ABJRS( \SE{A}+ E_{h+1,k}, \SO{A} - E_{h,k}, \bs{ j },  r )  \\
&+  {(-1)}^{\SOE{\tilde{a}}_{h,h+1} }
    \sum_{k=0}^{j_{h+1}} {(\substack{ j_{h+1} \\ \ \\ k })}
	{(-1)}^{k}
	\ABJRS( \SE{A}, \SO{A}  - E_{h,h+1}, \bs{ j } +(1- k) \bs{\ep}_{h+1}, r ) \\
&+
	\sum _{ \substack{ 1 \le k \le n \\ k \ne h }}
	{(-1)}^{\SOE{\tilde{a}}_{h,k}}
	\ABJRS( \SE{A} -E_{h,k}, \SO{A} + E_{h+1,k}, \bs{ j },  r ) \\
&+
	{(-1)}^{\SOE{\tilde{a}}_{h,h}}
	\sum_{k=0}^{j_h} {\binom{j_h}{k}}
	\ABJRS( \SE{A}, \SO{A} +  E_{h+1,h}, \bs{ j } - k \bs{\ep}_h, r ). \\
\end{align*}
\end{prop}
\begin{proof}
Theorem \ref{L0A} and \ref{L1A} will be applied to prove these formulas.

{\rm (i)}.
\begin{align*}
&\ABJRS( E_{h+1,h}, O, \bs{ 0 },  r ) \cdot \ABJRS( \SE{A}, \SO{A}, \bs{ j },  r ) \\
& =\sum_{\substack{  \mu \in \CMN(n, r-1) } }
	\mu ^{\bs{0}} \phi_{(E_{h+1,h} + \mu | O)}
	\cdot
	\sum_{\substack{ \\ \lambda \in \CMN(n,r-\snorm{A}) } }
	\lambda ^{\bs{j}} \phi_{(\MNZADD(A, \lambda))} \\
& =\sum_{\substack{
		\lambda \in \CMN(n,r-\snorm{A}),
		\mu \in \CMN(n, r-1) \\
		\co(E_{h+1,h}) + \mu = \ro(A) +\lambda
	} }
	\mu ^{\bs{0}}
	\lambda ^{\bs{j}}
	\cdot
	\phi_{(E_{h+1,h} + \mu | O)}
	\phi_{(\MNZADD(A, \lambda) )} \\
& =\sum_{\substack{  \lambda \in \CMN(n,r-\snorm{A}) } }
	\lambda ^{\bs{j}}
	\sum _{\substack{ 1 \le k \le n }}  [
	(\SEE{(A + \lambda)}_{h+1,k} + 1)
	\phi_{(\SE{A} + \lambda -E_{h,k} + E_{h+1,k} | \SO{A})} \\
	& \qquad + \phi_{(\SE{A} + \lambda  | \SO{A} -E_{h,k} + E_{h+1,k})} ]
	\quad (\mbox{by Theorem }\ref{L0A})\\
& = F_1 + F_2 + F_3 + F_4 ;
\end{align*}
where
\begin{align*}
F_1 &=
	\sum_{\substack{  \lambda \in \CMN(n,r-\snorm{A}) } }
	\lambda ^{\bs{j}}
	\sum _{\substack{ 1 \le k \le n \\k \ne h,h+1}}
	(\SEE{(A + \lambda)}_{h+1,k} + 1)
	\phi_{(\SE{A} + \lambda -E_{h,k} + E_{h+1,k} | \SO{A})} \\
&=
	\sum _{\substack{ 1 \le k \le n \\k \ne h,h+1}}
	(\SEE{a}_{h+1,k} + 1)
	\sum_{\substack{  \lambda \in \CMN(n,r-\snorm{A}) } }
	\lambda ^{\bs{j}}
	\phi_{(\SE{A} + \lambda -E_{h,k} + E_{h+1,k} | \SO{A})} \\
&=
	\sum _{\substack{ 1 \le k \le n \\k \ne h,h+1}}
	(\SEE{a}_{h+1,k} + 1)
	\ABJRS( \SE{A} -E_{h,k} + E_{h+1,k}, \SO{A}, \bs{ j },  r ); \\
F_2 &= \sum_{\substack{  \lambda \in \CMN(n,r-\snorm{A}) } }
	\lambda ^{\bs{j}}
	(\SEE{a}_{h+1,h} + 1)
	\phi_{(\SE{A} + \lambda -E_{h,h} + E_{h+1,h} | \SO{A})} \\
&=
	(\SEE{a}_{h+1,h} + 1)
	\sum_{k=0}^{j_h} {\binom{j_h}{k}}
	\ABJRS( \SE{A} + E_{h+1,h}, \SO{A}, \bs{ j } - k \bs{\ep}_h, r ); \\
F_3&= \sum_{\substack{  \lambda \in \CMN(n,r-\snorm{A}) } }
	\lambda ^{\bs{j}}
	(\SEE{(A + \lambda)}_{h+1,h+1} + 1)
	\phi_{(\SE{A} + \lambda -E_{h,h+1} + E_{h+1,h+1} | \SO{A})} \\
&= \sum_{\substack{  \lambda \in \CMN(n,r-\snorm{A}) } }
	\lambda ^{\bs{j}} ({\lambda}_{h+1} +1 )
	\phi_{(\SE{A} -E_{h,h+1} + \lambda + E_{h+1,h+1} | \SO{A})}
	\quad (\mbox{since $\SEE{a}_{h+1,h+1}=0$}) \\
&=
    \sum_{k=0}^{j_{h+1}} {(\substack{ j_{h+1} \\ \ \\ k })}
	{(-1)}^{k}
	\ABJRS( \SE{A}-E_{h,h+1}, \SO{A}, \bs{ j } +(1- k) \bs{\ep}_{h+1}, r )
	\quad (\mbox{by Lemma } \ref{mulformajrb1.5}); \\
F_4 &= \sum_{\substack{  \lambda \in \CMN(n,r-\snorm{A}) } }
	\lambda ^{\bs{j}}
	\sum _{\substack{ 1 \le k \le n  }}
	\phi_{(\SE{A} + \lambda  | \SO{A} -E_{h,k} + E_{h+1,k})}  \\
&=
	\sum _{\substack{ 1 \le k \le n }}
	\ABJRS( \SE{A}, \SO{A} -E_{h,k} + E_{h+1,k}, \bs{ j },  r ).
\end{align*}
Hence the formula is proved.\\
{\rm (ii)}.
\begin{align*}
& \ABJRS( O, E_{h+1,h}, \bs{ 0 },  r  ) \cdot \ABJRS( \SE{A}, \SO{A}, \bs{ j },  r ) \\
& =\sum_{\substack{  \mu \in \CMN(n, r-1) } }
	\mu ^{\bs{0}} \phi_{( \mu | E_{h+1,h})}
	\cdot
	\sum_{\substack{ \\ \lambda \in \CMN(n,r-\snorm{A}) } }
	\lambda ^{\bs{j}} \phi_{(\MNZADD(A, \lambda))} \\
& =\sum_{\substack{
		\lambda \in \CMN(n,r-\snorm{A}),
		\mu \in \CMN(n, r-1) \\
		\co(E_{h+1,h}) + \mu = \ro(A) +\lambda
	} }
	\mu ^{\bs{0}}
	\lambda ^{\bs{j}}
	\cdot
	\phi_{(\mu | E_{h+1,h})}
	\phi_{(\MNZADD(A, \lambda) )} \\
& =\sum_{\substack{ \lambda \in \CMN(n,r-\snorm{A}) } }
	\lambda ^{\bs{j}}
	\sum _{ \substack{ 1 \le k \le n  }}  [
		{(-1)}^{\SOE{\tilde{a}}_{h,k} }
		(\SEE{(A+\lambda)}_{h+1,k} + 1)
		\phi_{(\SE{A} + \lambda + E_{h+1,k}| \SO{A} - E_{h,k})}  \\
	& \qquad + {(-1)}^{\SOE{\tilde{a}}_{h,k}}
		\phi_{(\SE{A} + \lambda -E_{h,k} | \SO{A} + E_{h+1,k})}
	] \quad (\mbox{by Theorem \ref{L1A}})\\
&= F_1 + F_2 + F_3 + F_4 ;
\end{align*}
where
\begin{align*}
F_1 &=\sum_{\substack{ \lambda \in \CMN(n,r-\snorm{A}) } }
	\lambda ^{\bs{j}}
	\sum _{ \substack{ 1 \le k \le n \\ k \ne h+1 }}
		{(-1)}^{\SOE{\tilde{a}}_{h,k} }
		(\SEE{(A+\lambda)}_{h+1,k} + 1)
		\phi_{(\SE{A} + \lambda + E_{h+1,k}| \SO{A} - E_{h,k})}  \\
&=
	\sum _{ \substack{ 1 \le k \le n \\ k \ne h+1 }}
		{(-1)}^{\SOE{\tilde{a}}_{h,k} }
		(\SEE{a}_{h+1,k} + 1)
	\sum_{\substack{ \lambda \in \CMN(n,r-\snorm{A}) } }
	\lambda ^{\bs{j}}
		\phi_{(\SE{A} + \lambda + E_{h+1,k}| \SO{A} - E_{h,k})}  \\
&=
	\sum _{ \substack{ 1 \le k \le n \\ k \ne h+1 }}
		{(-1)}^{\SOE{\tilde{a}}_{h,k} }
		(\SEE{a}_{h+1,k} + 1)
	\ABJRS( \SE{A}+ E_{h+1,k}, \SO{A} - E_{h,k}, \bs{ j },  r );  \\
F_2 &=\sum_{\substack{ \lambda \in \CMN(n,r-\snorm{A}) } }
	\lambda ^{\bs{j}}
	{(-1)}^{\SOE{\tilde{a}}_{h,h+1} }
	(\SEE{(A+\lambda)}_{h+1,h+1} + 1)
	\phi_{(\SE{A} + \lambda + E_{h+1,h+1}| \SO{A} - E_{h,h+1})}  \\
&=
    {(-1)}^{\SOE{\tilde{a}}_{h,h+1} }
    \sum_{\substack{ \lambda \in \CMN(n,r-\snorm{A}) } }
	\lambda ^{\bs{j}}  ({\lambda}_{h+1} + 1)
	\phi_{(\SE{A} + \lambda + E_{h+1,h+1}| \SO{A} - E_{h,h+1})} \\
&=
    {(-1)}^{\SOE{\tilde{a}}_{h,h+1} }  \!
    \sum_{k=0}^{j_{h+1}} {(\substack{ j_{h+1} \\ \ \\ k })}
	{(-1)}^{k}
	\ABJRS( \SE{A}, \SO{A}  - E_{h,h+1}, \bs{ j } +(1- k) \bs{\ep}_{h+1}, r )  \!
	\mbox{ (by Lemma  \ref{mulformajrb1.5})};\\
F_3 &=\sum_{\substack{ \lambda \in \CMN(n,r-\snorm{A}) } }
	\lambda ^{\bs{j}}
	\sum _{ \substack{ 1 \le k \le n \\ k \ne h }}
	{(-1)}^{\SOE{\tilde{a}}_{h,k}}
	\phi_{(\SE{A} + \lambda -E_{h,k} | \SO{A} + E_{h+1,k})} \\
&=
	\sum _{ \substack{ 1 \le k \le n \\ k \ne h }}
	{(-1)}^{\SOE{\tilde{a}}_{h,k}}
	\ABJRS( \SE{A} -E_{h,k}, \SO{A} + E_{h+1,k}, \bs{ j },  r ); \\
F_4 &=\sum_{\substack{ \lambda \in \CMN(n,r-\snorm{A}) } }
	\lambda ^{\bs{j}}
	{(-1)}^{\SOE{\tilde{a}}_{h,h}}
	\phi_{(\SE{A} + \lambda -E_{h,h} | \SO{A} + E_{h+1,h})} \\
&=
	{(-1)}^{\SOE{\tilde{a}}_{h,h}}
	\sum_{k=0}^{j_h} {\binom{j_h}{k}}
	\ABJRS( \SE{A}, \SO{A} +  E_{h+1,h}, \bs{ j } - k \bs{\ep}_h, r )
	\quad (\mbox{by Lemma \ref{mulformajrb1.5}}).
\end{align*}
Taking the sum of $F_1, F_2, F_3$ and $ F_4$,   formula {\rm (ii)} is proved.
\end{proof}

{$\sqrt{-1} \in \CC$} is a square root of {$-1$}.
Define the even elements {$1_{\lambda, r}$}, {$1_{r}$},  {$e_{j, r}$}, {$f_{j, r} $}
and odd elements {$h_{\ol{i}, r}$}, {$e_{\ol{j}, r}$}, {$f_{\ol{j}, r} $} as :
\begin{align*}
	&1_{\lambda,r} =  {\phi}_{\lambda| O},  \quad
	1_r =  \sum_{\lambda \in \Lambda(n,r)} 1_{\lambda, r} =\AJRS(\bs{O}, \bs{0},r); \\
	&h_{i,r} = \AJRS(O, \bs{\ep}_i,  r ),  \quad
	h_{\ol{i},r}
		= {\sqrt{-1}} \cdot \ABJRS( O, E_{i,i}, \bs{ 0 },  r ); \\
	&e_{j,r}
		=\ABJRS( E_{j,j+1}, O, \bs{ 0 },  r ), \qquad
	e_{\ol{j},r}
		={\sqrt{-1}} \cdot \ABJRS( O, E_{j,j+1}, \bs{ 0 },  r );   \\
	&f_{j,r}
		=\ABJRS( E_{j+1,j}, O, \bs{ 0 },  r ),   \qquad
	f_{\ol{j},r}
		={\sqrt{-1}} \cdot   \ABJRS( O, E_{j+1,j}, \bs{ 0 },  r ),
\end{align*}
where {$1 \le i \le n, 1 \le j \le n-1$}.

Recall that for any homogeneous elements $x$, $y$, the super bracket product is
\begin{align*}
[x, y] = xy - {(-1)}^{p(x)p(y)}yx,
\end{align*}
where {$p(z) = 0 $} if $z$ is even, and {$p(z) = 1 $} if $z$ is odd.

The set of simple roots of $\Uqn$ is {$\{\alpha_j = \bs{\ep}_j - \bs{\ep}_{j+1}|1\leq j<n\}$ and $(\bs{\ep}_i,\bs{\ep}_j)=\delta_{i,j}.$ For any $\lambda\in \mathbb{N}^n$, we can rewrite $\lambda=\sum^n_{i=1}\lambda_i\bs{\ep}_i$.
\begin{prop}\label{propRel1}
The elements
{$ \{  1_{\lambda,r},  h_{\ol{i},r},  e_{j,r},  e_{\ol{j},r},  f_{j,r},  f_{\ol{j},r}
\where 1 \le i \le n, 1 \le j \le  n-1
\}$}
defined above satisfy the relations (QR1) to (QR6),
which are the same as (QS1'), (QS2'), (QS3), (QS4'), (QS5), (QS6) in \cite[Theorem 4.10]{DW2}.
\begin{align*}
(QR1) \qquad
	&1_{\lambda,r} 1_{\mu,r} = \delta_{\lambda, \mu} 1_{\lambda,r},  \qquad
	\sum_{\lambda \in \Lambda(n,r)} 1_{\lambda,r} = 1, \qquad
	 h_{\ol{i},r} 1_{\lambda,r} = 1_{\lambda,r} h_{\ol{i},r}, \\
	& [ h_{\ol{i}, r}, h_{\ol{j}, r}]
		= \dij	\sum_{\lambda \in \Lambda(n,r)} 2 \lambda_i 1_{\lambda, r},
	 \qquad
	 h_{\ol{i},r} 1_{\lambda,r} = 0 \,\mbox{ if } \lambda_i = 0; \\
(QR2) \qquad
& e_{j, r}1_{\lambda, r}= 1_{\lambda + \alpha_{j}, r} e_{j,r},
\quad e_{\ol{j}, r}1_{\lambda, r}=1_{\lambda + \alpha_{j}, r} e_{\ol{j}, r},  \\
& f_{j, r}1_{\lambda, r} = 1_{\lambda - \alpha_{j}, r} f_{j,r},
\quad f_{\ol{j}, r}1_{\lambda, r}	=	1_{\lambda - \alpha_{j}, r} f_{\ol{j}, r},\\
&1_{\lambda, r} e_{j, r} = e_{j,r} 1_{\lambda - \alpha_{j}, r} ,
\quad 1_{\lambda, r} e_{\ol{j}, r} =  e_{\ol{j}, r} 1_{\lambda - \alpha_{j}, r} ,  \\
& 1_{\lambda, r} f_{j, r} = f_{j,r} 1_{\lambda + \alpha_{j}, r} ,
\quad 1_{\lambda, r} f_{\ol{j}, r} =	f_{\ol{j}, r} 1_{\lambda + \alpha_{j}, r} ,\\
&\mbox{
where
{$ 1_{\lambda + \alpha_{j}, r}  = 0$}  if {$ \lambda + \alpha_{j} \notin \Lambda(n, r) $},
{$ 1_{\lambda - \alpha_{j}, r}= 0 $}  if {$ \lambda - \alpha_{j} \in \Lambda(n, r)$}}; \\
(QR3)  \qquad
&[h_{\ol{i}, r}, e_{j, r} ] = (\bs{\ep}_{i}, \alpha_{j})e_{\ol{j}, r},  \qquad
[h_{\ol{i}, r}, f_{j, r}] = -(\bs{\ep}_{i}, \alpha_{j}) f_{\ol{j}, r}, \\
&[h_{\ol{i}, r}, e_{\ol{j}, r}] =
\left\{
\begin{aligned}
	&e_{j, r}, \qquad & \mbox{ if } i=j \mbox{ or } j+1,  \\
	& 0, \qquad & \mbox{otherwise,}
\end{aligned}
\right.  \\
&[h_{\ol{i}, r}, f_{\ol{j}, r}] =
\left\{
\begin{aligned}
	&f_{j, r}, \qquad  & \mbox{ if } i=j \mbox{ or } j+1,  \\
	& 0, \qquad  & \mbox{otherwise};
\end{aligned}
\right. \\
(QR4) \qquad
& [e_{i,r}, f_{j,r}]
	= \dij \sum_{\lambda \in \Lambda(n,r)} (\lambda_{i} - \lambda_{i+1}) 1_{\lambda,r}, \\
&  [e_{\ol{i},r}, f_{\ol{j},r} ]
	= \dij \sum_{\lambda \in \Lambda(n,r)} (\lambda_{i} + \lambda_{i+1}) 1_{\lambda,r}, \\
& [ e_{i,r},  f_{\ol{j},r} ]  =  \dij (h_{\ol{i},r} - h_{\ol{i+1},r}), \qquad
  [ e_{\ol{i},r}, f_{j,r} ] =  \dij (h_{\ol{i},r} - h_{\ol{i+1},r}); \\
(QR5) \qquad
	&[e_{i, r}, e_{\ol{j}, r}] = [e_{\ol{i}, r}, e_{\ol{j}, r}] = [f_{i, r}, f_{\ol{j}, r}] = [f_{\ol{i}, r}, f_{\ol{j}, r}] = 0, \qquad \mbox{for } |i-j| \ne 1, \\
	&[e_{i, r}, e_{j, r}] = [f_{i, r}, f_{j, r}] = 0, \qquad \mbox{for } |i-j| > 1, \\
	&[e_{i, r}, e_{i+1, r}] = [e_{\ol{i}, r}, e_{\ol{i+1}, r}],
	\qquad [e_{i, r}, e_{\ol{i+1}, r}] = [e_{\ol{i}, r}, e_{i+1, r}],\\
	&[f_{i+1, r}, f_{i, r}] = [f_{\ol{i+1}, r}, f_{\ol{i}, r}],
	\qquad [f_{i+1, r}, f_{\ol{i}, r}] = [f_{\ol{i+1}, r}, f_{i, r}];\\
(QR6) \qquad
	&[e_{i, r}, [e_{i, r}, e_{j, r}]] = [e_{\ol{i}, r}, [e_{i, r}, e_{j, r}]] = 0,\\
	&[f_{i, r}, [f_{i, r}, f_{j, r}]] = [f_{\ol{i}, r}, [f_{i, r}, f_{j, r}]] = 0,
	\qquad \mbox{ for } |i-j| = 1.
\end{align*}
\end{prop}
\begin{proof}
	(QR1)
		Proposition \ref{mulformajrb1} {\rm (iii)} implies
\begin{align*}
	h_{\ol{i}, r}h_{\ol{j}, r} &= {(\sqrt{-1})}^2 \ABJRS( O, E_{i,i}, \bs{ 0 },  r ) \cdot \ABJRS( O, E_{j,j}, \bs{ 0 },  r ) \\
	&= \left\{
\begin{aligned}
&- \ABJRS( O, E_{i,i}+E_{j,j}, \bs{ 0 }, r ), &\mbox{ if } i<j; \\
& \ABJRS( O, O, \bs{\ep}_i, r ), &\mbox{ if } i=j; \\
& \ABJRS( O, E_{i,i}+E_{j,j}, \bs{ 0 }, r ), &\mbox{ if } i>j.
\end{aligned}
\right.
\end{align*}
As a consequence,
	{$h_{\ol{i}, r}h_{\ol{j}, r} + h_{\ol{j}, r}h_{\ol{i}, r} = \dij
	\sum_{\lambda \in \Lambda(n,r)} 2 \lambda_i 1_{\lambda, r} $} is proved.
	The other equations in {\rm(i)} are trivial.

	(QR2) Direct calculation proves it.

	(QR4)
	Only the second equation will be proved. The others  can be proved in a similar way.
	
	If {$|i-j|>1$}, assume {$i<j$}, we have
\begin{align*}
	e_{\ol{i}, r} f_{\ol{j}, r}
		&= {{(\sqrt{-1})}}^2
			\ABJRS( O, E_{i,i+1}, \bs{ 0 },  r )
			\cdot \ABJRS( O, E_{j+1,j}, \bs{ 0 },  r )  \\
		&= - \ABJRS( O, E_{i,i+1} + E_{j+1, j}, \bs{ 0 },  r ),\\
	f_{\ol{j}, r} e_{\ol{i}, r}
		&= {{(\sqrt{-1})}}^2
			\ABJRS( O, E_{j+1,j}, \bs{ 0 },  r )
			\cdot \ABJRS( O, E_{i,i+1}, \bs{ 0 },  r ) \\
		&= \ABJRS( O, E_{i,i+1} + E_{j+1, j}, \bs{ 0 },  r ).
\end{align*}
	If {$|i-j|=1$}, assume {$j = i+1 $},
\begin{align*}
	e_{\ol{i}, r} f_{\ol{i+1}, r}
		&= {{(\sqrt{-1})}}^2
			\ABJRS( O, E_{i,i+1}, \bs{ 0 },  r )
			\cdot \ABJRS( O, E_{i+2,i+1}, \bs{ 0 },  r )  \\
		&= - \ABJRS( O, E_{i,i+1} + E_{i+2, i+1}, \bs{ 0 }, r ), \\
	f_{\ol{i+1}, r} e_{\ol{i}, r}
		&= {{(\sqrt{-1})}}^2
			\ABJRS( O, E_{i+2,i+1}, \bs{ 0 },  r )
			\cdot \ABJRS( O, E_{i,i+1}, \bs{ 0 },  r ) \\
		&= \ABJRS( O, E_{i,i+1} + E_{i+2, i+1}, \bs{ 0 }, r ).
\end{align*}
The assertions in both cases above cause
{$	[ e_{\ol{i}, r}, f_{\ol{j}, r}] = 0 $}.

	If {$i=j$},
\begin{align*}
	e_{\ol{i}, r} f_{\ol{i}, r}
		&= {{(\sqrt{-1})}}^2
			\ABJRS( O, E_{i,i+1}, \bs{ 0 },  r )
			\cdot \ABJRS( O, E_{i+1,i}, \bs{ 0 },  r )  \\
		&= \ABJRS( O, O, \bs{\ep}_i,  r ) + \ABJRS( O, E_{i, i+1} + E_{i+1, i}, \bs{ 0 },  r ), \\
	f_{\ol{i}, r} e_{\ol{i}, r}
		&= {{(\sqrt{-1})}}^2
			\ABJRS( O, E_{i+1,i}, \bs{ 0 },  r )
			\cdot  \ABJRS( O, E_{i,i+1}, \bs{ 0 },  r ) \\
		&= \ABJRS( O, O, \bs{\ep}_{i+1},  r ) - \ABJRS( O, E_{i, i+1} + E_{i+1, i}, \bs{ 0 },  r ),
\end{align*}
which follows
\begin{align*}
	[e_{\ol{i}, r} , f_{\ol{j}, r} ]
		&= \ABJRS( O, O, \bs{\ep}_{i},  r ) + \ABJRS( O, O, \bs{\ep}_{i+1},  r ) \\
		&=\sum_{\lambda \in \Lambda(n,r)} ( \lambda_i+\lambda_{i+1}) 1_{\lambda, r} .
\end{align*}
The proof of relations (QR3), (QR5), (QR6) is similar to the proofs of (QR1), (QR2), (QR4), respectively.
\end{proof}

\section{A New Realization of {$\Uqn$}}\label{sec_realization}
In this section, by using these uniform multiplication formulas in the last section, we will construct a superalgebra $\USn$ presented by a basis and multiplication formulas by generators and then prove that $\USn$ is isomorphic to the universal enveloping algebra $\Uqn$ (over $\mathbb C$). 


By mimicking the algebraic limit process from a polynomial algebra $k[x]$ to its formal power series algebra $k[[x]]$, where $k$ is a field,
we first define a limit algebra $\Qqsn$ of queer Schur superalgebra {$\Qnr$} by taking a direct product. We then extract a subalgebra $\USn$ from $\Qqsn$ which has a natural superalgebra structure. We then prove that $\USn$ is isomorphic to $\Uqn$.

Set
\begin{align*}
    &\Qqsn = \prod_{r \ge 0 } \Qnr,  \\
    &\AJS(A, \bs{j}) = \sum_{r \ge 0 } \AJRS(A, j, r) \in \Qqsn,
        \, \mbox{ for any } A \in \MNZNS(n), \bs{j} \in \NN^n, \\
    &\LS = \{  \AJS(A, \bs{j})  \where   A \in \MNZNS(n), \bs{j} \in \NN^n \}, \\
    &\USn = \tspan\{\AJS(A, \bs{j})   \where  A \in \MNZNS(n), \bs{j} \in \NN^n \}.
\end{align*}

Choose a subset of {$\USn$},
\begin{align*}
	\fsG &= \{ \    \AJS(\bs{O}, \bs{\ep}_i), 	\  \ABJS( O, E_{i,i},\bs{ 0 }), \  \ABJS( E_{j,j+1}, O,\bs{ 0 }), 	\  \ABJS( O, E_{j,j+1},\bs{ 0 }),	 \\
	& \qquad  \ABJS( E_{j+1,j}, O,\bs{ 0 }),	\   \ABJS( O, E_{j+1,j},\bs{ 0 })
	\where   1\leq i\leq n, 1\leq j \leq n-1
	\}.
\end{align*}

In the rest of this section, we will show
$\USn$  is a subsuperalgebra of {$\Qqsn$}  generated by {$\fsG$},
and isomorphic to {$\Uqn$}.

\begin{lem}\label{U(n) basis}
        The set {$\LS=\{  \AJS(A, \bs{j}) \where A \in \MNZNS(n), \bs{j} \in \NN^n\}$}
	forms a basis of {$\USn$}.
\end{lem}
\begin{proof}
It can be proved in the similar way as  \cite[Theorem 6.3.4]{DDF}.
\end{proof}

\begin{prop}\label{A[j]}
The multiplication formulas of
\begin{align*}
& {\rm (i)}  \quad \AJS(\bs{O}, \bs{\ep}_i) \cdot \ABJS( \SE{A}, \SO{A},\bs{ j }), \quad
\ABJS(O, E_{h,h}, \bs{0}) \cdot \ABJS(\SE{A}, \SO{A}, \bs{j}), \\
& {\rm (ii)} \quad \ABJS( E_{h, h+1}, O,\bs{ 0 }) \cdot \ABJS( \SE{A}, \SO{A},\bs{ j }),\quad
 \ABJS( O, E_{h, h+1},\bs{ 0 }) \cdot \ABJS( \SE{A}, \SO{A},\bs{ j }), \\
& {\rm (iii)}  \quad \ABJS( E_{h+1, h}, O,\bs{ 0 }) \cdot \ABJS( \SE{A}, \SO{A},\bs{ j }), \quad
 \ABJS( O, E_{h+1, h},\bs{ 0 }) \cdot \ABJS( \SE{A}, \SO{A},\bs{ j }).
\end{align*}
have the same form with those in
Proposition \ref{mulformajrb1}, \ref{mulformajrb2} and \ref{mulformajrb3},
by replacing all $\AJRS(M, \bs{j}, r)$ with $\AJS(M, \bs{j})$.
\end{prop}
\begin{proof}
According to the definition, $\AJS(A, \bs{j}) = \sum_{r \ge 0 } \AJRS(A, \bs{j}, r)$  for all {$A \in \MNZNS(n)$}.
Observing the multiplication formulas in Proposition \ref{mulformajrb1}, \ref{mulformajrb2} and \ref{mulformajrb3},
the coefficients are independent of $r$.
In the other hand,
since  {$\AJRS(A, \bs{j}, r_1) \AJRS(B, \bs{j}', r_2)=0$}  when  {$r_1 \neq r_2$},
we have
{$\AJS(A, \bs{j}) \AJS(B, \bs{j}')=\sum_{r\geq 0}\AJRS(A, \bs{j}, r)\AJRS(B, \bs{j}', r)$}.
As a consequence, replacing all $\AJRS(M, \bs{j}, r)$ with $\AJS(M, \bs{j})$ in these formulas,
the multiplication formulas of the above elements are derived.
\end{proof}

As a consequence of Proposition \ref{A[j]},  there is an algebra epimorphism
\begin{align}
	\bs{\rho}_r: \USn  \longrightarrow \Qnr,
	\qquad
	\AJS(A, \bs{j}) \mapsto \AJRS(A, \bs{j}, r). \label{def_rho_proj}
\end{align}

For ${\bs{j}},{\bs{j}'}\in \NN^n$, denote $\bs{j}\leq \bs{j}'$ if and only if $j_k\leq j'_k$ for all $k\in[1,n]$.

\begin{cor}\label{OjA0}
For $A\in \MNZN(n)$ and ${\bs j}=(j_1,j_2,\cdots, j_n)\in \NN^n$, we have
\begin{align*}
  \AJS(\bs{O}, \bs{j}) \cdot \ABJS( \SE{A}, \SO{A},\bs{ 0 })
&= \ABJS( \SE{A}, \SO{A},\bs{ j })
		 + \sum_{\bs{k}\in \NN^n,\bs{k}<\bs{j}}
			f_{\bs{j},\bs{k}} \ABJS( \SE{A}, \SO{A},\bs{ k }),
\end{align*}
where {$f_{\bs{j},\bs{k}}  \in \NN$}.
\end{cor}

\begin{proof}
According to Proposition \ref{A[j]}{\rm (i)}, we have
\begin{align*}
\AJS(\bs{O}, \bs{j})= \prod_{i=1}^n {(\AJS(\bs{O},\bs{\ep}_i) )}^{{j}_i},
\end{align*}
and
\begin{align*}
\AJS( \bs{O}, \bs{\ep}_i )^{ {j}_i } \cdot \ABJS( \SE{A}, \SO{A},\bs{l})
&= \ABJS( \SE{A}, \SO{A}, \bs{l} + {j}_i \bs{\ep}_i  )
 + \sum_{\substack{\bs{k}\in \NN^n, \\ \bs{l} \le  \bs{k}<\bs{l} + {j}_i \bs{\ep}_i }}
			f_{\bs{l},\bs{k}} \ABJS( \SE{A}, \SO{A},\bs{ k }).
\end{align*}
As a consequence,
multiplying {$\AJS( \bs{O}, \bs{\ep}_i )^{ {j}_i } $} on the left side of
{$\ABJS( \SE{A}, \SO{A},\bs{0})$} one by one, where $i$ varies from $n$ to $1$,
the result is then proved
\begin{align*}
\AJS(\bs{O}, \bs{j}) \cdot \ABJS( \SE{A}, \SO{A},\bs{ 0 })
&=\prod_{i=1}^n {(\AJS(\bs{O},\bs{\ep}_i) )}^{\bs{j}_i}  \cdot \ABJS( \SE{A}, \SO{A},\bs{ 0 }) \\
&=\prod_{i=1}^{n-1} {(\AJS(\bs{O},\bs{\ep}_i) )}^{\bs{j}_i}
		\cdot \{ {(\AJS(\bs{O},\bs{\ep}_n) )}^{\bs{j}_n}  \ABJS( \SE{A}, \SO{A},\bs{ 0 }) \} \\
& \qquad \mbox{( $i$ varies from $n$ to $1$) }\\
&= \ABJS( \SE{A}, \SO{A},\bs{ j })
		 + \sum_{\bs{k}\in \NN^n,\bs{k}<\bs{j}}
			f_{\bs{j},\bs{k}} \ABJS( \SE{A}, \SO{A},\bs{ k }).
\end{align*}
%
%
%
\end{proof}

For {$M = (m_{i,j}) \in \MN(n)$}, set
\begin{align*}
& M^+= (m^+_{i,j}), \mbox{ with } m^+_{i,j} = m_{i,j} \mbox{ for all }  i < j,
			\mbox{ and } m^+_{i,j} =0 \mbox{ otherwise}; \\
& M^-= (m^-_{i,j}), \mbox{ with } m^-_{i,j} = m_{i,j} \mbox{ for all }  i > j,
			\mbox{ and } m^-_{i,j} =0 \mbox{ otherwise} ;\\
& M^{\ge 0}= (m^{\ge 0}_{i,j}), \mbox{ with } m^{\ge 0}_{i,j} = m_{i,j} \mbox{ for all }  i \le j,
			\mbox{ and } m^{\ge 0}_{i,j} =0 \mbox{ otherwise}.
\end{align*}
Denote {$M^\pm = M^+ + M^- $}.

For any given {$A = (\SUP{A}) \in \MNZN(n)$},  denote
\begin{align*}
& A^+ =  ( {(\SE{A})}^+ |  {(\SO{A})}^+) , \qquad
  A^{\ge 0} =  ( {(\SE{A})}^{\ge 0} |  {(\SO{A})}^{\ge 0}) ,\qquad
  A^{\pm} = ( {(\SE{A})}^{\pm} |  \SO{A}).
\end{align*}
Recall that $\ABSUM{A}=\SE{A}+\SO{A}$ and $\ABSUM{B}=\SE{B} + \SO{B}$.
Set {$\Tr(A) = \{ i | \SOE{a}_{i,i} = 1 \}$}.

In order to get the triangular multiplication formula, which is a key to the realization problem,
we need to define a partial order for the set {$\MNZN(n)$}.

Similar to \cite{DDPW}, for {$A=(a_{i,j}), B=(b_{i,j}) \in \MN(n)$}, define the order {$\preceq$} as
\begin{displaymath}\label{preorder}
	B \preceq A \iff
\left\{
\begin{aligned}
	&\sum_{i \le s, j\ge t } b_{i,j} \le \sum_{i \le s, j\ge t } a_{i,j}, & \mbox{ for all } s<t; \\
	&\sum_{i \ge s, j\le t } b_{i,j} \le \sum_{i \ge s, j\le t } a_{i,j}, & \mbox{ for all } s>t; \\
\end{aligned}
\right.
\end{displaymath}
We say {$B \prec A $} if {$B \preceq A$} and {$B^{\pm} \ne A^{\pm}$}.
Specially, if $A$ and $B$ are equal under this order, denote $A\simeq B$,
and in this case we have  {$A^{\pm} = B^{\pm}$}.

Notice that
{$\ABSUM{A}^{\pm}$} denotes for {${(\SE{A} + \SO{A})}^{\pm}$},
but not for {$\SE{(A^{\pm})} + \SO{(A^{\pm})} = {(\SE{A})}^{\pm} + \SO{A}$}.

For any {$A=(\SUP{A}), B=(\SUP{B}) \in \MNZNS(n)$},
abusing the notations, define {$B \preceq A $} when {$  \ABSUM{B} \preceq \ABSUM{A}$},
and specially,   we say {$B \prec A$} if they satisfy one of the following conditions:
\begin{align*}
&(a) \quad \ABSUM{B} \prec \ABSUM{A},  \mbox{ or } \\
&(b) \quad \ABSUM{B}^{\pm} = \ABSUM{A}^{\pm}  \mbox{ and }  \Tr(B) \subset \Tr(A) .
\end{align*}

Set {$\fsP = [1,n] \times [1,n]$},
and its subsets:
\begin{align*}
	&\fsP^+ = \{(i,j) \in [1,n] \times [1,n] \where  i<j \}, \\
	&\fsP^0 = \{(i,j) \in [1,n] \times [1,n] \where i=j \}, \\
	&\fsP^- = \{(i,j) \in [1,n] \times [1,n] \where i>j \}.
\end{align*}

Define partial order  {$\le$}  over {$ \fsP $} as follows:
\begin{align*}
&(i_1, j_1) \le (i_2, j_2) \le (i_3, j_3),  \qquad \forall
	(i_1, j_1) \in \fsP^-,  (i_2, j_2) \in \fsP^0,  (i_3, j_3) \in \fsP^+ ; \\
&(i,j) \le (i',j') \iff
	\left\{
	\begin{aligned}
		& j > j', \mbox{ or }\\
		& j=j', i > i',  \\
	\end{aligned}
	\right.
	\qquad \forall (i,j), (i', j') \in \fsP^+ ; \\
&(i,i) \le (i',i') \iff i < i',
	\qquad \forall (i,i), (i', i') \in \fsP^0; \\
&(i,j) \le (i',j') \iff
	\left\{
	\begin{aligned}
		& j < j', \mbox{ or }\\
		& j=j', i < i', \\
	\end{aligned}
	\right.
	\qquad \forall (i,j), (i', i') \in \fsP^-;
\end{align*}
Denote {$(i,j) < (i', j')$} if  {$(i,j) \le (i', j')$} and {$(i,j) \ne (i', j')$}.

For example, this partial order over {$ [1,3] \times [1,3] $} is:
\begin{align*}
(2,1)  \le  (3,1)  \le  (3,2) \le (1,1)  \le  (2,2)  \le  (3,3)  \le  (2,3)  \le  (1,3)  \le  (1,2).
\end{align*}
This order is deduced from the orders $\leq_1$ and $\leq_2$ in  \cite[(8.2.1)]{DG}.

Given {$A=(\SUP{A})  \in \MNZN(n),  $} for {$ (i,j) \in \fsP^+$}, set
\begin{align*}
	& U^{A, \bs{0}}_{(i,j)} =
			\ABJS( O, \SOE{a}_{i,j} E_{i,i+1},\bs{ 0 })  \ABJS( \SEE{a}_{i,j} E_{i,i+1}, O,\bs{ 0 })
			\cdot \prod_{i< h<j}^{(h,h+1)_{<}} \ABJS( {a}_{i,j} E_{h,h+1}, O,\bs{ 0 });
\end{align*}
for {$ (i,j) \in \fsP^-$}, set
\begin{align*}
	& L^{A, \bs{0}}_{(i,j)} =
			 \ABJS( O, \SOE{a}_{i,j} E_{i,i-1},\bs{ 0 }) \ABJS( \SEE{a}_{i,j} E_{i,i-1}, O,\bs{ 0 })
			\cdot \prod_{j \leq h\leq i-2}^{(h+1,h)_<} \ABJS( {a}_{i,j} E_{h+1,h}, O,\bs{ 0 });
\end{align*}
for {$ (i,i) \in \fsP^0$}, set
\begin{align*}
	& D^{A, \bs{0}}_{(i,i)} =  \ABJS( O, \SOE{a}_{i,i} E_{i,i},\bs{ 0 }) .
\end{align*}

We shall apply Proposition  \ref{A[j]}
to show the important theorem followed, which is called the triangular multiplication formula.

\begin{thm}\label{triang2}
	Suppose {$A=(\SUP{A})  \in \MNZNS(n)$}, then
\begin{align*}
	& \prod_{(i,j) \in \fsP^- }^{\le} L^{A, \bs{0}}_{(i,j)}
		\prod_{(i,i) \in \fsP^0 }^{\le} D^{A, \bs{0}}_{(i,i)}
		\prod_{(i,j) \in \fsP^+ }^{\le} U^{A, \bs{0}}_{(i,j)}
	 = \pm\AJS(A,\bs{0})
	 + \sum_{\substack{
	 		B \in \MNZNS(n) \\
	 		B \prec A \\
	 		{\bs{j}} \in {\NN}^n
	} } g_{B,A, \bs{j} }\AJS(B, {\bs{j}}),
\end{align*}
where $ g_{B,A,\bs{j}} \in \QQ$, and there are only finite numbers of non-zero {$g_{B,A,\bs{j}}$}.
%
%
\end{thm}

\begin{proof}
	According to  Proposition \ref{A[j]} {\rm(ii)} and  {\rm(iii)}, for any {$k \in \NN$}, we have
\begin{equation}
\begin{aligned}\label{devpow}
	&  \ABJS( k E_{h, h+1}, O,\bs{ 0 }) = \frac{1}{k!}{( \ABJS( E_{h, h+1}, O,\bs{ 0 }))}^k,  \\
	&  \ABJS( k E_{h+1, h}, O,\bs{ 0 }) = \frac{1}{k!}{(\ABJS( E_{h+1, h}, O,\bs{ 0 }))} ^k.
\end{aligned}
\end{equation}

Recall that  $\ABSUM{A}=\SE{A} + \SO{A}$.
According to Proposition \ref{A[j]},
the expressions of  $\ABJS( E_{h, h+1}, O,\bs{ 0 }) \cdot \AJS( A,\bs{ j })$
and $\ABJS( E_{h, h+1}, O,\bs{ 0 }) \cdot \AJS( A,\bs{ j })$ have the form
\begin{align*}
\sum_{\substack{ X \in \MNZNS(n) \\{\bs{j}}' \in {\NN}^n }} f_{X,A, \bs{j}'} \AJS( X,\bs{j}' )
\end{align*}
where $f_{X,A, \bs{j}'} \in \QQ$ and
$\ABSUM{X}=\ABSUM{A}+E_{h,k}-E_{h+1,k}$, $\ABSUM{A}+E_{h,h+1}$ or $\ABSUM{A}-E_{h+1,h}$ with $k\neq h,h+1$.
In the other hand,  the expressions of
$\ABJS( E_{h+1, h}, O,\bs{ 0 }) \cdot\AJS( A,\bs{ j })$,
$\ABJS( O, E_{h+1, h},\bs{ 0 }) \cdot \AJS( A,\bs{ j })$
are of  the form
\begin{align*}
\sum_{\substack{ X \in \MNZNS(n) \\{\bs{j}}' \in {\NN}^n }}  f'_{X,A, \bs{j}'} \AJS( X,\bs{j}' )
\end{align*}
where $f'_{X,A, \bs{j}'} \in \QQ $
 and $\ABSUM{X}=\ABSUM{A}-E_{h,k}+E_{h+1,k}$, $\ABSUM{A}-E_{h,h+1}$ or $\ABSUM{A}+E_{h+1,h}$ with $k\neq h,h+1$.
These formulas are similar to \cite[Proposition 6.6]{DG}.


The result is proved  in three steps.

{\textbf{Step 1.}}
Similar to  \cite[Theorem 9.1]{DG}, applying Proposition \ref{A[j]} {\rm(ii)},
\begin{align*}
\prod_{(i,j) \in \fsP^+ }^{\le}  U^{A, \bs{0}}_{(i,j)}
&=	 \pm \AJS( {A}^+, \bs{ 0 })
	+	\sum_{\substack{
				C''\in \MNZNS(n) \\
				\ABSUM{C''} \prec  {\ABSUM{A}^+},
				\bs{j}'  \in \NN^n
			} }
			f_{C'',A, \bs{j}'} \AJS( {C''}, \bs{j}' )
\end{align*}
where $f_{C'',A, \bs{j}'} \in \QQ $, $\SE{C''}$ and $\SO{C''}$   are  strictly upper triangle matrices.

{\textbf{Step 2.}}
For any {$M \in \MNZN(n)$}, Proposition \ref{A[j]}{\rm (i)} implies
\begin{align*}
	\ABJS( O, E_{h,h},\bs{ 0 }) \cdot \AJS({M},\bs{ j })
		=\sum_{\substack{ X \in \MNZNS(n) \\{\bs{j}}' \in {\NN}^n }}
			f_{X,M, \bs{j}' }\AJS({X}, \bs{j}')
\end{align*}
where $f_{X,M, \bs{j}'} \in \QQ$ and {$\ABSUM{X} = \ABSUM{M}$},
{$\ABSUM{M}+E_{h,h}$} or {$\ABSUM{M}-E_{h,h}$}.
Direct computation shows $\ABSUM{X}^{\pm} = \ABSUM{M}^{\pm}$.
As a consequence,
\begin{equation}\label{dandu}
\begin{aligned}
&\prod_{(i,i) \in \fsP^0 }^{\le} D^{A, \bs{0}}_{(i,i)}
		\prod_{(i,j) \in \fsP^+ }^{\le} U^{A, \bs{0}}_{(i,j)}   \\
&=
		 \pm \prod_{(i,i) \in \fsP^0 }^{\le} 	
		D^{A, \bs{0}}_{(i,i)} \cdot  \AJS( {A}^+, \bs{ 0 })
		+
		\prod_{(i,i) \in \fsP^0 }^{\le} 	
		\sum_{\substack{
				C''\in \MNZNS(n) \\
				\ABSUM{C''} \prec  {\ABSUM{A}}^+,
				\bs{j}'  \in \NN^n
			} }	 D^{A, \bs{0}}_{(i,i)}	
				\cdot f_{C'',A^+, \bs{j}'} \AJS( {C''}, \bs{j}')   \\
&=	\pm\AJS( {A}^{\ge 0}, \bs{ 0 })
		+ \sum_{\substack{
	 				B' \in \MNZNS(n)\\
					{\ABSUM{B'}}^{\pm} = {\ABSUM{A}}^+\\
	 				\Tr(B') \subset \Tr(A)\\
	 				\bs{j}' \in {\NN}^n
		} }   f_{B',A, \bs{j}' } \AJS(B', \bs{j}')
	 + \sum_{\substack{
				C' \in \MNZNS(n)\\
				\ABSUM{C'} \prec  \ABSUM{A}^+ \\
				\bs{j}''  \in \NN^n
			} }  f_{C',A, \bs{j}''}  \AJS( {C'}, \bs{j}''),
\end{aligned}
\end{equation}
where $f_{B',A, \bs{j}'}, f_{C',A, \bs{j}'' }\in  \QQ$
and $\SE{B}$, $\SE{C}$, $\SO{B}$, $\SO{C}$  are  upper triangular matrices.

Next  we consider
\begin{align*}
	\prod_{(i,j) \in \fsP^- }^{\le} L^{A, \bs{0}}_{(i,j)}\AJS( {A}^{\ge 0}, \bs{ 0 }),
	\prod_{(i,j) \in \fsP^- }^{\le} L^{A, \bs{0}}_{(i,j)}\AJS( {B'}, \bs{j}'),
	\prod_{(i,j) \in \fsP^- }^{\le} L^{A, \bs{0}}_{(i,j)}\AJS( {C'}, \bs{j}'').
\end{align*}
{\textbf{Step 3.}}
Similar to  \cite[Theorem 9.1]{DG}, applying Proposition \ref{A[j]} {\rm(iii)},
\begin{equation}\label{pftri1}
\begin{aligned}
&\prod_{(i,j) \in \fsP^- }^{\le} L^{A, \bs{0}}_{(i,j)} \AJS( {A}^{\ge 0}, \bs{ 0 })    \\
&=	\pm \AJS( {A}, \bs{ 0 })
	 + \sum_{\substack{
			B\in \MNZNS(n) \\
			{\ABSUM{B}}^{\pm}  =  {\ABSUM{A}}^{\pm}   \\
			B \ne \pm A   \\
			\Tr(B) \subset \Tr(A) \\
			\bs{j}' \in {\NN}^n
		}} f_{B,A, \bs{j}' }\AJS( B,  \bs{ j }' )
	+ \sum_{\substack{
			C \in \MNZNS(n)  \\
			\ABSUM{C} \prec \ABSUM{A} \\
			\bs{j}'' \in {\NN}^n
		}} f_{C,A, \bs{j}''}\AJS(  C, \bs{ j }'' ),
\end{aligned}
\end{equation}
where $f_{B,A, \bs{j}'}, f_{C,A, \bs{j}''}\in  \QQ$.

For every {$X = B'$} or {$X=C'$} appearring in \eqref{dandu},
\begin{equation}\label{pftri2}
\begin{aligned}
&\prod_{(i,j) \in \fsP^- }^{\le} L^{A, \bs{0}}_{(i,j)}\ABJS( \SE{X}, \SO{X},\bs{ 0 })
=
	\! \sum_{\substack{
			B\in \MNZNS(n) \\
			{\ABSUM{B}}^{\pm}  =  {\ABSUM{A}}^{\pm}   \\
			B \ne \pm A   \\
			\Tr(B) \subset \Tr(A) \\
			\bs{j}' \in {\NN}^n
		}}  	 f_{B,X, \bs{j}'}\AJS( B,\bs{j}' )
	 \!+	\!\sum_{\substack{
				C\in \MNZNS(n) \\
				\ABSUM{C}\prec \ABSUM{A}\\
				\bs{j}'' \in {\NN}^n
		}}  	 f_{C,X, \bs{j}''}\AJS( C,\bs{ j }'' ),
\end{aligned}
\end{equation}
where $f_{B,X, \bs{j}'}, f_{C,X, \bs{j}''}\in \QQ$.

Summarize the equations \eqref{pftri1} and \eqref{pftri2},  then we have
\begin{align*}
	& \prod_{(i,j) \in \fsP^- }^{\le} L^{A, \bs{0}}_{(i,j)}
		\prod_{(i,i) \in \fsP^0 }^{\le} D^{A, \bs{0}}_{(i,i)}
		\prod_{(i,j) \in \fsP^+ }^{\le} U^{A, \bs{0}}_{(i,j)}  \nonumber \\
	 &= \pm\AJS(A,\bs{0})
	 + \sum_{\substack{
	 				B \in \MNZNS(n) \\
					\ABSUM{B}^{\pm} = \ABSUM{A}^{\pm}  \\
	 				B \ne   A  \\
	 				\Tr(B) \subset \Tr(A)\\
	 				\bs{j}' \in {\NN}^n
		} } 	g_{B,A, \bs{j}'}\AJS(B, {\bs{j}'})
	 + \sum_{\substack{
	 		C \in \MNZNS(n) \\
	 		\ABSUM{C} \prec \ABSUM{A}\\
	 		\bs{j}'' \in {\NN}^n
	} } g_{C,A, \bs{j}'' }\AJS(C, {\bs{j}''})  \nonumber \\
	&= \pm\AJS(A,\bs{0})
	 + \sum_{\substack{
	 		B \in \MNZNS(n) \\
	 		B \prec A \\
	 		\bs{j}  \in {\NN}^n
	} } g_{B,A, \bs{j} }\AJS(B, {\bs{j}}),
\end{align*}
where $ g_{B,A,\bs{j}' },  g_{C,A, \bs{j}'' }, g_{B,A,\bs{j} }\in \QQ$.
%
\end{proof}
\begin{thm}\label{U(n)Alg}
The superalgebra $\USn$ is generated by
\begin{align*}
	\fsG &= \{ \    \AJS(\bs{O}, \bs{\ep}_i), 	\  \ABJS( O, E_{i,i},\bs{ 0 }), \  \ABJS( E_{j,j+1}, O,\bs{ 0 }), 	\  \ABJS( O, E_{j,j+1},\bs{ 0 }),	 \\
	& \qquad  \ABJS( E_{j+1,j}, O,\bs{ 0 }),	\   \ABJS( O, E_{j+1,j},\bs{ 0 })
	\where   1\leq i\leq n, 1\leq j \leq n-1
	\}.
\end{align*}
\end{thm}
\begin{proof}
Set the subsuperalgebra of {$\Qqsn$} generated by {$\fsG$}
to be $\mathfrak{u}(n)$.
According to Proposition  \ref{A[j]}, $\mathfrak{u}(n)$ is a subspace of $\USn$.


According the formula in Theorem \ref{triang2},
by induction  about the order `$\preceq$' over $\MNZNS(n)$, for any $A\in\MNZNS(n)$,
{$\AJS(A, \bs{0})$} can be generated from {$\fsG$}.

For any  {$\bs{j} \in \NN^n$},
Corollary \ref{OjA0} shows {$\AJS(A, \bs{j})$}
can be derived from $\AJS(A,\bs{0})$ and $\AJS(\bs{O},\bs{\ep}_i)$ with $1\leq i\leq n$.
Then $\USn$ is a subspace of $\mathfrak{u}(n)$,
hence $\USn = \mathfrak{u}(n)$ is a superalgebra generated by  $\fsG$.
\end{proof}



Next, we want to prove that $\USn$ is isomorphic to $\Uqn$.
The universal enveloping algebra $\Uqn$ of $\qn$
 is a superalgebra defined by generators and relations in  \cite{DW2} and \cite{LS}.

\begin{defn}
{$\Uqn$} is generated by
{$\{ h_{i}, h_{\ol{i}}, e_j, e_{\ol{j}}, f_{j}, f_{\ol{j}} \where 1 \le i \le n, 1 \le j \le n-1 \}$}
subject to the following relations:
\begin{align*}
(QS1) \qquad
&	[h_{i}, h_{j}] =0, \qquad [h_{i}, h_{\ol{j}}] = 0,
	\qquad [h_{\ol{i}}, h_{\ol{j}}] = \dij 2 h_{i}; \\
(QS2) \qquad
	&[h_{i}, e_{j}] = (\bs{\ep}_{i}, \alpha_j) e_{j}, \qquad
	[f_{i}, e_{\ol{j}}] = (\bs{\ep}_{i}, \alpha_j) e_{\ol{j}}, \\
	&[h_{i}, f_{j}] = -(\bs{\ep}_{i}, \alpha_j) f_{j}, \qquad
	[h_{i}, f_{\ol{j}}] = -(\bs{\ep}_{i}, \alpha_j) f_{\ol{j}}, \\
(QS3)  \qquad
&[h_{\ol{i}}, e_{j} ] = (\bs{\ep}_{i}, \alpha_j)	e_{\ol{j}},  \qquad
[h_{\ol{i}}, f_{j}] = -(\bs{\ep}_{i}, \alpha_j) f_{\ol{j}}, \\
&[h_{\ol{i}}, e_{\ol{j}}] =
\left\{
\begin{aligned}
	&e_{j} \qquad \mbox{ if } i=j \mbox{ or } j+1,  \\
	& 0 \qquad \mbox{otherwise,}
\end{aligned}
\right.  \\
&[h_{\ol{i}}, f_{\ol{j}}] =
\left\{
\begin{aligned}
	&f_{j} \qquad \mbox{ if } i=j \mbox{ or } j+1,  \\
	& 0 \qquad \mbox{otherwise,}
\end{aligned}
\right. \\
(QS4)  \qquad
	&[e_{i}, f_{j}] = \dij (h_{i} - h_{i+1}), \qquad
	[e_{\ol{i}}, f_{\ol{j}}] = \dij (h_{i} + h_{i+1}), \\
	&[e_{\ol{i}}, f_{j}] = \dij (h_{\ol{i}} - h_{\ol{i+1}}), \qquad
	[e_{i}, f_{\ol{j}}] = \dij (h_{\ol{i}} - h_{\ol{i+1}}), \\
(QS5) \qquad
	&[e_{i}, e_{\ol{j}}] = [e_{\ol{i}}, e_{\ol{j}}] = [f_{i}, f_{\ol{j}}] = [f_{\ol{i}}, f_{\ol{j}}] = 0, \qquad \mbox{for } |i-j| \ne 1, \\
	&[e_{i}, e_{j}] = [f_{i}, f_{j}] = 0, \qquad \mbox{for } |i-j| > 1, \\
	&[e_{i}, e_{i+1}] = [e_{\ol{i}}, e_{\ol{i+1}}],
	\qquad [e_{i}, e_{\ol{i+1}}] = [e_{\ol{i}}, e_{i+1}],\\
	&[f_{i+1}, f_{i}] = [f_{\ol{i+1}}, f_{\ol{i}}],
	\qquad [f_{i+1}, f_{\ol{i}}] = [f_{\ol{i+1}}, f_{i}]; \\
(QS6) \qquad
	&[e_{i}, [e_{i}, e_{j}]] = [e_{\ol{i}}, [e_{i}, e_{j}]] = 0,\\
	&[f_{i}, [f_{i}, f_{j}]] = [f_{\ol{i}}, [f_{i}, f_{j}]] = 0,
	\qquad \mbox{ for } |i-j| = 1.
\end{align*}
\end{defn}

Denote
\begin{align*}
	&H_{i}= \AJS(\bs{O}, \bs{\ep}_i), \qquad \qquad
	H_{\ol{i}}=  {\sqrt{-1}} \ABJS( O, E_{i,i},\bs{ 0 }), \quad 1\leq i\leq n,\\
	&E_{j}=\ABJS( E_{j,j+1}, O,\bs{ 0 }), \quad
	E_{\ol{j}}={\sqrt{-1}} \ABJS( O, E_{j,j+1},\bs{ 0 }),\\
	&F_{j}=\ABJS( E_{j+1,j}, O,\bs{ 0 }), \quad
	F_{\ol{j}} ={\sqrt{-1}} \ABJS( O, E_{j+1,j},\bs{ 0 }), \quad 1\leq j<n.
\end{align*}

\begin{thm}\label{realuqn}
There is a superalgebra isomorphism
	{$ \pi:\Uqn\longrightarrow \USn $}
mapping
\begin{align*}
 &h_i\mapsto H_i,	\qquad h_{\bar{i}}\mapsto H_{\bar{i}}, \\
 &e_j\mapsto E_j,	\qquad e_{\bar{j}}\mapsto E_{\bar{j}},\\
 &f_j\mapsto F_j,	\qquad f_{\bar{j}}\mapsto F_{\bar{j}},
\end{align*}
for all $1 \leq i\leq n, 1\leq j<n$.
\end{thm}
\begin{proof}
To prove {$\pi$} is a superalgebra homomorphism,
it is enough to show
{$H_i$}, {$H_{\ol{i}}$}, {$E_j$}, {$E_{\ol{j}}$}, {$F_j$}, {$F_{\ol{j}}$}
 satisfy the relations in (QS1)-(QS6).

Applying Proposition \ref{propRel1} and \ref{A[j]}, we  will show the relations.

The definition shows
\begin{align*}
	&H_{i} =  \sum_{r \ge 0 } ( \sum_{\lambda \in \Lambda(n,r)} \lambda_i 1_{\lambda, r} )
		= \sum_{r \ge 0 } \AJRS(\bs{O}, \bs{\ep}_i, r), \quad
	H_{\ol{i}} = \sum_{r \ge 0 } h_{\ol{i},r}
		=\sum_{r \ge 0 } {{(\sqrt{-1})}} \cdot \ABJRS( O, E_{i,i}, \bs{ 0 },  r ), \\
	&E_{j} = \sum_{r \ge 0 } e_{j,r}, \quad
	E_{\ol{j}} = \sum_{r \ge 0 } e_{\ol{j},r}, \quad
	F_{j} = \sum_{r \ge 0 } f_{j,r}, \quad
	F_{\ol{j}} = \sum_{r \ge 0 } f_{\ol{j},r}.
\end{align*}
Since the formulas for queer Schur superalgebra in Proposition \ref{propRel1}(QR3), (QR5) and (QR6) are independent with the diagonal entries of $A$, through
folding the $r$-level to all $r\geq 0$,
we can derive relations (QS3), (QS5), (QS6).

For (QS1), as
\begin{align*}
	H_i H_j
		=\sum_{r \ge 0 } \AJRS(\bs{O}, \bs{\ep}_i, r) \sum_{r \ge 0 } \AJRS(\bs{O}, \bs{\ep}_j, r)
		=\sum_{r \ge 0 } \AJRS(\bs{O}, \bs{\ep}_i + \bs{\ep}_j, r)
		= H_j H_i.
\end{align*}
	which implies
	{$ [H_i, H_j ] = 0$}.
And
\begin{align*}
	H_{i} H_{\ol{j}} &= \sum_{r \ge 0 }  \AJRS(\bs{O}, \bs{\ep}_i, r)
			\cdot \sum_{r \ge 0 } {{(\sqrt{-1})}} \cdot \ABJRS( O, E_{j,j}, \bs{ 0 },  r )
		= {{(\sqrt{-1})}} \sum_{r \ge 0 } \ABJRS( O, E_{j,j}, \bs{\ep}_i,  r )
		= H_{\ol{j}} H_{i}
\end{align*}
implies
{$	[H_i, H_{\ol{j}}] = 0 $}.

From Proposition \ref{propRel1} (QR1), and
		{$	\sum_{\lambda \in \Lambda(n,r)} 2 \lambda_i 1_{\lambda, r}
		= 2 \AJRS(\bs{O}, \bs{\ep}_i, r)	$},
		it follows
\begin{align*}
	[H_{\ol{i}}, H_{\ol{j}}] &= H_{\ol{i}} H_{\ol{j}} +  H_{\ol{j}} H_{\ol{i}}
		= \sum_{r \ge 0 } h_{\ol{i},r} \sum_{r \ge 0 } h_{\ol{j},r} + \sum_{r \ge 0 } h_{\ol{j},r}  \sum_{r \ge 0 } h_{\ol{i},r} \\
		&= \sum_{r \ge 0 } (h_{\ol{i},r} h_{\ol{j},r} + h_{\ol{j},r} h_{\ol{i},r})
		=\sum_{r \ge 0 }  ( \dij 2 \AJRS(\bs{O}, \bs{\ep}_i, r)) =\dij 2 H_i.
\end{align*}
	For the relations in (QS2), as
\begin{align*}
	[H_i, E_j] &= H_i E_j - E_j H_i \\
		&=  \sum_{r \ge 0 } ( \sum_{\lambda \in \Lambda(n,r)} \lambda_i 1_\lambda ) \sum_{r \ge 0 } e_{j,r}
		   - \sum_{r \ge 0 } e_{j,r}  \sum_{r \ge 0 } ( \sum_{\lambda \in \Lambda(n,r)} \lambda_i 1_\lambda ) \\
		&= \sum_{r \ge 0 } \sum_{\lambda \in \Lambda(n,r)}
			\lambda_i 1_\lambda e_{j,r}
			-
			\sum_{r \ge 0 } \sum_{\lambda \in \Lambda(n,r)}
			e_{j,r} \lambda_i 1_\lambda
\end{align*}
(applying the relation (QR2) in Proposition \ref{propRel1}, ignore the case {$1_{\lambda, r} = 0$})
\begin{align*}
		&= \sum_{r \ge 0 } \sum_{\lambda \in \Lambda(n,r)}
			\lambda_i e_{j,r} 1_{\lambda - \alpha_j}
			-
			\sum_{r \ge 0 } \sum_{\lambda \in \Lambda(n,r)}
			\lambda_i e_{j,r}  1_\lambda  \\
		&=
\left\{
\begin{aligned}
		&\sum_{r \ge 0 } \sum_{\lambda \in \Lambda(n,r)}
			(\lambda_i+1) e_{i,r} 1_{\lambda, r}
			-
			\sum_{r \ge 0 } \sum_{\lambda \in \Lambda(n,r)}
			\lambda_i e_{i,r}  1_\lambda,
			\mbox{ if } j=i; \\
		&\sum_{r \ge 0 } \sum_{\lambda \in \Lambda(n,r)}
			(\lambda_i-1) e_{i-1,r} 1_{\lambda, r}
			-
			\sum_{r \ge 0 } \sum_{\lambda \in \Lambda(n,r)}
			\lambda_i e_{i-1,r}  1_\lambda,
			\mbox{ if } j=i-1; \\
		&\sum_{r \ge 0 } \sum_{\lambda \in \Lambda(n,r)}
			\lambda_i e_{j,r} 1_{\lambda, r}
			-
			\sum_{r \ge 0 } \sum_{\lambda \in \Lambda(n,r)}
			\lambda_i e_{j,r}  1_\lambda,
			\mbox{ others };
\end{aligned}
\right. \\
		&= (\bs{\ep}_{i}, \alpha_j) E_j.
\end{align*}
The other relations in (QS2) can be proved similarly.

For (QS4), apply  the relations in (QR4) in Proposition \ref{propRel1}, then
\begin{align*}
	[E_{\ol{i}}, F_{\ol{j}}] &=  \sum_{r \ge 0 } (e_{\ol{i},r} f_{\ol{j},r} + f_{\ol{j},r} e_{\ol{i},r}) = \sum_{r \ge 0 } (\dij \sum_{\lambda \in \Lambda(n,r)} (\lambda_i + \lambda_{i+1}) 1_{\lambda, r}) \\
		&= \dij \sum_{r \ge 0 } ( \sum_{\lambda \in \Lambda(n,r)} \lambda_i 1_{\lambda, r})
			+ \dij \sum_{r \ge 0 } (\sum_{\lambda \in \Lambda(n,r)}  \lambda_{i+1} 1_{\lambda, r}) \\
		&= \dij(H_i + H_{i+1}).
\end{align*}
Other relations in (QS4) are similar to prove.

In summary, $H_i,H_{\bar{i}},E_j,E_{\bar{j}},F_j,F_{\bar{j}}$
satisfy the relations (QS1)-(QS6).
So $\pi$ is a  superalgebra homomorphism.

For any $M \in M({\NN}|\ZG)$, denote {$\bs{j}_{M} = (\SEE{m}_{1,1}, \SEE{m}_{2,2},\cdots, \SEE{m}_{n,n})$}.
Then there is a bijection from $M({\NN}|\ZG)$
to $\LS = \{  \AJS(A, \bs{j})  \where   A \in \MNZNS(n), \bs{j} \in \NN^n \}$,
mapping {$M$} to {$\AJS({M}^{\pm}, \bs{j}_{M})$}.
According to Proposition \ref{U(n) basis},
 {$\LS$} is a basis of $\USn$,
 and \cite[Proposition 3.3]{DW2} induces that $\Uqn$ has a PBW type basis $\{{\mathfrak{m}}^A \where A\in M({\NN}|\ZG) \}$, hence $\USn$ is isomorphic to $\Uqn$ as vector spaces.
As a consequence, we have $\pi$ is an algebra isomorphism from $\Uqn$ to $\USn$.
\end{proof}

In Theorem \ref{realuqn}, we construct a realization of $\Uqn$ as the subalgebra $\USn$ of $\Qqsn$.

Referring to \cite{DW2}, there is an epimorphism from $\Uqn$ to $\Qnr$,
then another presentation of $\Qnr$ can be deduced when we treat $\Uqn$ and $\USn$ as the same one.
The quantum case is more interesting and could be found in the ongoing papers.

\section{The integral Schur-Weyl-Sergeev duality}\label{sec_sws}

Set
{$	\fsG_r = \{ \    h_{i,r}, h_{\ol{i},r}, e_{j,r}, e_{\ol{j},r}, f_{j,r}, f_{\ol{j},r}
	\where   1\leq i\leq n, 1\leq j \leq n-1
	\} $},
 and recall the surjection
 {$\bs{\rho}_r :  \USn \longrightarrow \Qnr $},
 mapping {$\AJS(A, \bs{j}) \mapsto \AJRS(A, \bs{j}, r)$},
\begin{align*}
& \bs{\rho}_r(H_i) = h_{i, r},	\qquad \bs{\rho}_r(H_{\ol{i}}) = h_{\ol{i}, r}, \\
 & \bs{\rho}_r(E_j) =  e_{j, r},	\qquad \bs{\rho}_r(E_{\ol{j}}) = e_{\ol{j}, r}, \\
 & \bs{\rho}_r(F_j) = f_{j, r},	\qquad \bs{\rho}_r(F_{\ol{j}}) = f_{\ol{j}, r}.
\end{align*}
By Theorem \ref{U(n)Alg}, $\USn$ is a superalgebra generated by  {$\fsG$},
hence {$\Qnr$} is generated by {$\fsG_r$}.

For the quantum queer Schur superalgebra {$\QqscR$},
specialize {$R=\ZZ$} and {$q=1$},   {$\QqscR$}  degenerates to {$\QqscZ$}.
Denote  {$\QqscZ$} as {$\QnrZ$} for short.
By \cite[Theorem 5.3 ]{DW1},
{$\QnrZ$} is a free {$\ZZ$}-module with the basis
\begin{align*}
	\{ \phi_{M} \where M \in \MNZ(n,r)\}
\end{align*}
as in \eqref{basis_phi}.
So  {$\QnrZ$} is a {$\ZZ$}-subalgebra of   {$\Qnr$},
and there is a superalgebra isomorphism
\begin{align*}
	\QnrZ \otimes \CC \longrightarrow \Qnr,
\end{align*}
hence {$\QnrZ$} is an integral form of   {$\Qnr$}.

For any {$X \in \fsG$} and {$m \in \NN$},
set
\begin{align*}
X^{(m)} = \frac{X^m}{m!}, \qquad
{\binom{X}{m}} = \frac{X(X-1) \cdots (X-m+1)}{m!}, \qquad
{\binom{X}{0}} = 1.
\end{align*}

For any   {$\bs{k} \in \NN^n$},  set
\begin{align*}
{\binom{\bs{h}_r} {\bs{k}}} = \prod_{i=1}^{n}\binom{h_{i, r}}{k_i}.
\end{align*}

By  \cite[Section 4]{BK2},  the Kostant {$\ZZ$}-form of $\Uqn$, denoted by $\UqnZ$, is the {$\ZZ$}-subalgebra generated by
\begin{align*}
\fsG_{\ZZ} = \left\{ \sum_{r \ge 0} \binom{\bs{h}_{r}}{\alpha_r},   h_{\ol{i}},  e^{(k)}_{j},   e_{\ol{j}},  f^{(k)}_{j},   f_{\ol{j}}
\where i\in[1,n],j\in[1,n-1], k\in \NN,  \alpha_r \in \CMN(n, r)  \right\}.
\end{align*}
The image of {$\fsG_{\ZZ}$} under the map {$\bs{\rho}_r$} is denoted as
\begin{align*}
\fsG_{r, \ZZ} = \left\{ \binom{\bs{h}_{r}}{\alpha},   h_{\ol{i}, r},
		e^{(k)}_{j, r},   e_{\ol{j}, r},  f^{(k)}_{j, r},   f_{\ol{j}, r}
	\where  i\in[1,n],j\in[1,n-1], k\in \NN, \alpha \in \CMN(n, r) \right\}.
\end{align*}


From the equations in  \eqref{devpow}, we have
\begin{align*}
	  e^{(k)}_{h, r}= \ABJR( k E_{h, h+1}, O,\bs{ 0 }, r), \qquad
	  f^{(k)}_{h, r} = \ABJR( k E_{h+1, h}, O,\bs{ 0 }, r).
\end{align*}
For any  {$A  \in \MNZ(n, r)$},
recall that  {$A^{\pm} = ( {(\SE{A})}^{\pm} |  \SO{A}) \in \MNZNS(n)$}.
  The following corollary is derived from acting $\rho_r$ on the both sides of the triangular multiplication formula in Theorem \ref{triang2}.
\begin{cor}\label{cor_triang2}
	Suppose {$A=(\SUP{A})  \in \MNZ(n,r)$}, then
\begin{align*}
& \prod_{(i,j) \in \fsP^- }^{<}  \!\!
		{( f_{\ol{i-1}, r} )} ^{\SOE{a}_{i,j}} f^{(\SEE{a}_{i,j} )}_{i-1, r}
	\prod_{h=i-2}^{j} f^{({a}_{i,j} )}_{h, r}    \!\!
			\prod_{(i,i) \in \fsP^0 }^{<}  \!\!
			{( h_{\ol{i}, r} )}^{\SOE{a}_{i,i} }
	   \prod_{(i,j) \in \fsP^+ }^{<}  \!\!
			{( e_{\ol{i}, r} )} ^{\SOE{a}_{i,j}} e^{(\SEE{a}_{i,j} )}_{i, r}
			  \prod_{h=i+ 1}^{j-1}   \!\!  e^{({a}_{i,j} )}_{h, r}    \\
	&= \pm\AJRS(A^{\pm}, \bs{0}, r)
	+
	\sum_{ \substack{
	 				B \in \MNZNS(n), \\
	 				B \prec A^{\pm}, \\
	 				\bs{j}  \in {\NN}^n
		} } g_{B, \bs{j} }\AJRS(B, \bs{j}, r),
\end{align*}
	where all {$g_{B, \bs{j}}  \in \QQ$}.
\end{cor}

To calculate {$ {\binom{\bs{h}_r} {\bs{k}}} \AJRS(A, \bs{j},  r)$},
we need the following lemma.
\begin{lem}\label{binorm_mul}
For any  {$\bs{k} \in \NN^n$} with {$\snorm{\bs{k}} \le r$}, we have
\begin{align*}
{\binom{\bs{h}_r} {\bs{k}}}
= {\binom{\bs{h}_r} {\bs{k}}}   \AJRS(\bs{O}, \bs{0},  r)
=  \AJRS(\bs{O}, \bs{0},  r)  {\binom{\bs{h}_r} {\bs{k}}}
&=
	 \sum_{\substack{ \lambda \in \CMN(n,r) \\ {\lambda}  \ge \bs{k} } }
	\binom{\lambda}{\bs{k}}
	\phi_{ (\lambda | O )},
\end{align*}
where {${\lambda}  \ge  \bs{k}$} means {${\lambda}_i  \ge  {k}_i$} for all {$i$}.

\end{lem}
\begin{proof}
According to Proposition \ref{mulformajrb1}(i), we have
$${\binom{\bs{h}_r} {\bs{k}}}
= {\binom{\bs{h}_r} {\bs{k}}}   \AJRS(\bs{O}, \bs{0},  r)
=  \AJRS(\bs{O}, \bs{0},  r)  {\binom{\bs{h}_r} {\bs{k}}}.$$

Now we begin to prove the last equation.
For {$a \in \NN $}, direct calculation shows
\begin{align*}
 \AJRS(\bs{O}, \bs{0},  r) ( h_{i,r} - {a} )
&=  \AJRS(\bs{O}, \bs{0},  r)  \AJRS(\bs{O}, \bs{\ep}_i,  r)    - {a} \cdot  \AJRS(\bs{O}, \bs{0},  r)  \\
&=
	\sum_{\substack{ \lambda \in \CMN(n,r) } }
	\lambda ^{\bs{0}}
	\phi_{(\lambda | O)}
	\sum_{\substack{ \mu \in \CMN(n, r) } }
	\mu ^{\bs{\ep}_i}
	\phi_{(\mu | O )}
	- a \cdot \sum_{\substack{ \lambda \in \CMN(n, r) } }
	\lambda ^{\bs{0}}
	\phi_{( \lambda | O )} \\
& =\sum_{\substack{ \\ \lambda \in  \CMN(n, r) } }
		( {\lambda}_i - a)
	\phi_{ (\lambda | O )} .
\end{align*}
Notice that
\begin{align*}
{\lambda}_{i} ({\lambda}_{i} - 1)  \cdots ({\lambda}_{i} - a + 1)  = 0
\mbox{ when }
0 \le \lambda_i \le a - 1,
\end{align*}
it follows
\begin{align*}
 \AJRS(\bs{O}, \bs{0},  r) \binom{h_{i, r}}  {a}
&=
	\AJRS(\bs{O}, \bs{0},  r)
	\frac{ h_{i, r} ( h_{i, r}  - 1)  \cdots ( h_{i, r}- a + 1)}{ {a}! } \\
&=
	\frac{1}{ {a}! } \cdot
	\sum_{\substack{ \\ \lambda \in  \CMN(n, r) } }
		( {\lambda}_i - a + 1) \phi_{ (\lambda | O )}  	
	h_{i, r} ( h_{i, r}  - 1)  \cdots ( h_{i, r} - a + 2) \\
&=  \sum_{\substack{ \lambda \in \CMN(n,r) } }
	\frac{{\lambda}_{i} ({\lambda}_{i} - 1)  \cdots ({\lambda}_{i} - a + 1)}{ {a}! }
	\cdot \phi_{ (\lambda | O )} \\
&=  \sum_{\substack{ \lambda \in \CMN(n,r) \\ {\lambda}_i \ge a } }
	\frac{{\lambda}_{i} ({\lambda}_{i} - 1)  \cdots ({\lambda}_{i} - a + 1)}{ {a}! }
	\cdot \phi_{ (\lambda | O )},
\end{align*}
hence we have
\begin{align*}
 \AJRS(\bs{O}, \bs{0},  r)  {\binom{\bs{h}_r} {\bs{k}}}
&=
	\prod_{i=1}^{n}  \AJRS(\bs{O}, \bs{0},  r)  { \binom{h_{i, r}} {k_{i}} }   \\
&=
	\prod_{i=1}^{n-1}  { \binom{h_{i, r}} {k_{i}} }
	\cdot
	  \sum_{\substack{ \lambda \in \CMN(n,r ) \\ {\lambda}_n \ge k_n } }
	\frac{{\lambda}_{n} ({\lambda}_{n} - 1)  \cdots ({\lambda}_{n} - k_n + 1)}{ {k_n}! }
	\cdot \phi_{ (\lambda | O )} \\
&=
	 \sum_{\substack{ \lambda \in \CMN(n,r ) \\ {\lambda}_j \ge k_j \mbox{ for all } j } }
	\prod_{i=1}^{n}
	\frac{{\lambda}_{i} ({\lambda}_{i} - 1)  \cdots ({\lambda}_{i} - k_i + 1)}{ {k_i}! }
	\cdot \phi_{ (\lambda | O )} \\
&=
	 \sum_{\substack{ \lambda \in \CMN(n,r ) \\ {\lambda}  \ge \bs{k} } }
	\prod_{i=1}^{n}
	\frac{{\lambda}_{i} ({\lambda}_{i} - 1)  \cdots ({\lambda}_{i} - k_i + 1)}{ {k_i}! }
	\cdot \phi_{ (\lambda | O )} \\
&=
	 \sum_{\substack{ \lambda \in \CMN(n,r ) \\ {\lambda}  \ge \bs{k} } }
	\prod_{i=1}^{n}
	\binom{\lambda_i}{k_i}
	\phi_{ (\lambda | O )} =
	 \sum_{\substack{ \lambda \in \CMN(n,r) \\ {\lambda}  \ge \bs{k} } }
	\binom{\lambda}{\bs{k}}
	\phi_{ (\lambda | O )} .
\end{align*}
\end{proof}

\begin{cor}\label{cor_binorm_mul}
For any  {$ \alpha \in \CMN(n, r)$}, we have
\begin{align*}
 {\binom{\bs{h}_r} {\alpha}}   =	\phi_{ (\alpha | O )}.
\end{align*}
\end{cor}
\begin{proof}
For {$ \lambda \in \CMN(n,r) $}, {$ {\lambda}  \ge \alpha $} implies
{$ {\lambda} = \alpha $}.
Lemma \ref{binorm_mul} causes
\begin{align*}
{\binom{\bs{h}_r} { \alpha }}
=
	 \sum_{\substack{ \lambda \in \CMN(n,r-\snorm{M}) \\ {\lambda}  \ge \alpha } }
	\prod_{i=1}^{n}
	\binom{\lambda_i}{{\alpha}_i}
	\phi_{ (\lambda | O )}
=
	\prod_{i=1}^{n}
	\binom{\alpha_i}{{\alpha}_i}
	\phi_{ (\alpha | O )}
=
	\phi_{ (\alpha | O )}.
\end{align*}
\end{proof}

Applying  Theorem \ref{formD0} and Corollarly \ref{cor_binorm_mul}, it yields
the following assertions.
\begin{cor}\label{cor_mu3}
For any  {$ A \in \MNZNS(n)$},
 given {$\mu \in \CMN(n, r)$}, we have
\begin{align*}
{\rm(i)}  \qquad & {\binom{\bs{h}_r} {\mu}} \AJRS(A, \bs{0},  r)
 =  \phi_{B},  \mbox{ where } B^{\pm} = A \mbox{ and } \ro(B) = \mu ;\\
{\rm(ii)}  \qquad &  \AJRS(A, \bs{0},  r)  {\binom{\bs{h}_r} {\mu}}
 =  \phi_{C},  \mbox{ where } C^{\pm} = A \mbox{ and } \co(C) = \mu .
\end{align*}

\end{cor}

For any partition {$\mu \in \CMN(n, k)$}, denote
{$\mu ! = \prod_{i=1}^{k} (\mu_i !)$}, set  {$0!=1$}.

For any $k\in \mathbb N$, and $A \in \MNZ(n, r)$, set
\begin{align*}
\Gamma(A,k,h)
	=	\{(\nu^0,\nu^1) &  \in \Lambda(n,k_1)
		\times \Lambda(n,k_2)
		\where
		k_1 + k_2 = k,  \\
		& \qquad
		\SEE{a}_{h, i} \ge \nu^0_i,
		\SOE{a}_{h, i} \ge \nu^1_i ,
		\mbox{ for all } i \in [1,n]
		\}.
\end{align*}

\begin{prop}\label{ULk-A}
Assume $A \in \MNZ(n, r)$, $k\in \mathbb{N}$. \\
{\rm(i)} If  $k\leq \ro(A)_{h+1}$,  we have
\begin{align*}
& e_{h, r}^{k} \phi_A
	= k! \sum_{(\nu^0,\nu^1)\in \Gamma(A,k,h+1)}
	\prod_{\substack{p=1 \\ \nu^0_p >0}}^n
	{\binom{\SEE{a}_{h,p}+\nu^0_p} {\nu^0_p}}
	\phi_{A^+_{\nu}}
\end{align*}
where $A^+_{\nu} = (\SE{A}+\sum_{p=1}^n\nu^0_p(E_{h,p}-E_{h+1,p})|\SO{A}+\sum_{p=1}^n\nu^1_p(E_{h,p}-E_{h+1,p}))\in \MNZ(n, r)$.\\
{\rm(ii)} If  $k\leq \ro(A)_h$, we have
\begin{align*}
&  f_{h, r}^{k}  \phi_A
	= k! \sum_{(\nu^0,\nu^1)\in \Gamma(A,k,h)}
	\prod_{\substack{p=1 \\ \nu^0_p >0}}^n
	{\binom{\SEE{a}_{h+1,p}+\nu^0_p} {\nu^0_p}}
	\phi_{A^-_{\nu}}
\end{align*}
where $A^-_{\nu}
		= (\SE{A} + \sum_{p=1}^n\nu^0_p( E_{h+1,p}-E_{h,p} ) |
			 \SO{A}+  \sum_{p=1}^n\nu^1_p( E_{h+1,p}- E_{h,p}))
\in \MNZ(n, r)$.
\end{prop}
\begin{proof}
Let $\lambda=\ro(A)$.
Since
\begin{align*}
&e_{h, r} \phi_A
=  \sum_{\substack{\mu \in \CMN(n, r-1)} }
			\mu ^{\bs{j}} \phi_{ (E_{h, h+1}+ \lambda | O) }
	\cdot \phi_A
=  \phi_{ (\lambda+E_{h, h+1} - E_{h+1, h+1}| O) }
	\cdot \phi_A, \\
&f_{h, r} \phi_A
=  \sum_{\substack{\mu \in \CMN(n, r-1)} }
			\mu ^{\bs{j}} \phi_{ (E_{h+1, h}+ \lambda | O) }
	\cdot \phi_A
=  \phi_{ (\lambda+E_{h+1, h} - E_{h, h}| O) }
	\cdot \phi_A.
\end{align*}
Applying  Theorem \ref{upper0} and  \ref{L0A}  for $k$ times respectively,
one obtains
\begin{align*}
e_{h, r}^{k} \phi_A
&= \sum_{(\nu^0,\nu^1)\in \Gamma(A,k,h+1)}
		\frac{k!}{\nu^0 !}
		\prod_{\substack{p=1 \\ \nu^0_p >0}}^n  { (\SEE{a}_{h,p}+\nu^0_p) }
	\phi_{A^+_{\nu}} \\
&= k! \sum_{(\nu^0,\nu^1)\in \Gamma(A,k,h+1)}
	\prod_{\substack{p=1 \\ \nu^0_p >0}}^n
	{\binom{\SEE{a}_{h,p}+\nu^0_p} {\nu^0_p}}
	\phi_{A^+_{\nu}},\\
f_{h, r}^{k}  \phi_A
&= \sum_{(\nu^0,\nu^1)\in \Gamma(A,k,h)}
	\frac{ k! }{ \nu^0 !}
	\prod_{\substack{p=1 \\ \nu^0_p >0}}^n
	{ (\SEE{a}_{h+1,p}+\nu^0_p )}
	\phi_{A^-_{\nu}} \\
&= k! \sum_{(\nu^0,\nu^1)\in \Gamma(A,k,h)}
	\prod_{\substack{p=1 \\ \nu^0_p >0}}^n
	{\binom{\SEE{a}_{h+1,p}+\nu^0_p} {\nu^0_p}}
	\phi_{A^-_{\nu}} .
\end{align*}
\end{proof}

\begin{lem}\label{inte_coeff}
For any {$A \in \MNZ(n, r)$}, we have
\begin{align*}
e_{h, r}^{(k)} \phi_A
= \sum_{\substack{
	 		M \in \MNZ(n, r)
	} } g'_{A, M} \phi_{M}, \qquad
f_{h, r}^{(k)} \phi_A
= \sum_{\substack{
	 		M \in \MNZ(n, r)
	} }g''_{A, M} \phi_{M},
\end{align*}
where  {$g'_{A, M}$},  {$g''_{A, M} \in \ZZ$}.
\end{lem}
\begin{proof}
Referring to the equations in \eqref{devpow}, and applying Proposition \ref{ULk-A},
the result is then proved.
\end{proof}

%
%
%
\begin{thm}\label{iphi_gen}
	Suppose {$A=(\SUP{A})  \in \MNZ(n,r)$}, {$\mu = \co(A)$}, then
\small
\begin{align*}
	&
	\prod_{(i,j) \in \fsP^- }^{<} \!
		{( f_{\ol{i-1}, r} )}^{\SOE{a}_{i,j}}  \! f^{(\SEE{a}_{i,j} )}_{i-1, r}  \!
		 \prod_{h=i-2}^{j}  \! f^{({a}_{i,j} )}_{h, r}    \!
	\prod_{(i,i) \in \fsP^0 }^{<}		\!
			{( h_{\ol{i}, r} )}^{\SOE{a}_{i,i} }   \!
	\prod_{(i,j) \in \fsP^+ }^{<}  \!
			{( e_{\ol{i}, r} )} ^{\SOE{a}_{i,j}} e^{(\SEE{a}_{i,j} )}_{i, r}  \!
			  \prod_{h=i+ 1}^{j-1}   \! e^{({a}_{i,j} )}_{h, r}    \!
	\binom{  \bs{h}_{r} }{\mu}
	\\
	&= \pm\phi_{A}
	+
		\sum_{\substack{
	 				B \in \MNZ(n, r) \\
	 				B \prec  A
		} } 	g_{B}  \phi_{ B } .
\end{align*}
\normalsize
	where all {$g_{B}  \in \ZZ$}.
\end{thm}
\begin{proof}
From Corollary \ref{cor_triang2} and Lemma \ref{binorm_mul},
\small
\begin{align*}
&
	\prod_{(i,j) \in \fsP^- }^{<}   \!
		{( f_{\ol{i-1}, r} )}^{\SOE{a}_{i,j}}  \! f^{(\SEE{a}_{i,j} )}_{i-1, r}   \!
		 \prod_{h=i-2}^{j}  \! f^{({a}_{i,j} )}_{h, r}    \!
	  \prod_{(i,i) \in \fsP^0 }^{<}   \!
			{( h_{\ol{i}, r} )}^{\SOE{a}_{i,i} }   \!
	  \prod_{(i,j) \in \fsP^+ }^{<}   \!
			{( e_{\ol{i}, r} )}^{\SOE{a}_{i,j}}  \! e^{(\SEE{a}_{i,j} )}_{i, r}   \!
			  \prod_{h=i+ 1}^{j-1}   \! e^{({a}_{i,j} )}_{h, r}      \!
	 \binom{  \bs{h}_{r} }{\mu}	
		\nonumber \\
&=		
		 \pm  \AJS(A^{\pm}, \bs{0}) \cdot \phi_{ (\mu | O )}
	+
	\sum_{ \substack{
	 				B \in \MNZNS(n)  \\
	 				B \prec A^{\pm},
					\bs{j}  \in {\NN}^n
		} } g'_{B, \bs{j} } \AJS(B, \bs{j}) \cdot  \phi_{ (\mu | O )}  \\
&=		 \pm \phi_{ A }
	+
	\sum_{ \substack{
	 				B \in \MNZ(n, r)  \\
	 				B \prec A
		} } g_{B} \phi_{ B} .
\end{align*}
\normalsize
The coefficients {$g'_{B, \bs{j} } \in \QQ$} cause  {$g_{B} \in \QQ$}.
Lemma \ref{inte_coeff} ensures each {$g_{B} \in \ZZ$}.
%
%
%
\end{proof}
With Theorem \ref{iphi_gen}, the following two corollaries are natural consequences.
\begin{cor}
The superalgebra {$\QnrZ$} is a {$\ZZ$}-subalgebra of {$\Qnr$} generated by {$\fsG_{r, \ZZ} $},
where
\begin{align*}
\fsG_{r, \ZZ} = \left\{ \binom{h_{i, r}}{\alpha},   h_{\ol{i}, r},
		e^{(k)}_{j, r},   e_{\ol{j}, r},  f^{(k)}_{j, r},   f_{\ol{j}, r} \where
	 i\in[1,n], j\in[1,n-1], k\in \NN, \alpha \in \CMN(n, r)  \right\}.
\end{align*}
\end{cor}
\begin{proof}
Theorem \ref{iphi_gen} shows {$\{ \phi_{M} \where M \in \MNZ(n,r)\} $}
can be generated from {$\fsG_{r, \ZZ}$}.
The formulas in Proposition \ref{mulformajrb1}, \ref{mulformajrb2} and \ref{mulformajrb3}
show the products of  elements in  {$\fsG_{r, \ZZ}$} are still  in {$\QnrZ$}.
So   {$\QnrZ$} is the {$\ZZ$}-algebra generated by {$\fsG_{r, \ZZ}$}.
\end{proof}

\begin{cor}
Let {$\rho_r$} be the restriction of {$\bs{\rho}_r$} to {$\UqnZ$}, then
we have {$\rho_r(\UqnZ) = \QnrZ$}.
More precisely, {$\rho_r$} is a  {$\ZZ$}-algebra surjection {$\UqnZ \twoheadrightarrow \QnrZ$},
mapping
\begin{align*}
& \sum_{r \ge 0} \binom{\bs{h}_{r}}{\alpha_r} \mapsto  \binom{\bs{h}_{r}}{\alpha_r},  \qquad
& h_{\ol{i}} \mapsto  h_{\ol{i}, r}, &\  \\
&e^{(k)}_{j} \mapsto e^{(k)}_{j, r},  \qquad
 & e_{\ol{j}} \mapsto   e_{\ol{j}, r}, &\  \\
&  f^{(k)}_{j} \mapsto   f^{(k)}_{j, r}, \qquad
 & f_{\ol{j}}   \mapsto   f_{\ol{j}, r} .& \
\end{align*}
\end{cor}

Let {$\SerZ$} be the {$\ZZ$}-subalgebra of {$\SerA$} generated by {$\{ s_j, c_i \where 1 \le i \le n, 1 \le j \le n-1\}$}.

Let {$V = \CC^{n|n}$},  {$ \{ {\omega}_1,  \cdots, {\omega}_n, {\omega}_{\ol{1}}, \cdots, {\omega}_{\ol{n}} \}$} be its {$\CC$}-basis.
Let {$\VZ$} be the {$\ZZ$}-subspace of {$V$} with  the basis above.
The  {$({\qn},{\SerA})$}-bimodule structure of {$V^{\otimes r}$} is showed  in \cite{JK}.
Notice the actions of {$g \in \qn$}, {$c_i, s_j$} on the basis of {$V^{\otimes r}$},
which means {$\VZ^{\otimes r}$} is a  {$({\UqnZ},{\SerZ})$}-bimodule.
Analogous to the proof of  \cite[Section 1, Theorem 1.2 ]{JK} and \cite[Theorem 3,4]{Ser},
one proves  {${\End_{\UqnZ}(\VZ^{\otimes r})}^{op}   = \SerZ$} when {$n \ge r$}.

By \cite[Corollary 6.3]{DW1},  for the {$V_{R}$} and the commutative ring {$R$},
\begin{align*}
\End_{\HCR}(V_{R}^{\otimes r}) = \QqscR.
\end{align*}
Set {$R = \ZZ$} and {$q = 1$}, we have
\begin{align*}
& \End_{\SerZ}(\VZ^{\otimes r})   = \QnrZ.
\end{align*}

Hence we get the main result of this paper:
\begin{thm}
(The integral Schur-Weyl-Sergeev duality)
The superspace
{$\VZ^{\otimes r}$} has a {$({\UqnZ},{\SerZ})$}-bimodule structure, and\\
{\rm (i)} \ \    {$\End_{\SerZ}(\VZ^{\otimes r})   = \rho_r(\UqnZ)$}; \\
{\rm (ii)} \   {${\End_{\UqnZ}(\VZ^{\otimes r})}^{op}   = \SerZ$} when {$n \ge r$}.
\end{thm}

\vspace{1cm}
{\it Acknowledgement}: 
The first author gratefully acknowledges the support from the Natural Science Foundation of China under Grant No.11671234 and No.12071129. The second and third authors gratefully acknowledge support from the Natural Science Foundation of China under Grant No.11871404 and No.11971398.

\vspace{1cm}

\end{document}